\documentclass[11pt]{amsart}
    
\usepackage{amsmath,amsfonts,amsthm,amscd,amssymb,graphicx}
\usepackage{mathabx}

\usepackage{subfigure}
\usepackage{xcolor}

\numberwithin{equation}{section}

\usepackage{hyperref}

\usepackage{fullpage}

\newtheorem{theorem}{Theorem}[section]

\newtheorem{lemma}[theorem]{Lemma}
\newtheorem{lem}[theorem]{Lemma}
\newtheorem{proposition}[theorem]{Proposition}
\newtheorem{prop}[theorem]{Proposition}

\newtheorem{cor}[theorem]{Corollary}

\newtheorem{remark}[theorem]{Remark}

\newtheorem{rem}[theorem]{Remark}
\newtheorem{defi}[theorem]{Definition}

\newcommand{\RR}{\mathbb{R}}
\newcommand{\cO}{\mathcal{O}}

\newcommand{\CC}{\mathbb{C}}
\newcommand{\cC}{\mathcal{C}}

\def\cP{\mathcal{P}}
\def\cL{\mathcal{L}}
\def\hv{\check{v}}

\def\eps{\epsilon}

\def\FH{\widehat{H}}

\def\Fg{\widehat{g}}

\def\Fphi{\widehat{\phi}}
\def\FE{\widehat{E}}
\def\Fh{\widehat{h}}

\def\FS{\widehat{S}}

\def\FG{\widehat{G}}

\def\FA{\widehat{A}}

\def\Tg{\widetilde{g}}
\def\Trho{\widetilde{\rho}}

\def\TG{\widetilde{G}}

\def\TH{\widetilde{H}}
\def\Tphi{\widetilde{\phi}}

\def\Fg{\widehat{g}}
\def\FA{\widehat{A}}
\def\Fphi{\widehat{\phi}}
\def\FE{\widehat{E}}

\def\hv{\hat{v}}

\def\Tg{\widetilde{g}}

\def\Trho{\widetilde{\rho}}
\def\Tj{\widetilde{{\bf j}}}
\def\bj{{\bf j}}

\def\Tphi{\widetilde{\phi}}
\def\TA{\widetilde{A}}

\def\mP{\mathbb{P}}

\makeindex

\begin{document}

\title{Linear Landau damping for the Vlasov-Maxwell system in $\RR^3$} 

\author{Daniel Han-Kwan}\address{CNRS, Laboratoire de Math\'ematiques Jean Leray (UMR 6629), Nantes Universit\'e, 44322 Nantes Cedex 03, France.} \email{daniel.han-kwan@univ-nantes.fr}
\author{Toan T. Nguyen}\address{Penn State University, Department of Mathematics, State College, PA 16802, USA.} \email{nguyen@math.psu.edu}
\author{Fr\'ed\'eric Rousset}\address{Universit\'e Paris-Saclay, CNRS, Laboratoire de Math\'ematiques d'Orsay,  91405 Orsay Cedex, France.}\email{frederic.rousset@universite-paris-saclay.fr}

\maketitle

\begin{abstract}
In this work, we consider the  relativistic  Vlasov-Maxwell system, linearized around a spatially homogeneous equilibrium, set in the whole space $\RR^3 \times \RR^3$. The equilibrium is assumed to belong to a class of radial, smooth, rapidly decaying functions. Under appropriate conditions on the initial data, 
we prove algebraic decay (of dispersive nature) for the electromagnetic field. 
For the electric scalar potential, the leading behavior is driven by a dispersive wave packet with non-degenerate phase and compactly supported amplitude, 
 while for the magnetic vector potential, it is driven by a wave packet  whose phase behaves globally  like the one of   Klein-Gordon and the amplitude has unbounded support.
 \end{abstract}

\tableofcontents



\section{Introduction}


Consider the relativistic Vlasov-Maxwell system, which models the dynamics of charged particles in a plasma, and which reads 
\begin{equation}\label{VM} 
\left \{ \begin{aligned}
\partial_t f + \hat{v} \cdot \nabla_x f + (E + \hv \times B)\cdot \nabla_v f & =0,
\\
\partial_t B + \nabla_x \times E = 0, \qquad \nabla_x \cdot E  &= \rho[f] - n_0,
\\
-\partial_t E + \nabla_x \times B = \bj[f], \qquad \nabla_x \cdot B & =0.
\end{aligned}
\right.
\end{equation}
In this system of equations, the scalar function $f(t, x, v) \ge 0$ represents the density distribution of electrons, with the particle position $x\in \RR^3$, the particle momentum $v\in \RR^3$, the  particle energy\footnote{We will systematically use the notation $\langle \cdot \rangle = \sqrt{1 + |\cdot|^2}$ in this paper.} $\langle v\rangle = \sqrt{1+|v|^2}$, and the particle velocity $\hv = v/\langle v\rangle$. The electric and magnetic fields $E(t,x), B(t,x)$ are three-dimensional vector fields which are generated by the particles themselves, solving the Maxwell equations with sources given by the charge and current densities 
\begin{equation}\label{def-rhoj} 
\rho[f] = \int_{\RR^3} f(t,x,v) \; dv , \qquad \bj[f] = \int_{\RR^3} \hv f(t,x,v) \; dv,
\end{equation}
respectively. 
In the present setting, we have considered a fixed uniform background of ions with constant density $n_0\not =0$. For simplicity all other physical constants have been set equal to 1; however, our results do not depend on this normalization.

The Vlasov-Maxwell system \eqref{VM} can be solved at least locally in time considering smooth initial data $(f|_{t=0}, E|_{t=0}, B|_{t=0} )$ satisfying the compatibility conditions
$$
\nabla_x \cdot E|_{t=0} = \int_{\RR^3} f|_{t=0} \, dv -n_0 , \quad \nabla_x \cdot B|_{t=0} = 0.
$$
In \cite{DPL}, Diperna and Lions built global weak solutions to the system.
For what concerns classical solutions, the global-in-time 3D Cauchy problem is still up to now a well-known open problem; see e.g., \cite{GlasseyStrauss, BGP, Klainerman, Pallard, LukStrain} for the classical conditional results and recent advances, except for the case of lower dimensions \cite{Glassey-rVM2.5,Glassey-rVM2}, with cylindrical symmetry \cite{WangVM}, and for small initial data \cite{GlasseyStrauss1, GlasseyStrauss2, Big1, Wang, WY, Big2}; in the latter, the large time behavior of solutions was also established. 

\subsection{Main result}

The aim of this article is to investigate the large time behavior of solutions to the linearized Vlasov-Maxwell system near non-negative spatially homogenous and radial steady states $\mu(v)$ with nonzero mass $n_0$. Specifically, we consider the linearized system
\begin{equation}\label{lin-VM}
\left \{ \begin{aligned}
 \partial_t f + \hat v \cdot \nabla_x f + ( E +  \hat v \times B ) \cdot \nabla_v \mu &= 0,
\\
 \partial_t B + \nabla_x \times E = 0, \qquad \nabla_x \cdot E  &= \rho[f],
\\
- \partial_t E + \nabla_x \times B =  \bj[f], \qquad \nabla_x \cdot B & =0, 
\end{aligned}
\right.\end{equation}
in the whole space $\RR^3 \times \RR^3$. 
The charge and current densities $\rho[f]$ and $\bj[f]$ are defined as in \eqref{def-rhoj}. 
We consider initial data 
\begin{equation}\label{data}
f_{\vert_{t=0}} = f^0(x,v), \qquad E_{\vert_{t=0}} = E^0(x), \qquad B_{\vert_{t=0}} = B^0(x),
\end{equation}
which are systematically asked to satisfy the compatibility conditions
\begin{equation}\label{data-comp}
\nabla_x \cdot E^0 = \int_{\RR^3} f^0(x,v) \, dv  , \qquad \nabla_x \cdot B^0 = 0, \qquad  \iint f^0(x,v) \, dxdv =0.
\end{equation}
In our previous work \cite{HKNR3} (see also \cite{BMM-lin}), we studied the analogous problem for the linearized Vlasov-Poisson system
\begin{equation}\label{lin-VP}
\left \{ \begin{aligned}
&\partial_t f +  v \cdot \nabla_x f + E \cdot \nabla_v \mu = 0,
\\
&\nabla_x \cdot E  = \rho[f], \quad \nabla_x \times  E =0,
\end{aligned}
\right.\end{equation}
which can be recovered from \eqref{lin-VM} in the non-relativistic limit, that is when the renormalized speed of light $c$ is kept in the equations and assumed to go to $+\infty$. For smooth, radial, and decaying equilibria $\mu$, we proved in \cite{HKNR3}  that the density $\rho[f]$ enjoys a decomposition of the form $\rho[f]= \rho^R + \rho^S$, where $\rho^R$ is a regular part which roughly decays in $1/{t^3}$ like free transport (as in the near vacuum case $\mu=0$), 
while the singular part $\rho^S$ can be written as 
$$\rho^S(t,x) = \sum_\pm \int_0^t G_\pm (t-s,x) \star_x S(s,x) \, ds, \quad S(t,x)=\int_{\mathbb{R}^d} f^0(x-vt,v) \, dv,$$
where $G_\pm$ is the propagator of a linear dispersive PDE of Schr\"odinger type, namely $G_\pm (t,x) = \int_{\mathbb{R}^d} e^{tZ_\pm(\xi) + i x \cdot \xi} A_\pm(\xi)$, where $A_\pm$ is supported in a ball and the phase $Z_\pm$ is such that $\Re Z_\pm\leq 0$ and is  flat at $\xi =0$, the order of vanishing being linked to the rate of decay of the equilibrium $\mu$. However $\Im Z_\pm$ is non-degenerate, which gives rise to dispersive decay.

\bigskip

This work can be seen as a natural extension of these results to the linearized Vlasov-Maxwell system \eqref{lin-VM}.
We shall focus precisely on non-negative radial equilibria under the form
 $$  
\mu (v)= \varphi(\langle v\rangle),$$
 with $\int_{\RR^3} \mu(v) \; dv =n_0$.
 We shall ask that  $\varphi(s)= F(s^2)$, where 
the function $F$ satisfies the following assumptions:

\begin{itemize}

\item {\bf (H1)} $F$ has an holomorphic extension to   $\{\Re s >0\}$.

\item {\bf (H2)}  $F$ and its derivatives are  
rapidly decaying along rays: there exists  $N_0 \in \mathbb{N}$ large enough
such that for every $R>0$,  for every $n\in \mathbb{N}$,  there exists
  $C_{R,n}  $ such that for all  $s \geq 1$, and for all $z \in \{ \Re\, z >0\} \cap \{ |z| \leq R\}$, 
 we have 
 \begin{equation}
 \label{decay-mu-C}  |F^{(n)} (  s z ) | \le C_{R,n}  (  s \Re\, z)^{-{N_{0} \over 2 }- n}.
 \end{equation}
\end{itemize}

Note that from the second assumption \eqref{decay-mu-C}, we get by taking $z=1$ and the chain rule  that
\begin{equation}
\label{decay-mu}
 |\varphi^{(n)} (  \langle v \rangle ) | \le C \langle v  \rangle^{-N_0 -n}.
 \end{equation}

As an important example, covered by the above assumptions, we can consider  the normalized Maxwellian $\mu(v) = \frac{n_0}{(2\pi)^{3/2}}\exp(-|v|^2/2)$. For this distribution, we have $\varphi (s) = \frac{n_0}{(2\pi)^{3/2}} e^{1/2} \exp( - s^2/2)$,
and hence $F(z)=   \frac{n_0}{(2\pi)^{3/2}} e^{1/2} \exp( - z/2)$
 which clearly satisfies all the above assumptions.
 They are also matched for the equilibrium
 $\mu(v)= { c_{0} \over (1+ |v|^2)^M}$  with  $M$  sufficiently large and where $c_{0}>0$ is such that $\int \mu(v) = n_{0}$. Indeed, we have 
 that $\varphi(s)=\frac{c_{0}}{( s^2)^M}$ and 
 hence $F(z) = \frac{c_{0}}{  z^M}$.

Our main result reads as follows\footnote{The definition of the Besov spaces $B^k_{1,2}$ and $B^0_{\infty,2}$ is given in Appendix~\ref{sec:LittlewoodPaley}.}. 

\begin{theorem}\label{theo-main} Let $\mu(v)$ be a non-negative radial equilibrium of the form $\mu (v)= \varphi(\langle v\rangle)=F(\langle v\rangle^2)$, with $F$  satisfying {\bf (H1), (H2)} for $N_{0}$ sufficiently large. Let $(f,E,B)$ be the solution to the linearized Vlasov-Maxwell system \eqref{lin-VM}-\eqref{data-comp}. 
Let  $M_0\leq \lfloor (N_0- 19)/2 \rfloor$.
Suppose that the initial data satisfy 
\begin{equation}\label{data-assumptions}
\index{e@$\eps_0$:  bound on the initial conditions}
\begin{aligned}
\epsilon_0&:= \sum_{|\beta|\le 1+M_0}\|  \langle v\rangle^{6+M_0} \partial_v^\beta f^0(\cdot,v)\|_{L^{1}_x L^1_v\cap L^\infty_v} + \| \langle x\rangle \rho[f^0]\|_{L^1_x \cap L^\infty_x}  \\
 & \quad +  \|  \nabla_x \mP \bj[f^0] \|_{L^2_x \cap  B^1_{1,2}}  + \| \nabla_x  E^0 \|_{L^2_x \cap B^2_{1,2}} + \| B^0\|_{L^2_x \cap B^3_{1,2}}     \\
 &\quad+\| \nabla_x \Delta_x^{-1} \nabla_x \times B^0\|_{ B^3_{1,2}} + \| \Delta_x^{-1} \nabla_x \times B^0 \|_{L^{1}_x}~
 < +\infty.
 \end{aligned}
\end{equation}
Then, for all $t\ge 0$, we can write 
\begin{equation}\label{rep-EBosc}
\begin{aligned}
E &=E^{osc}(t,x) + E^r(t,x),
\\
B &= B^{osc}(t,x) + B^r(t,x),
\end{aligned}
\end{equation}
where $E^{osc}=E^{osc,(1)} + E^{osc,(2)}$, $B^{osc}$ are oscillatory and satisfy for all $p \in [2,\infty)$,
\begin{equation}\label{osc-decay}
\begin{aligned}
 \| E^{osc}(t)\|_{L^p_x} + \| B^{osc}(t)\|_{L^p_x} &\lesssim \epsilon_0 \langle t\rangle^{-3(\frac{1}{2}- \frac{1}{p})} , \\
 \| E^{osc,(1)}(t)\|_{L^\infty_x} &\lesssim  \langle t\rangle^{-\frac{3}{2}}\left(\epsilon_0 + \| \nabla_x \Delta_x^{-1} \rho_0\|_{L^1_x}  + \| \nabla_x \Delta_x^{-1} \nabla_x \cdot \bj[f^0]\|_{L^{1}_x} \right), \\
\| E^{osc,(2)}(t)\|_{B^0_{\infty,2}} + \| B^{osc}_\pm(t)\|_{B^0_{\infty,2}} &\lesssim \epsilon_0 \langle t\rangle^{-\frac{3}{2}},
 \end{aligned}
\end{equation}
while the regular electric and magnetic fields $E^r,B^r$ satisfy 
\begin{equation}\label{re-decay}
\| \partial_x^\alpha E^r(t)\|_{L^p_x} +\|\partial_x^\alpha B^r(t)\|_{L^p_x}  \lesssim \epsilon_0 \langle t\rangle^{-4/3 - |\alpha|/3+1/p+\delta} , \qquad p\in [1,\infty],
\end{equation}
for any $0\le |\alpha|\le M_0$ and for any small $\delta>0$. 
\end{theorem}

In the proof of Theorem \ref{theo-main}, we actually obtain more precise information on the structure of the oscillating part $(E^{osc}_\pm, B^{osc}_\pm)$ of the electromagnetic field.  We shall see that the leading dynamics is oscillatory and dispersive and  will be shown to behave like a Klein-Gordon wave of the form $e^{\pm i \sqrt{1 - \Delta_x} t}$. On the other hand, the remainder has the property that derivatives in $x$ gain extra decay. Interestingly, the extra decay rate is only of $t^{-1/3}$ per extra derivative, not $t^{-1}$ as would be expected from the dispersion for the transport dynamics. This is due to the emergence of a wave structure in the magnetic part due to the long range interaction. 
We shall further detail this point in the next section, and even more in the proof of the upcoming Proposition~\ref{prop-Green}.
Note that  we need $N_{0}$ to be sufficiently large, though we did not try to optimize this value, this is not surprising
in the relativistic case. Indeed even for the free transport, we need initial data more localized in the velocity variable in order to 
get the same dispersive inequality as in the non-relativistic case, see Lemma \ref{lem:bound-SfreeLp} below for example.

\subsection{Wave structure}

The proof of Theorem \ref{theo-main}  reveals much more structure than what is stated in the theorem, of which we would like to give an overview. First, we express the electromagnetic field as 
\begin{equation}\label{express-EB}
E = -\nabla_x \phi - \partial_t A , \qquad B = \nabla_x \times A,
\end{equation}
together with the Coulomb gauge $\nabla_x \cdot A = 0$. We shall establish, see Section \ref{sec-fields}, that the electric and magnetic potentials satisfy the decomposition 
\begin{equation}\label{rep-phiAosc}
\begin{aligned}
\phi &= \sum_\pm \phi^{osc}_\pm(t,x) + \phi^r(t,x),
\\
A &= \sum_\pm A^{osc}_\pm(t,x) + A^r(t,x),
\end{aligned}
\end{equation}
where 
$$\phi^{osc}_\pm = G_{\pm}^{osc}\star_{t,x} S^\phi(t,x), \qquad A^{osc}_\pm = H_{\pm}^{osc}\star_{t,x} S^A(t,x)$$
for some integral kernels $G_{\pm}^{osc}$ and $H^{osc}_\pm$, whose Fourier transform in $x$ is of the form $e^{\pm i\tau_*(|k|) t} a_\pm(k)$ and $e^{\pm i\nu_*(|k|) t} b_\pm(k)$, with sufficiently smooth symbols $a_\pm(k), b_\pm(k)$. The electric dispersion relation is of the form of a Schr\"odinger wave $\tau_*(|k|) \sim 1 + |k|^2$ and only present in the low frequency regime $|k|\lesssim 1$, while the magnetic dispersion relation is that of a Klein-Gordon wave $\nu_*(|k|) \sim \sqrt{1 + |k|^2}$ for all values of $k\in \RR^3$. The source terms $S^\phi(t,x)$ and $S^A(t,x)$ can be expressed explicitly in terms of initial data $f^0, E^0, B^0$. This defines the oscillatory fields $E^{osc}_\pm, B^{osc}_\pm$ via \eqref{express-EB} as stated in Theorem \ref{theo-main}.  

In view of the previous works on the Vlasov-Poisson system, see, e.g., \cite{HKNR2,HKNR3,Toan}, one may expect that the remainder $\phi^r$ to behave like the sources generated by the relativistic free transport dynamics. Precisely, we shall indeed show in Section \ref{sec-decay} that 
\begin{equation}\label{fastdecayphir} \| \partial_x^\alpha\nabla_x \phi^r(t)\|_{L^p_x} \lesssim \langle t\rangle^{-3+3/p - |\alpha|} \end{equation}
for $1\le p\le \infty$. On the other hand, for what concerns the magnetic potential $A^r$, we shall prove that the corresponding Green function $H^r(t,x)$ has its Fourier transform $\FH^r_k(t)$ satisfying an estimate of the form
$$|\FH^r_k(t)|  \lesssim |k|\langle k\rangle^{-2} \langle kt\rangle^{-N}  + |k| \langle |k|^3t\rangle^{-N)} \chi_{(\text{low frequency})},$$
where $\chi_{(\text{low frequency})}$ stands for a certain low-frequency (in time and space) cutoff function and $N$ is 
large (under the form $N_{0}/2$ minus a fixed positive number).
Namely, besides the classical decay in $\langle kt\rangle$ dictated by the free transport dynamics, there is a new scaling due to the term in $\langle k^3 t\rangle$ emerging at  low frequencies. The presence of this last term is rather unexpected, though appears sharp, see Remark \ref{rem-k3t} and Section \ref{sec-GreenB} for details. This leads to a decay of order $t^{-1/3}$ as stated in the main theorem.

\subsection{Landau damping and related decay results}\label{sec-introLandau}

The classical notion of Landau damping, as discovered by Landau himself in his seminal work \cite{Landau-paper}, refers to the damping of plasma oscillations due to the resonant interaction between particles and oscillatory waves. At the linearized level of the Vlasov-Poisson system, this is only seen in the very low frequency regime where dispersive  oscillations, often referred to as Langmuir waves, are present, see \cite{HKNR3,BMM-lin} for mathematical justifications. 
In \cite{Landau-paper}, around homogeneous equilibria $\mu(v) = \varphi(|v|)$, Landau managed to calculate the damping rate of these oscillatory modes: precisely,  the zeros $\lambda = \lambda_\pm(k)$  of the electric dispersion relation $D(\lambda_\pm(k),k) = 0$ satisfy 
\begin{equation}
\label{Landau-law}
\begin{aligned}
\Re \lambda_\pm(k) \approx {1 \over 4| k |^2} \varphi'\Bigl({1 \over | k | } \Bigr) , \qquad 
\Im \lambda_\pm(k) \approx \pm ( 1 + \frac32 e_0 |k|^2 ) ,
\end{aligned}\end{equation}
for sufficiently small $|k|$, in which $e_0 =\int |v|^2 \mu(v) \; dv >0.$ 
That is, the damping (or possibly the growth) of the electric field is explicitly computed through $e^{\lambda_\pm(k) t}$. In particular if the map $|v| \mapsto \varphi(|v|)$ is decreasing, the electric field
$\FE_k(t)$ is damped, and the speed of damping is linked to $\partial_v \mu(|k|^{-1})$. This exponential  damping rate is polynomially small for power-law equilibria and exponentially small for Gaussian equilibria. The damping rate is thus very sensitive to the decay of $\mu(v)$ at  large $|v|$, as justified in the works of Glassey and Schaeffer \cite{GS-LD1,GS-LD2}. For compactly supported equilibria or in the case relativistic velocities, this Landau damping is negligible or zero as shown in \cite{GS-LD1,GS-LD2,Toan}; see also the upcoming Theorem \ref{theo-LangmuirE}.
Physically, this leads to a transfer of energy from the electric energy to the
kinetic energy of these particles, which takes place at the resonant velocity $v \sim |k|^{-1}$.  
Let us finally highlight the work of Bernstein \cite{Bernstein} who extended some aspects of Landau's work to the linearized Vlasov-Maxwell system. It is in particular shown that in the presence of an external magnetic field, some waves may not be Landau damped.

There has been recently an important mathematical activity  concerned with Landau damping (broadly understood in the sense of asymptotic stability results for solutions to Vlasov equations near non-trivial equilibria) starting with the seminal work of Mouhot and Villani \cite{MV} who studied the Vlasov-Poisson system near Penrose stable equilibria on the torus $\mathbb{T}^d$. They proved that the phase mixing stability mechanism survives for the full non-linear equations in analytic or strong enough Gevrey regularity. This result was then revisited and sharpened in \cite{BMM-apde} and more recently in \cite{GNR1,GNR2}. In \cite{Bed2}, it is proved that  the nonlinear mechanism put forward in \cite{MV} does not hold in finite regularity.
 These results were also adapted to the \emph{relativistic} Vlasov-Poisson system in \cite{Young,YoungJDE} on the torus;
 as pointed out in these works, a new feature of the relativistic case is that at the linearized level  exponential decay can hold only under a restrictive
 condition on the size of the torus.
  Let us also mention the works \cite{Tri,Bed1,CLN} which are concerned with the regime of weak collisions, described by a Fokker-Planck operator for the former twos, and by a Landau collisional operator for the latter.

Systems including magnetic fields were less studied in the mathematical literature. We mention the works  \cite{BW,CDRW} which are concerned with the
 linearized Vlasov-Poisson on the torus,  in the presence of a  constant magnetic field, in relation with the findings of \cite{Bernstein}. 
 For the study of the linear stability of general  equilibria, we refer to \cite{Lin-Strauss}.
  In \cite{HKNR1}, we have considered the relativistic Vlasov-Maxwell on the torus, and provided long (finite) time stability estimates, for well-prepared data; long time has to be understood in terms of powers of $c$, where $c$ denotes the renormalized speed of light, in the regime $c \to +\infty$.
 
 For what concerns equations set in the whole space, the work \cite{BMM-cpam} considered the screened Vlasov-Poisson system and proved stability in finite regularity around Penrose stable equilibria in dimensions $d\geq 3$. Finite regularity (as opposed to Gevrey regularity on the torus) is possible thanks to the dispersive properties of free transport in the whole space. The strategy of \cite{BMM-cpam}  is inspired by that of \cite{BMM-apde} for the torus case.
  This was later revisited in \cite{HKNR2}, where we obtained $L^1$ and $L^\infty$ dispersive  estimates for the linearized problem in the spirit
  of the estimates obtained here in Theorem \ref{theo-main}  and  then used a  Lagrangian approach for the nonlinear problem inspired by \cite{BD}; this strategy was even sharpened in \cite{HNX1,HNX2} to reach all dimensions $d\geq 2$.
  Finally, for the Vlasov-Poisson system (without screening), as already mentioned, the linearized system was studied in \cite{HKNR3,BMM-lin}, see also \cite{Ionescu-Pausader0} while a nonlinear result is obtained in \cite{IPWW} for the special case of the Poisson equilibrium; for other types of   equilibria, we refer to \cite{Toan}.

\subsection{Organization of the paper}

We conclude this introduction with an outline of the paper, which is entirely dedicated to the proof of Theorem~\ref{theo-main}. It is divided in three main parts.

\bigskip 

\noindent {\large\bf I.} In Section~\ref{sec:FL}, we initiate the Fourier-Laplace analysis of the linearized Vlasov-Maxwell system \eqref{lin-VM}. First of all, we reformulate~\eqref{lin-VM} by using the electric scalar potential $\phi$ and the magnetic vector potential $A$ in Coulomb gauge, and by introducing a \emph{shifted} distribution function, which somehow allows us to {\bf decouple} the equations for $\phi$ and $A$. Applying the Fourier-Laplace transform, we are led to study an electric dispersion (scalar) function, denoted by $D(\lambda, k)$, and a magnetic dispersion (again, scalar) function, denoted by $M(\lambda,k)$, for $\Re \lambda>0$, $k \in \mathbb{R}^3$. In Proposition~\ref{prop-nogrowth}, we first show that both dispersion functions do not admit zeros with non-zero real part, which entails {\bf spectral stability}.
The dispersion functions $D$ and $M$ are then thoroughly studied in Sections~\ref{sec:D} and \ref{sec:M}; in particular we establish identities and extension results (that is, beyond the domain  $\{\Re \lambda>0\}$), as well as decay estimates for large $|\lambda|$ or $|k|$. Let us highlight the fact that by opposition to the Vlasov-Poisson case  considered in \cite{HKNR3,BMM-lin}, the dispersion functions cannot be extended as holomorphic functions on domains of the form  $\{\Re \lambda > - \theta |k|\}$.
In the key Sections~\ref{sec-D} and \ref{sec-M}, we finally study the electric and magnetic  dispersion relations, that correspond to the zeros of $D$ and $M$. 

For the {\bf electric dispersion relation}, we show in Theorem~\ref{theo-LangmuirE} that
\begin{itemize}

\item there exists a threshold $\kappa_0>0$ such that for all $0\leq |k| \leq \kappa_0$, there exists exactly two purely imaginary zeros of $D$, while there is no zero for $|k|>\kappa_0$.
\end{itemize}
In view of the upcoming Laplace inversion analysis, we wish to understand the behavior of the zeros of an appropriate extension of 
$D$ when $|k|$ lies in a small open interval containing $\kappa_0$. 
The main   difficulty  is due to the fact that  the function $D$ cannot be extended as an holomorphic function in $\lambda$ for $|k|$ lying in
 any   interval containing $\kappa_{0}$ in its interior. We thus   prove that 

\begin{itemize}
\item the curve of zeros can be smoothly extended for $|k|<\kappa_0+\delta$ (with $0<\delta\ll 1$) as the only zeros of a smooth 
 (non-analytic) extension of $D$;  these zeros have negative real part and are  in the domain where it is possible to define an analytic  extension of $D$  thanks to our assumptions (H1-2);
 
\item the imaginary part of the zeros has a {\bf non-degenerate Hessian}, paving the way for {\bf dispersive decay estimates}.

\end{itemize}

For the {\bf magnetic dispersion relation}, we show in Theorem~\ref{theo-LangmuirB} that for all $k \in \mathbb{R}^3$ there are exactly two purely imaginary zeros of $M$, which behave globally  like a {\bf Klein-Gordon} phase.

\bigskip

\noindent {\large\bf II.} Relying on these key results, we study in Section~\ref{sec:Green} the resolvent kernels $1/D$ and $1/M$ and their associated Green functions in Fourier space, see Propositions~\ref{prop-GreenG} and \ref{prop-Green}. In particular, for the electric Green function $\FG_k(t)$ in Fourier space, we obtain in Proposition~\ref{prop-GreenG} the decomposition
\begin{equation*}
\FG_k(t) = \delta_{t=0} + \sum_\pm \FG^{osc}_{k,\pm}(t)  +   \FG^{r}_k(t) ,
\end{equation*}
where $ \FG^{r}_k(t) $ corresponds to the regular part, and 
$
\FG^{osc}_{k,\pm}(t) = e^{\lambda_\pm^{\text{elec}}(k) t} a_\pm(k),
$
and the $\lambda_\pm^{\text{elec}}(k)$ correspond to the zeros of the electric dispersion relation, which were constructed in the previous section. We obtain an analogous decomposition for the magnetic Green function in Proposition~\ref{prop-Green}.
As already said, contrary to  \cite{HKNR3,BMM-lin}, the resolvent kernels cannot be meromorphically extended to a domain of the type $\{\Re \lambda > - \theta |k|\}$.
To bypass this issue, we rely on a sufficiently accurate approximation by rational functions with appropriate poles, for which we can use Cauchy's residue theorem, and then by proving that the errors made are acceptable. 
We mention that this idea was also recently used in \cite{Toan}.
As the proofs are rather long, in order to ease global readability, we have chosen to postpone the proofs of Propositions~\ref{prop-GreenG} and \ref{prop-Green} to subsequent appendices.
In Section~\ref{sec-fields}  we finally obtain a  decomposition formula for the electromagnetic field $(E,B)$ in terms of the initial data, in relation with the decompositions obtained in Propositions~\ref{prop-GreenG} and \ref{prop-Green}.

\bigskip

\noindent {\large\bf III.}  To conclude, Section~\ref{sec-decay} relies on all previous results to prove decay for the electromagnetic field $(E,B )$, mainly using (non-)stationary phase estimates. Theorem~\ref{theo-LangmuirE} and Proposition~\ref{prop-GreenG} (respectively Theorem~\ref{theo-LangmuirB} and Proposition~\ref{prop-Green}) lead to decay in physical space for the electric part (respectively the magnetic part), as obtained in Proposition~\ref{prop-Greenphysical}; in particular the decay associated to the oscillatory part corresponds  to a Schr\"odinger (respectively Klein-Gordon) decay for the electric scalar potential (respectively magnetic vector potential) part. Decay also depends on the relativistic free transport dispersion (as recalled in Lemma~\ref{lem:bound-SfreeLp}). The resulting decay statements for the electromagnetic field are finally gathered in Propositions~\ref{lem:phidecay} and~\ref{lem:Adecay}.

\bigskip

\noindent The paper ends with five appendices.  Appendices~\ref{sec:proof1} and~\ref{sec:proof2} are dedicated to the proofs of Proposition~\ref{prop-GreenG} and~\ref{prop-Green}, respectively.
Appendix~\ref{sec:LittlewoodPaley} provides reminders of classical definitions and results related to the Littlewood-Paley decomposition. 
Appendix~\ref{sec-potential} recalls some well-known potential estimates.
Appendix~\ref{sec:OI} eventually proves some variants of classical non-stationary and stationary phase estimates, in view of the study of some oscillating integrals.

\bigskip

\noindent This work introduces a rather large amount of notations; to ease readability, we provide an Index at the end of the paper.

\section{Fourier-Laplace stability analysis}
\label{sec:FL}

In this section, we shall introduce the Fourier-Laplace approach to study the linearized Vlasov-Maxwell system \eqref{lin-VM}. We first reformulate the problem. 

\subsection{Reformulation of the linearized Vlasov-Maxwell  system}

Recalling that  $\mu(v) = \varphi(\langle v\rangle)$, we compute $\nabla_v \mu = \hv \varphi'(\langle v\rangle)$ and so $(\hat v \times B) \cdot \nabla_v\mu = 0$. Therefore, the Vlasov equation simply becomes 
\begin{equation}\label{Vlasov} \partial_t f + \hat v \cdot \nabla_x f + E \cdot \nabla_v \mu = 0.\end{equation}
To study the Maxwell equations, it is standard to introduce the electric scalar potential $\phi$ and magnetic vector potential $A$
through 
\begin{equation}\label{def-phiA}
   \index{p@$\phi(t,x)$: electric scalar potential}
      \index{A@$A(t,x)$: magnetic vector potential}
E = -\nabla_x \phi - \partial_t A , \qquad B = \nabla_x \times A, \qquad \nabla_x \cdot A = 0,
\end{equation}
in which we have chosen to impose the Coulomb gauge $\nabla_x \cdot A =0$. The Maxwell equations then become 
\begin{equation}\label{Maxwell}
 - \Delta_x \phi = \rho[f] , \qquad ( \partial_t^2 - \Delta_x) A = \mP \bj[f]
 \end{equation}
 where $\index{p@$\mathbb{P}$: Leray projector}\mP$ denotes the classical Leray projector, that is $\mP \bj = \mathbb{I} - \nabla_x (\Delta_x)^{-1} \nabla_x \cdot \bj$. Note that $\mP A = A$, since the potential $A$ is a divergence-free vector field. That is, the linearized Vlasov-Maxwell system \eqref{lin-VM} reduces to the system \eqref{Vlasov}-\eqref{Maxwell}. 

Next, as $E$ in \eqref{Vlasov} involves $\partial_tA$, it turns out to be convenient to further introduce a shifted distribution function 
\begin{equation}\label{def-g}
     \index{g@$g(t,x,v)$: shifted distribution function}
g = f - A \cdot \nabla_v\mu,
\end{equation}
we refer to \cite{NS,HKN,HKNR1} for other recent uses of this change of the unknown function.
Note that 
\begin{equation}\label{shift-rho}
 \rho[f] = \rho[g], \qquad \bj[f] = \bj[g] - \tau_0^2 A, 
 \end{equation}
 for some positive constant $\tau_0$. Indeed, we have
$$\bj[A \cdot \nabla_{v} \mu] = \int   {\mu(v) \over (1 + |v|^2)^{1 \over 2}}\left({I_{3} - \hv \otimes \hv}\right)\, dv\, A$$
and hence since $\mu$ is radial, the above matrix is scalar and we set
\begin{equation}\label{eq:deftau0} 
     \index{t@$\tau_0$}
 \tau_{0}^2 :=   \int \frac{1  + \frac{2}{3} |v|^2}{( 1 + |v|^2)^{\frac{3}{2}} } \mu(v)\, dv.
\end{equation}
Note that $\tau_0^2$ is therefore positive  for  non identically zero equilibria. It follows from \eqref{Vlasov} and \eqref{Maxwell} that 
\begin{equation}\label{VM-g} 
\left \{ 
\begin{aligned}
\partial_t g + \hat v \cdot \nabla_x g &= \nabla_x (\phi - \hv \cdot A) \cdot \nabla_v\mu,
\\
 - \Delta_x \phi = \rho[g] , &\qquad ( \partial_t^2 - \Delta_x+\tau_0^2)A = \mP \bj[g].
\end{aligned}
\right.
\end{equation}
The system \eqref{VM-g} is finally solved with initial data 
\begin{equation}\label{data-g}
g_{\vert_{t=0}} = g^0, \qquad A_{\vert_{t=0}} = A^0, \qquad \partial_t A_{\vert_{t=0}} = A^1,
\end{equation}
where 
\begin{equation}\label{new-data}
\begin{aligned}
g^0 = f^0 &- A^0 \cdot \nabla_v\mu, 
\\
A^0 = -\Delta_x^{-1} \nabla_x \times B^0, \quad A^1 &= -E^0 + \nabla_x \Delta_x^{-1} \rho[f^0].
\end{aligned}\end{equation}
In what follows, we shall focus on solving the system \eqref{VM-g}-\eqref{new-data}. The solution $(f,E,B)$ to the linearized Vlasov-Maxwell system \eqref{lin-VM} is recovered through the relations \eqref{def-phiA} and \eqref{def-g}.

\subsection{Resolvent equations}

Taking the  Laplace-Fourier transform of \eqref{VM-g} with respect to the variables $(t,x)$, and denoting by $\Tphi_k(\lambda), \TA_k(\lambda), \Tg_k(\lambda,v)$ their Laplace-Fourier coefficients, 
we get the following.

\begin{lemma}\label{lem-resolvent} Let $\Tphi_k(\lambda), \TA_k(\lambda), \Tg_k(\lambda,v)$ be the Laplace-Fourier coefficients in variables $t,x$ of $\phi, A$ and $g$, respectively.  Then, for $k\not =0$ and for all $\Re \lambda>0$, there hold
\begin{equation}\label{resolvent}
\begin{aligned}
D(\lambda,k) \Tphi_k(\lambda) &= \frac{1}{|k|^2}\int \frac{\Fg_k^0}{\lambda +  ik \cdot \hat v} dv,
\\
M(\lambda,k) \TA_k(\lambda) &=   \mP_k\int \frac{\hv \Fg_k^0 }{\lambda +  ik \cdot \hat v} dv + \lambda \FA^0_k + \FA_k^1,
\end{aligned}
\end{equation}
where $D(\lambda,k)$ and $M(\lambda,k)$ are scalar functions defined by 
\begin{equation}\label{def-DMlambda}
\begin{aligned}
   \index{D@$D(\lambda,k)$: electric dispersion function}
      \index{M@$M(\lambda,k)$: magnetic dispersion function}
D(\lambda,k) &:=  1 - \frac{1}{|k|^2} \int \frac{ ik\cdot  \hv }{\lambda +  ik \cdot \hat v} \varphi'(\langle v\rangle) dv,
\\
M(\lambda,k) &:=  \lambda^2 + |k|^2  +\tau_0^2  + \frac12 \int \frac{ (ik\cdot \hv) |\mP_k \hv|^2}{\lambda +  ik \cdot \hat v} \varphi'(\langle v\rangle)  dv,
\end{aligned}
\end{equation}
where  $\mP_k := (\mathbb{I} - \frac{k\otimes k}{|k|^2}) $ is the orthogonal  projection onto the subspace $k^\perp$. 

\end{lemma}

We shall refer to $D$ and $M$ as the electric dispersion function and the magnetic dispersion function, respectively.

\begin{proof}
Taking the Laplace-Fourier transform of \eqref{VM-g} we obtain 
\begin{equation}\label{VM-lambda1}
\begin{aligned}
(\lambda +  ik \cdot \hat v ) \Tg_k &= \Fg_k^0 +ik \cdot \nabla_v\mu (\Tphi_k - \hv \cdot \TA_k) 
\end{aligned}\end{equation}
and 
\begin{equation}\label{VM-lambda2}
\begin{aligned}
|k|^2\Tphi_k &= \Trho_k[g], 
\\(\lambda^2 + |k|^2+\tau_0^2) \TA_k &= \mP_k\Tj_k[g] + \lambda \FA^0_k + \FA_k^1,
\end{aligned}\end{equation}
where $\mP_k = (\mathbb{I} - \frac{k\otimes k}{|k|^2}) $. We first write 
$$  \Tg_k  = \frac{  ik \cdot \nabla_v\mu}{\lambda +  ik \cdot \hat v} (\Tphi_k - \hv \cdot \TA_k)+ \frac{\Fg_k^0}{\lambda +  ik \cdot \hat v},$$
which gives 
$$
\begin{aligned} 
\Trho_k [g] 
&=  \int \frac{ ik\cdot\nabla_v \mu}{\lambda +  ik \cdot \hat v}  (\Tphi_k - \hv \cdot \TA_k) \; dv + \int \frac{\Fg_k^0}{\lambda +  ik \cdot \hat v} dv
\\
&= \Big( \int \frac{ ik\cdot  \nabla_v\mu}{\lambda +  ik \cdot \hat v}  dv\Big)  \Tphi_k-  \int \frac{ ik\cdot\nabla_v \mu}{\lambda +  ik \cdot \hat v}  \hv \cdot \TA_k  dv + \int \frac{\Fg_k^0}{\lambda +  ik \cdot \hat v} dv .
\end{aligned}$$
Let us study the second integral term. For each $k\not =0$, we may write 
\begin{equation}\label{change-vw0}
v = |k|^{-2}(k\cdot v) k + w, \qquad k\cdot w =0.
\end{equation}
It follows that $|v|^2 = |k|^{-2}(k\cdot v)^2 + |w|^2$, and so $\langle v\rangle$ is an even function in $w$. Therefore, recalling that $ik \cdot \TA_k =0$ and  $\nabla_v \mu = \hv \varphi'(\langle v\rangle)$, we compute 
$$
\begin{aligned}
\int \frac{ ik\cdot\nabla_v \mu}{\lambda +  ik \cdot \hat v}  \hv \cdot \TA_k  dv 
&= \int \frac{ ik\cdot\hv}{\lambda +  ik \cdot \hat v}  \frac{w \cdot \TA_k}{\langle v\rangle}  \varphi'(\langle v\rangle) dv =0.
\end{aligned}
$$
This proves that 
$$
\begin{aligned} 
\Trho_k [g] 
&= \Big( \int \frac{ ik\cdot  \nabla_v\mu}{\lambda +  ik \cdot \hat v}  dv\Big)  \Tphi_k + \int \frac{\Fg_k^0}{\lambda +  ik \cdot \hat v} dv,
\end{aligned}$$
which, together with the first equation in \eqref{VM-lambda2}, yields the closed equation for $\Tphi_k(\lambda)$ as stated. 

Similarly, we compute
$$
\begin{aligned} 
\Tj_k [g] 
&=  \int \hv  \frac{  ik \cdot \nabla_v\mu}{\lambda +  ik \cdot \hat v} (\Tphi_k - \hv \cdot \TA_k)\; dv + \int \frac{ \hv \Fg_k^0}{\lambda +  ik \cdot \hat v} dv
\\
&= \Big( \int \frac{ (ik\cdot\hv)\hv}{\lambda +  ik \cdot \hat v} \varphi'(\langle v\rangle) dv\Big)  \Tphi_k  - \Big(\int \frac{(ik\cdot \hv) \hv\otimes \hv }{\lambda +  ik \cdot \hat v} \varphi'(\langle v\rangle)  dv\Big)\TA_k + \int \frac{\hv \Fg_k^0 }{\lambda +  ik \cdot \hat v} dv.
\end{aligned}$$
As above, we write $v = |k|^{-2}(k\cdot v) k + w$, with $k\cdot w =0$, and so 
$$
\int \frac{ (ik\cdot\hv)\hv}{\lambda +  ik \cdot \hat v} \varphi'(\langle v\rangle) dv = -\frac{1}{|k|^2} \int \frac{ (ik\cdot\hv)^2 k}{\lambda +  ik \cdot \hat v} \varphi'(\langle v\rangle) dv + \int \frac{ik\cdot\hv}{\lambda +  ik \cdot \hat v} \frac{w}{\langle v\rangle}\varphi'(\langle v\rangle) dv,
$$
in which the last integral vanishes, since  $\langle v\rangle$ is an even function in $w$. On the other hand, the first integral on the right is parallel to the vector $k$, and so it is in the kernel of the projection $\mP_k$. This proves 
$$
\begin{aligned} 
\mP_k\Tj_k [g] 
&=  - \Big(\mP_k\int \frac{(ik\cdot \hv) \hv\otimes \hv }{\lambda +  ik \cdot \hat v} \varphi'(\langle v\rangle)  dv\Big)\TA_k +  \mP_k\int \frac{\hv \Fg_k^0 }{\lambda +  ik \cdot \hat v} dv .
\end{aligned}$$
Let us further study the first integral term on the right. Again, setting $ w = v -   \frac{(k\cdot v)k}{|k|^2}$ as above and using $\mP_k(k\otimes k) =0$,  the fact that $\langle v\rangle$ is even in $w$ and the radial symmetry,  we can write  
$$
\begin{aligned}
 \mP_k\int \frac{(ik\cdot \hv) \hv\otimes \hv }{\lambda +  ik \cdot \hat v} \varphi'(\langle v\rangle)  \;dv 
 & 
 = \mP_k 
 \int \frac{(ik\cdot \hv) w\otimes w}{\lambda +  ik \cdot \hat v} \frac{\varphi'(\langle v\rangle)}{\langle v\rangle^2}  \;dv
 \\
 & 
 = \frac12 \mP_k 
 \int \frac{(ik\cdot \hv) |w|^2}{\lambda +  ik \cdot \hat v} \frac{\varphi'(\langle v\rangle)}{\langle v\rangle^2}  \;dv.
\end{aligned}
$$
Note that by definition, $w = \mP_k v$ and $\mP_k \TA_k = \TA_k$, since $k \cdot \TA_k =0$. This together with the second equation in \eqref{VM-lambda2} yields the closed equation for $\TA_k(\lambda)$ as stated. The lemma follows. 
\end{proof}

\begin{lemma}\label{lem-rewDM} Let $D(\lambda,k)$ and $M(\lambda,k)$ be defined as in 
\eqref{def-DMlambda}. Then, for each $k\not =0$ and $\Re \lambda > 0$, we may write 
\begin{equation}\label{rew-DMlambda}
\begin{aligned}
D(\lambda,k) &=  1  - \frac{1}{|k|^2} \int_{-1}^1 \frac{ u \kappa(u)}{-i\lambda/|k| +  u}\; du ,
\\
M(\lambda,k) &=  \lambda^2 + |k|^2  + \frac{i\lambda}{2|k|}  \int_{-1}^1 \frac{ q(u)}{-i\lambda/|k| + u} 
\; du ,
\end{aligned}
\end{equation}
where, for $u\in (-1,1)$, we have set 
\begin{equation}\label{def-kqu}
   \index{k@$\kappa(u)$}
      \index{q@$q(u)$}
\begin{aligned}
\kappa(u) :=2\pi\int_{1/\sqrt{1-u^2}}^\infty \varphi'(s) s^2\; ds
, \qquad q(u) :=-4\pi (1-u^2)\int_{1/\sqrt{1-u^2}}^\infty \varphi(s) s\; ds.
\end{aligned}
\end{equation}
Moreover, both $\kappa(u)$ and $q(u)$ are even and non-positive functions on $[-1,1]$. 
\end{lemma}

\begin{proof}
For $k\not =0$, we introduce the change of variables $v \mapsto (u,w)$ defined by 
\begin{equation}\label{change-vw}
u := \frac{k\cdot v}{|k|\langle v\rangle} , \qquad w := v -   \frac{(k\cdot v)k}{|k|^2},
\end{equation}
with $u \in [-1,1]$ and $w\in k^\perp$, the hyperplane orthogonal to $k$. Note that the Jacobian determinant is  
$$J_{u,w} = \langle v\rangle (1-u^2)^{-1}, $$
with  $\langle v\rangle  =\langle w\rangle / \sqrt{1-u^2}$. Therefore, 
\begin{equation}\label{rew-Dlambda}
\begin{aligned}
D(\lambda,k) 
&=  1 - \frac{1}{|k|^2} \int_{\RR^3} \frac{ ik\cdot  \hv }{\lambda +  ik \cdot \hat v} \varphi'(\langle v\rangle) dv
\\
&=  1 - \frac{1}{|k|^2} \int_{-1}^1 \frac{ u }{-i\lambda/|k| +  u} \Big(\int_{w\in k^\perp}\varphi'\Big(\frac{\langle w\rangle}{\sqrt{1-u^2}}\Big) \frac{\langle w\rangle}{(1-u^2)^{3/2}} \; dw\Big) \;du
\\
&=  1 - \frac{1}{|k|^2} \int_{-1}^1 \frac{ u \kappa(u)}{-i\lambda/|k| +  u}  \;du,
\end{aligned}
\end{equation}
which is the formulation \eqref{rew-DMlambda} for $D(\lambda,k)$, where we have set for the moment
\begin{equation}\label{def-kuw}
\begin{aligned}
\kappa(u) :&= \int_{w\in k^\perp}\varphi'\Big(\frac{\langle w\rangle}{\sqrt{1-u^2}}\Big) \frac{\langle w\rangle}{(1-u^2)^{3/2}} \; dw \\
&= \int_{\mathbb{R}^2}\varphi'\Big(\frac{\langle w\rangle}{\sqrt{1-u^2}}\Big) \frac{\langle w\rangle}{(1-u^2)^{3/2}} \; dw.
\end{aligned}
\end{equation}
Before simplifying the expression of $\kappa$,  we can also  compute  $M(\lambda,k) $. First, by radial symmetry, note that 
\begin{equation}
\label{eq:tau0sym}
\tau_0^2 = -\frac13 \int |\hv|^2 \varphi'(\langle v\rangle)\;dv = -\frac12 \int |\mP_k\hv|^2 \varphi'(\langle v\rangle)\;dv
\end{equation}
and so, in view of the definition of $M$ in \eqref{def-DMlambda}, we can write  
\begin{equation}\label{new-Mlambda}
\begin{aligned}
M(\lambda,k) 
&=
\lambda^2 + |k|^2  - \frac{\lambda}{2} \int \frac{ |\mP_k\hv|^2 }{\lambda +  ik \cdot \hat v} \varphi'(\langle v\rangle)  dv.
\end{aligned}
\end{equation}
Now noting that $ \mP_k \hat v = \langle v\rangle^{-1} \mP_k v = \langle v\rangle^{-1}w$ and so 
$$ |\mP_k \hat v|^2 = |w|^2 \langle v\rangle^{-2} = (1-u^2) |w|^2\langle w\rangle^{-2},$$
we thus have 
\begin{equation}\label{rew-Mlambda}
\begin{aligned}
M(\lambda,k) 
&=  \lambda^2 + |k|^2   + \frac{i\lambda}{2|k|} \int_{-1}^1 \frac{1 }{-i\lambda/|k| + u} \Big(\int_{\RR^2}\varphi'\Big(\frac{\langle w\rangle}{\sqrt{1-u^2}}\Big) \frac{|w|^2}{\langle w\rangle\sqrt{1-u^2}} \; dw\Big) \; du 
\\
&=  \lambda^2 + |k|^2  + \frac{i\lambda}{2|k|}   \int_{-1}^1 \frac{ q(u)}{-i\lambda/|k| + u} 
\; du ,
\end{aligned}
\end{equation}
which is the formulation \eqref{rew-DMlambda} for $M(\lambda,k)$, where we have defined 
\begin{equation}\label{def-quw}
\begin{aligned}
q(u) :&= \int_{w\in k^\perp}\varphi'\Big(\frac{\langle w\rangle}{\sqrt{1-u^2}}\Big) \frac{|w|^2}{\langle w\rangle\sqrt{1-u^2}} \; dw \\
&= \int_{\mathbb{R}^2}\varphi'\Big(\frac{\langle w\rangle}{\sqrt{1-u^2}}\Big) \frac{|w|^2}{\langle w\rangle\sqrt{1-u^2}} \; dw.
\end{aligned}
\end{equation}
It remains to prove that $\kappa(u)$ and $q(u)$ defined as in \eqref{def-kuw}-\eqref{def-quw} are indeed of the same form as in \eqref{def-kqu}. To proceed, for each $u \in (-1,1)$, we parametrize the hyperplane $\mathbb{R}^2$ via polar coordinates with radius $r = |w|$, and then introduce $s = \sqrt{1+r^2} / \sqrt{1-u^2}$, giving
$$
\begin{aligned}
\kappa(u) &= \int_{\mathbb{R}^2}\varphi'\Big(\frac{\langle w\rangle}{\sqrt{1-u^2}}\Big) \frac{\langle w\rangle}{(1-u^2)^{3/2}} \; dw= 2\pi \int_0^\infty \varphi'\Big(\frac{\sqrt{1+r^2}}{\sqrt{1-u^2}}\Big)\frac{\sqrt{1+r^2}}{(1-u^2)^{3/2}} \; r dr 
\\&= 2\pi\int_{1/\sqrt{1-u^2}}^\infty \varphi'(s) s^2\; ds.
\end{aligned}
$$
We may now integrate the above integral by parts in $s$ to get 
$$
\begin{aligned}
\kappa(u) &= -\frac{2\pi }{1-u^2}\varphi\Big(\frac{1}{\sqrt{1-u^2}}\Big)  - 4\pi\int_{1/\sqrt{1-u^2}}^\infty \varphi(s) s\; ds.
\end{aligned}
$$
Similarly, we compute
$$
\begin{aligned}
q(u) &=\int_{\mathbb{R}^2}\varphi'\Big(\frac{\langle w\rangle}{\sqrt{1-u^2}}\Big) \frac{|w|^2}{\langle w\rangle\sqrt{1-u^2}} \; dw= 2\pi \int_0^\infty \varphi'\Big(\frac{\sqrt{1+r^2}}{\sqrt{1-u^2}}\Big)\frac{r^2}{\sqrt{1+r^2}\sqrt{1-u^2}} \; r dr 
\\&= 2\pi\int_{1/\sqrt{1-u^2}}^\infty \varphi'(s) \Big( s^2 (1-u^2) - 1\Big)\; ds.
\end{aligned}
$$
Thus, integrating by parts in $s$ gives \eqref{def-kqu}, upon noting that the boundary terms vanish. 
The lemma follows. 
\end{proof}

\begin{remark}\label{rem-tau0} Recall the definition of $\tau_0^2$ in~\eqref{eq:deftau0} and the formula~\eqref{eq:tau0sym}. We have the identity
\begin{equation}\label{comp-tau0}\tau_0^2  = - \frac12 \int  |\mP_k\hv|^2 \varphi'(\langle v\rangle) \; dv = - \frac12 \int_{-1}^1 q(u) \; du
\end{equation}
where $q(u)$ is defined as in \eqref{def-kqu}, by calculations similar to \eqref{rew-Mlambda}-\eqref{def-quw}. 
Likewise, still by radial symmetry, we also have
\begin{equation}\label{comp-tau0-2}\tau_0^2  =- \int_{-1}^1 u^2 \kappa (u) \; du.
\end{equation}
\end{remark}

\subsection{Spectral stability}

In this section, we prove that there is no exponential growing mode of the linearized Vlasov-Maxwell system \eqref{VM-g}. Namely, we obtain the following. 

\begin{proposition}\label{prop-nogrowth} 
For any non-negative radial equilibria $\mu(v)= \varphi(\langle v\rangle)$ in $\RR^3$, the linearized system \eqref{VM-g} has no nontrivial solution of the form $e^{\lambda t + ik\cdot x} (\Tg_k, \Tphi_k, \TA_k)$ with $\Re \lambda \neq 0$ for any nonzero triple $(\Tg_k, \Tphi_k, \TA_k)$. 
\end{proposition}

\begin{proof} In view of the resolvent equations \eqref{resolvent}, it suffices to prove that  for each $k\in \RR^3$, there are no zeros of the electric or magnetic dispersion relation: $D(\lambda,k) =0$ or $M(\lambda,k)=0$ with $\Re \lambda \neq 0$. 

We first study the electric dispersion relation $D(\lambda,k) =0$. Using \eqref{rew-DMlambda} for $\lambda = \gamma + i\tau$, we write 
$$
\begin{aligned}
D(\gamma + i\tau,k) &=  1  - \frac{1}{|k|^2} \int_{-1}^1 \frac{ u \kappa(u)}{-i\gamma/|k| + \tau/|k|+  u}\; du 
\\
&=  1  - \frac{1}{|k|^2} \int_{-1}^1 \frac{ u (\tau/|k|+  u)\kappa(u)}{|-i\gamma/|k| + \tau/|k|+  u|^2}\; du - \frac{i\gamma}{|k|^3} \int_{-1}^1 \frac{ u \kappa(u)}{|-i\gamma/|k| + \tau/|k|+  u|^2}\; du .
\end{aligned}
$$
Now fix $k\in \RR^3$, and suppose that $D(\gamma + i\tau,k)=0$ for some $\gamma\neq0$. Taking the imaginary part of the above identity yields 
$$ \int_{-1}^1 \frac{ u \kappa(u)}{|-i\gamma/|k| + \tau/|k|+  u|^2}\; du  = 0.$$
Plugging this identity into $D(\gamma+ i\tau,k)$, we get 
$$
\begin{aligned}
D(\gamma + i\tau,k) 
& = 1 - \frac{1}{|k|^2} \int_{-1}^1 \frac{u^2\kappa(u)}{|-i\gamma/|k| + \tau/|k|+  u|^2}\; du,
\end{aligned}
$$
which never vanishes, since $\kappa(u)\le 0$ by Lemma \ref{lem-rewDM}. That is, $D(\lambda,k)$ never vanishes for $\Re \lambda \neq 0$. 

Similarly, we study the magnetic dispersion relation $M(\lambda,k)=0$. Using \eqref{rew-DMlambda} 
for $\lambda = \gamma + i\tau$, we compute 
$$
\begin{aligned}
M(\gamma+ i\tau,k) &= \gamma^2 - \tau^2 + 2 i \gamma \tau + |k|^2  - \frac{|\gamma+i\tau|^2}{2|k|^2}  \int_{-1}^1 \frac{ q(u)}{|-i\gamma/|k| + \tau/|k|+  u|^2} 
\; du  \\&\quad + \frac{i(\gamma + i\tau)}{2|k|}  \int_{-1}^1 \frac{ uq(u)}{|-i\gamma/|k| + \tau/|k|+  u|^2} 
\; du.
\end{aligned}$$
Now suppose that $M(\gamma+ i\tau,k) =0$ for some $\gamma \neq 0$. The vanishing of the imaginary part gives 
$$
2\tau  + \frac1{2|k|} \int_{-1}^1 \frac{ uq(u)}{|-i\gamma/|k| + \tau/|k|+  u|^2} 
\; du = 0.
$$ 
Plugging this identity into $M(\gamma+ i\tau,k)$, we get 
$$
\begin{aligned}
M(\gamma+ i\tau,k) 
&= 
\gamma^2 + \tau^2 + |k|^2 -  \frac{|\gamma+i\tau|^2}{2|k|^2} \int_{-1}^1 \frac{ q(u)}{|-i\gamma/|k|+\tau/|k| +  u|^2} \; du,
\end{aligned}$$
which again never vanishes, since $q(u)\le 0$ by Lemma \ref{lem-rewDM}. That is, $M(\lambda,k)$ never vanishes for $\Re \lambda \neq 0$. The proposition follows. 
\end{proof}

In the next sections, we initiate a systematic study of the dispersion functions $D$ and $M$, providing in particular expansions and reformulations that will prove useful for the study of the dispersion relations, see the upcoming Sections~\ref{sec-D} and \ref{sec-M}.

\subsection{The electric dispersion function $D$}
\label{sec:D}

Let us first consider $D$. By definition, for $\Re \lambda>0$, we may write 
\begin{equation}\label{tow-L1}
\begin{aligned}
D(\lambda,k) 
& = 1 - \frac{1}{|k|^2} \int \frac{ik \cdot \hv}{ \lambda + ik \cdot \hat v} \varphi'(\langle v\rangle)\; dv 
\\
&= 1 - \frac{1}{|k|^2} \int_0^\infty e^{-\lambda t} \int e^{-ik t \cdot \hat v}ik \cdot \hv \varphi'(\langle v\rangle)\; dv dt .
\end{aligned}
\end{equation}
Introduce \begin{equation}\label{def-Kt}
   \index{K@$K_k(t)$}
K_k(t) := -\frac{1}{|k|^2} \int e^{-ikt \cdot \hat v } ik  \cdot \hv \varphi'(\langle v\rangle) \; dv,
\end{equation} 
and denote by $\mathcal{L}[F]$ the Laplace transform of $F(t)$, we thus have 
\begin{equation}\label{rew-DLKt}
\begin{aligned}
D(\lambda,k) &
 = 1 + \cL[K_k(t)](\lambda) = 1 + \int_0^\infty e^{-\lambda t} K_k(t) dt.
\end{aligned}
\end{equation}
We start with the following lemma.

\begin{lem}
\label{lem:Kk} 
For all $k \neq 0$, the function $K_k$ is $\mathcal{C}^\infty$ and for every $ N\leq (N_{0}-5)/2$,  we have for all $n\ge 0$, for all $k \in \RR^3$ and all $t \geq 0$,
\begin{equation}\label{bounds-Kkt}
|\partial_t^n K_k(t)|\le C_{n,N} |k|^{n-1} \langle kt \rangle^{-N} ,
\end{equation}
 where $C_{n,N}>0$ depends only on  $n$ and  $N$.
\end{lem}

\begin{proof}

Thanks to the decay at infinity of $\mu$ and its derivatives, for all $k \neq 0$, $K_k(t)$ is of class $\mathcal{C}^n$ for all $n$ and we have 
\begin{multline*}
\partial_t^n K_k(t) = -\frac{1}{|k|^2} \int_{\mathbb{R}^3} e^{-ikt \cdot \hat v } (-i k \cdot \hat v)^n ik  \cdot \hv \varphi'(\langle v\rangle) \; dv
\\ = -  \frac{1}{|k|^2} \int_{|w| \leq 1} e^{-ikt \cdot w } (-i k \cdot w)^n ik  \cdot w\, \varphi'\left(\frac{1}{ (1-|w|^2)^{1 \over 2}}\right) \frac{1}{ (1- |w|^2)^{5 \over 2}} \; dw
\end{multline*}
where we have set $w= \hat{v}$ to get the second expression.
By using that 
$$\frac{k}{|k|^2} \cdot \nabla_w e^{-ikt \cdot w }= -i t  e^{-ikt \cdot w }, $$  we can integrate by parts $N$ times in $w$
and use that there is no boundary term because of the fast enough decay of $\varphi$ and of its successive derivatives given by \eqref{decay-mu}
 to get  the bound for all $k \neq 0$ and all $t \geq 0$
\begin{equation}\label{bounds-Kktproof}
|\partial_t^n K_k(t)|\le C_{n,N} |k|^{n-1} | kt |^{-N}.
\end{equation}
Since for $|k| t \leq 1$, we can just use the straightforward bound $|\partial_{t}^n K|\lesssim |k|^{n-1}$ is bounded, this ends the proof.
\end{proof}

\begin{rem}
\label{remarque2.5} Note that 
by integration by parts in $w$, we can also write that 
$$
\begin{aligned}
 \int e^{-ikt \cdot \hat v } ik  \cdot \hv \varphi'(\langle v\rangle) \; dv &=  \int e^{-ikt \cdot \hat v } ( ik  \cdot \nabla_v)  \varphi(\langle v\rangle) \; dv \\
 &= -   \int e^{-ikt \cdot \hat v }  t \left(\frac{ |k|^2}{\langle v \rangle} -\frac{(k\cdot v)^2 }{\langle v \rangle^3} \right)  \varphi(\langle v\rangle) \; dv,
 \end{aligned}
$$
thus for all $k \in \mathbb{R}^3$, $| K_k(t)| \lesssim  t$; in particular $K_k(t)$ is actually not singular at $k = 0$. 

\end{rem}

As a first consequence of Lemma~\ref{lem:Kk}, we have

\begin{cor}
\label{coro:LKK}
For all $n \in\mathbb{N}$, such that $ n \leq (N_{0}-9)/2$,  there exists $C_n>0$ such that, for all $\lambda \in \mathbb{C}$ with $\Re \lambda \geq 0$,  for all $ k \neq 0$,
\begin{equation}
\label{eq:LKK}
|\partial_\lambda^{n} \cL[K_k](\lambda)| \leq   \frac{C_n}{(|k|^2 + |\lambda|^2)^{ n/2 +1}}.
\end{equation}
\end{cor}

\begin{proof}
We have
$$\partial_\lambda^{n} \cL[K_k](\lambda) = (-1)^n \int_0^{\infty} e^{-\lambda t} t^n K_k(t) \; dt =: (-1)^n  \mathcal{K}_{n}(\lambda, k).$$
We first observe that $\mathcal{K}_{n}$ is homogeneous degree $-n-2$: by setting $r= |\lambda|+ |k|$, we have that
$$  \mathcal{K}_{n}(\lambda, k)= r^{-n-2} \mathcal{K}_{n}(\tilde \lambda, \tilde k)$$
where $\tilde \lambda= \lambda/r$, $\tilde k = k/r$  and thus $|\tilde \lambda| + |\tilde k|=1$.
 To get the claim, it thus suffices to show that $ \mathcal{K}_{n}$ is bounded on  $\mathbb{S}_{+}= \{ (\lambda, k), \,  |\lambda| + | k|=1, 
 \, \Re \lambda \geq 0\}$.
 
 If $ 1 \geq |k| \geq 1/2$,  we can just use \eqref{bounds-Kkt} with $n=0$ and $N=n+2$ to get
 $$|\mathcal{K}_{n}(\lambda, k)| \leq C \int_{0}^{+ \infty} { 1 \over \langle t \rangle^2} \, dt \lesssim 1.$$
  If  $| k | \leq 1/2$, we must have $|\lambda |\geq 1/2$. For the case $| \Re \lambda | \geq 1/4$, we can just use Remark \ref{remarque2.5} to obtain that
  $$  |\mathcal{K}_{n}(\lambda, k)| \leq C \int_{0}^{+ \infty}  t^{n+1}e^{- \Re \lambda t}\, dt \lesssim 1.$$
   Finally, if $|\Im \lambda| \geq 1/4$, by $(n+2)$ successive integrations by parts in $t$, noting that $(t^n K_k)^{(j)}(0)=0$ for $j=0,\ldots n$, 
 and  $(t^n K_k)^{(n+1)}(0)=n ! K_{k}'(0)$,  we get 
$$  |\mathcal{K}_{n}(\lambda, k)| \lesssim |K_{k}'(0)|  + \int_{0}^{ +\infty} \sum_{l=0}^{n} |t|^{(n-l)} |K_{k}^{(n+2-l)}(t)| \, dt$$
and hence
thanks to Lemma~\ref{lem:Kk} for $N= (n-l)+ 2$, we deduce that 
$$   |\mathcal{K}_{n}(\lambda, k)|
 \lesssim 1  + \int_{0}^{+\infty} |t|^{(n-l)} \frac{ |k|^{n+1-l}}{ \langle |k|t \rangle^{n-l+ 2} } \, dt
  \lesssim 1 +  \int_{0}^{+\infty} |s|^{(n-l)} \frac{ 1}{ \langle s\rangle^{n-l+ 2} } \, ds \lesssim 1.$$
  The corollary follows.
\end{proof}

The main statements concerning the electric dispersion function $D$ are gathered in the following proposition.
\begin{prop}
\label{prop:D} Let $D(\lambda,k)$ be defined as in \eqref{rew-DMlambda}. The following holds.

\begin{enumerate}
\item For all $k \neq 0$, the function $\lambda \mapsto D(\lambda, k)$ is holomorphic on $\{\Re \lambda >0\}$
with a continuous extension   to $\{\Re \lambda =0\}$ and the function $\tau \mapsto D(i \tau , k)$ is $\mathcal{C}^K$
 on $\mathbb{R}$ for $k \neq 0$ with $K \leq (N_{0}-9)/2$. 
Moreover, 
 there exists $C_0>0$ such that for all $|k|\geq C_0$ and all $\Re \lambda \geq 0$, $|D(\lambda,k)| \geq 1/2$.

For all $m \in \mathbb{N}$ we have the expansion for all $k \neq 0$ and all $\Re \lambda \geq 0$, $\lambda \neq 0$,
\begin{equation}
\label{expansion-DL}
\begin{aligned}
D(\lambda,k)
 &=1+  \sum_{j = 0}^{m}\frac{1}{\lambda^{2j+2}} \partial_t^{2j+1}[K_k](0) 
+  R_{m}^D(\lambda, k),
\end{aligned}\end{equation}
where 
\begin{equation}\label{comp-dtjKprop}
\begin{aligned}
\partial^{2j+1}_t[K_k](0) &=   \frac{(-1)^{j+1}}{|k|^{2}} \int (k \cdot \hat v)^{2j+2} \varphi'(\langle v\rangle) \; dv , 
\\
 R_{m}^D(\lambda, k) &= \frac{1}{\lambda^{2m+2}} 
\int_0^\infty e^{- \lambda  t} \partial_t^{2m+2}K_k(t) \; dt .
\end{aligned}
\end{equation}
In addition, the remainder $R_{m}^D(\lambda, k)$ satisfies for all $k \in \RR^3$ and all $\Re \lambda >0$
\begin{equation}
\label{bd-RK}
\Big| R_{m}^D(\lambda, k) \Big| \lesssim_m \frac{|k|^{2m}}{|\lambda|^{2m+2}}.
\end{equation}

\item For all $k \neq 0$, for all $|\tau|\geq  |k|$, there holds
\begin{equation}
\label{eq:formulaDtauleqk}
   \index{O@$\Omega(y)$}
D(i\tau ,k) = 1 - \frac{1}{|k|^{2}} \Omega(|k|^2/\tau^2),\qquad \Omega(y) :=  -\int_{-1}^1 \frac{y u^2}{1-y u^2}  \kappa(u) \; du 
\end{equation}
for $\kappa(u)$ defined as in \eqref{def-kqu}.
Furthermore, the function $\Omega(y)$ is a $\mathcal{C}^K$  function on $[0,1]$, for $ K \leq \frac{N_{0}}{2}-1$,  with $ \sup_{[0,1]} |\Omega^{(l)}| \lesssim_l 1$. 

\item For all $\lambda \notin i[-|k|,|k|]$, $\Re \lambda \geq 0$, there holds 

\begin{equation}
\label{eq-formulaD}
   \index{P@$\Phi(z)$}
D(\lambda,k) = 1 - \frac{1}{|k|^2} \Phi(-i\lambda /|k|),\qquad \Phi(z) :=\int_{-1}^1 \frac{u\kappa(u)}{u+z} \, du.
\end{equation}
The function $\Phi$ is holomorphic in $\mathbb{C}\setminus[-1,1]$ with a continuous extension (together with its derivatives) from
$\{ \Im \, z<0\}$ to $\{ \Im z \leq 0\}$
 and one can find :
\begin{itemize}
\item a  function  $\index{P@$\widetilde\Phi(z)$} \widetilde{\Phi}(z)$ which is holomorphic on 
 $\{\Im z <0\} \cup \{ \Im z \geq 0, 1 - (\Re z)^2 + (\Im z)^2>0\}$,
  and  coincides with the continuous extension of  $\Phi$ on $\{\Im z \leq 0\}$;
\item for any $R$ sufficiently large  a $ \mathcal{C}^{N_0/2-3}$ function   $\index{P@$\widetilde{\widetilde{\Phi}}(z)$}\widetilde{\widetilde{\Phi}}$  defined in the whole large box
$\{ | \Re z | < R, |\Im z| < R\}$ 
which coincides with $\widetilde{\Phi}$ on $\{ |\Re z| \leq 1\}$ and  with the continuous extension of  $\Phi$  on $ \{\Im z \leq 0\}$.  Moreover,
\begin{equation}
\label{eq:diffpm1}
\mathrm{d} \widetilde{\widetilde{\Phi}}(\pm 1)   = \left( \int_{-1}^1 \frac{u\kappa(u)}{(u\pm1)^2} \, du\right) \mathrm{I}  \neq 0
\end{equation}
where $\mathrm{d} \cdot$ stands for the differential of a function seen as a $\mathcal{C}^1$ function on $\mathbb{C}$
 identified with $\mathbb{R}^2$.
\end{itemize}

\end{enumerate}

\end{prop}

\begin{rem}
The only place in the  paper where the assumption on the holomorphy of  $\varphi(s) = F(s^2)$  is useful is for the proof of the  item 3. of Proposition~\ref{prop:D}. This will be crucial in our study of the dispersion relation in the vicinity of the imaginary axis and near the
threshold frequency $ k= \kappa_{0}$.
\end{rem}

As a straightforward application of the last item of Proposition~\ref{prop:D}, we have
\begin{cor}
\label{coro:extenD}
For all $k\neq 0$, the electric dispersion function $D(\cdot,k)$  has:
\begin{itemize}
\item a holomorphic extension, still denoted by $D$, on $\{z \in \mathbb{C},  z  \notin i[-|k|, |k|]\}$,  which coincides with  $D(\cdot,k)$ on $\{z \in \mathbb{C}, \Re z \geq 0, \,  z \notin i[-|k|, |k|]\}$ and this function $D$  has a continuous extension (together with its derivatives) from $\{\Re z >0 \}$ to $\{\Re z \geq 0\}$;
\item a holomorphic extension, denoted by $\index{D@$\widetilde{D}(\lambda,k)$: extension of $D$} \widetilde{D}$, on $\{z \in \mathbb{C}, \Re z \leq 0,  |\Im z |^2< |k|^2 + |\Re z|^2\} \cup \{\Re z >0\}$,  which 
has a continuous extension together with its derivatives up to the boundary and 
coincides with the continuous extension of   $D(\cdot,k)$ on $\{z \in \mathbb{C}, \Re z \geq 0\}$;
\item a $\mathcal{C}^{N_0/2-3}$ extension, denoted by  $\index{D@$\widetilde{\widetilde{D}}(\lambda,k)$: extension of $D$} \widetilde{\widetilde D}$  in the whole box $\{ | \Re z | < R, |\Im z| < R\}$ 
 for $R$ arbitrarily large, which coincides with the continuous extension of  $D(\cdot,k)$ on $\{  \Re z \geq  0\} $  and with $\widetilde D$ on $\{z \in \mathbb{C}, \Re z \leq 0, \,  |\Im z |< |k|\}.$
\end{itemize}

\end{cor}

\begin{proof}[Proof of Proposition~\ref{prop:D}]

\noindent $1.$ We recall from \eqref{rew-DLKt} that $D(\lambda,k) 
 = 1 + \cL[K_k(t)](\lambda)$. 
By Lemma~\ref{lem:Kk}, namely estimate~\eqref{bounds-Kkt} for $n=0$ and $N=2$, the Laplace transform $\cL[K_k(t)](\lambda)$ is holomorphic in $\Re \lambda >0$ and well-defined up to the imaginary axis $\Re \lambda \ge 0$. In addition,  still using \eqref{bounds-Kkt}, there holds
\begin{equation}\label{bdLtau0}
|\cL[K_k](\lambda)| \le C_0 \int_0^\infty |k|^{-1} \langle kt \rangle^{-2} \; dt \le C_0 |k|^{-2},
\end{equation}
uniformly in $k$ and $\Re \lambda \ge 0$. In particular, we note that we have $|D(\lambda,k)|\ge 1/2$ for all $|k|\ge \sqrt{2C_0}$. 

Next, for $\lambda \not=0$, writing $e^{-\lambda  t} = -\lambda^{-1} \partial_t e^{-\lambda t}$, and integrating by parts repeatedly in $t$, for any $m\ge 0$, we have 
$$
\begin{aligned}
 \int_0^\infty e^{-\lambda  t} K_k(t)\; dt  
= \sum_{j = 0}^{m}\frac{1}{\lambda^{j+1}} \partial_t^j[K_k](0)
 + \frac{1}{\lambda^{m+1}}\int_0^\infty e^{- \lambda  t} \partial_t^{m+1}K_k(t) \; dt ,
\end{aligned}
$$
for $\Re \lambda \ge 0$. Note that by recalling \eqref{bounds-Kkt},
 there is no boundary contribution at $t = \infty$, since 
$$ 
\lim_{t\to \infty} |e^{-\lambda t} \partial_t^j K_k(t) | \le C_j \lim_{t\to \infty} |k|^{j-1} e^{-\Re \lambda t}\langle kt \rangle^{-1}= 0,
$$
for all $j \ge 0$ and $\Re \lambda \ge 0$. 
We have
\begin{equation}\label{comp-dtjK}
\begin{aligned}
\partial^j_t[K_k](0) &=   \frac{(-i)^{j+1}}{|k|^{2}} \int (k \cdot \hat v)^{j+1} \varphi'(\langle v\rangle) \; dv , 
\end{aligned}
\end{equation}
for all $j\ge 0$. In particular, $\partial^{2j}_t[K_k](0) =0$ and $\partial^{2j+1}_t[K_k](0) =\cO(|k|^{2j})$ for all $j\ge 0$. Therefore, for any $m\ge 0$ and $\Re \lambda \ge0$, we obtain 
\begin{equation}\label{expansion-DLproof}
\begin{aligned}
 \cL[K_k](\lambda  ) 
 &= \sum_{j = 0}^{m}\frac{1}{\lambda^{2j+2}} \partial_t^{2j+1}[K_k](0)
 + \frac{1}{\lambda^{2m+2}}\int_0^\infty e^{- \lambda  t} \partial_t^{2m+2}K_k(t) \; dt ,
\end{aligned}\end{equation}
where, using again \eqref{bounds-Kkt}, the last integral term satisfies
\begin{equation}\label{bd-RKproof}
\Big|\int_0^\infty e^{-\lambda  t} \partial_t^{2m+2}K_k(t) \; dt \Big| \le C_{m} |k|^{2m}
\end{equation}
uniformly in small $k$, for $m\ge 0$ and $\Re \lambda \ge 0$. Finally, the fact that $\tau \mapsto D(i \tau , k)$ is  $\mathcal{C}^K$ for $ K\leq (N_{0}-9)/2$ on $\mathbb{R}$ follows from Corollary~\ref{coro:LKK} and Lebesgue's domination theorem.

\bigskip

\noindent $2.$ For $\lambda= i \tau$  and $|\tau|>|k|$, we recall from \eqref{rew-DMlambda} that 
$$
\begin{aligned}
D(i\tau,k) &=  1  - \frac{1}{|k|^2} \int_{-1}^1 \frac{ u \kappa(u)}{\tau/|k| +  u}\; du 
\\
&=  1  + \frac{1}{|k|^2} \int_{-1}^1 \frac{ u^2 \kappa(u)}{\tau^2/|k|^2 -  u^2}\; du ,
\end{aligned}
$$
in which the evenness of $\kappa(u)$ was used in the last identity. This yields \eqref{eq:formulaDtauleqk}. 
By using \eqref{decay-mu}, we deduce that for $|u|$ close to one, we have the estimate
\begin{equation}
\label{kappaprochedeun} | \kappa (u)| \lesssim |1-|u||^{ \frac{N_{0}}{2} -1}.
\end{equation}
 We thus deduce that $\Omega$ is a $\mathcal{C}^K$ function on $[0,1]$, for $K \leq \frac{N_{0}}{2}-1$, with $\sup_{[0,1]} |\Omega^{(l)}| \lesssim_l1$. 

\bigskip 

\noindent $3.$ The representation \eqref{eq-formulaD} reads off from \eqref{rew-DMlambda}, upon defining  
$\Phi(z) =\int_{-1}^1 \frac{u\kappa(u)}{u+z} \, du.
$
By definition,  since $\kappa$ is smooth on $(-1, 1)$ and with the property \eqref{kappaprochedeun} thanks to the decay of $\varphi,$     
$\Phi(z)$ is clearly holomorphic on $\mathbb{C} \setminus [-1,1]$.
Furthermore, since we also get from the decay of $\varphi$ that
\begin{equation}
\label{annulkappa}  | (u \kappa)^{(l)} (u)| \lesssim |1-|u||^{ \frac{N_{0}}{2} -1 -l}
\end{equation}
for $|u|$ close to one, we get by integration by parts that 
we have for all $l \leq N_{0}/2- 1$,  and for all $z \in \mathbb{C} \setminus [-1,1]$,
\begin{equation}
\label{laderiveedordreldephi}
\Phi^{(l)}(z) = (-1)^{l} \int_{-1}^1 \frac{(u\kappa)^{(l)}(u)}{(u+z)} \, du.
\end{equation}
Next, we claim that the function $\kappa$ admits a holomorphic extension on  
$  \{ \Im z \geq 0, 1 - (\Re z)^2 + (\Im z)^2>0\}$.
Indeed, we have that 
$$ \kappa(u)= 2 \pi \int_{\frac{1}{\sqrt{1- u^2}}}^{+ \infty} \varphi'(s) s^2\, ds= 4\pi \int_{1}^{+\infty} F( \frac{s^2}{1- u^2}) \frac{s^3}{(1-u^2)^2}
\, ds= \mathcal{K}(\frac{1}{1- u^2})$$
where $\mathcal{K}$ is defined by
$$ \mathcal{K}(z)= 4 \pi \int_{1}^{+\infty}  F(z s^2) z^2 s^3\, ds.$$
Thanks to the assumption (H2), for $N_{0}$ sufficiently large which ensures the absolute convergence of the integral on 
$\{\Re z \geq \alpha\}$ for every $\alpha >0$, we get that
  $\mathcal{K}$ is holomorphic on $\{\Re z>0\}$. Since the map $u \mapsto { 1 \over 1-u^2}$ sends every compact set of 
   $  \{ \Im z \geq 0, 1 - (\Re z)^2 + (\Im z)^2>0\}$ onto a compact set of $\{\Re z>0\}$, we finally get that
   $\kappa$ is holomorphic on  $  \{ \Im z \geq 0, 1 - (\Re z)^2 + (\Im z)^2>0\}$.

By the Plemelj fomula, the function $\widetilde{\Phi}(z)$  defined as
\begin{equation}\label{eq-extensionPhi}
\widetilde{\Phi}(z)
= \left \{ \begin{aligned} \Phi(z) \quad& \quad \mbox{if}\quad \Im z<0,\\
 P.V. \int_{-1}^1 \frac{u\kappa(u)}{u+ z} \, du - i \pi z \kappa( z)  \quad& \quad \mbox{if}\quad \Im z= 0, \, |\Re z|<1, \\
\Phi(z) - 2i \pi z \kappa(z)  \quad& \quad \mbox{if}\quad  \Im z>0,\, 1 - (\Re z)^2 + (\Im z)^2>0
  \end{aligned}\right.
\end{equation}
is a holomorphic function on $ \{ \Im z<0\} \cap \{ \Im z \geq 0, 1 - (\Re z)^2 + (\Im z)^2>0\}$
, which by construction coincides with $\Phi$ on $\{\Im z <0\}$. Note that for the construction of this extension, we need
the holomorphy of $\kappa$ and hence that of $\phi$.

Let us also observe that if we extend the above definition by 
$$ \tilde{\Phi}(z)= \int_{-1}^1 \frac{u\kappa(u)}{u+ z} \, du, \quad \Im z=0, \, |\Re z| > 1, \quad \tilde \Phi(\pm 1)= 0, $$ 
we get that $\tilde{\Phi}$ and its derivatives extend continuously from $\{\Im z <0\}$ to $\{\Im z =0\}.$
This is a direct consequence of \eqref{laderiveedordreldephi}
  and \eqref{annulkappa}.

For the last item, we shall use a smooth non analytic extension.
We first observe that $\tilde{\Phi}$ is by construction analytic and thus $\mathcal{C}^\infty$ in $\mathcal{O}=\{ | \Re z | <1 \} \cup \{\Im z <0\}$. Moreover,  still by construction,
$\tilde{\Phi}$ coincides with the continuous extension of $\Phi$ in $\{ \Im z \leq 0\}$ and is holomorphic in  $ \{ \Im z<0\} \cup \{ \Im z \geq 0, 1 - (\Re z)^2 + (\Im z)^2>0\}$, 
we thus get that in the vicinity  of every point of the boundary  of $\mathcal{O}$ except $z= \pm 1$, 
 all the partial derivatives extend continuously and thus are bounded.
  It remains to study the limit when $z \rightarrow \pm 1$, with  $|\Re z |<1$.  By using \eqref{laderiveedordreldephi}
  and \eqref{annulkappa}, we also get that all the partial derivatives of $\tilde \Phi$ of order less than $N_{0}/2-2$
  have a limit when $z \rightarrow \pm 1$ which is given by taking
  the value $z=\pm1$ in the integral. 
   To get the extension, we can  thus  apply Whitney's extension theorem  \cite{Whi}. 
Alternatively, we can observe that $\mathcal{O}\cap \{ | \Re z | < R, |\Im z| < R\}$  is a Lipschitz domain and that we have just obtained that  $\tilde{\Phi}$ is $W^{k, \infty}$
    for every $k$ less than  $ k \leq N_{0}/2-2$ .
     This allows to use Stein extention theorem \cite{Steinext} to get a $W^{k, \infty}$ extension in the whole space.

This ends the proof of the proposition. 
\end{proof}

\subsection{The magnetic dispersion function $M$}
\label{sec:M}

Let us now continue with the study of the magnetic dispersion function $M$.
As for the electric dispersion function, we can recast $M$ 
as 
\begin{equation}\label{rep-Mlambda}
M(\lambda,k) =  \lambda^2 + |k|^2  +\tau_0^2  +  \cL[N_k(t)](\lambda)
\end{equation}
recalling $\cL[N_k(t)](\lambda) = \int_0^\infty e^{-\lambda t} N_k(t)\; dt $,
where we have set 
\begin{equation}\label{def:Nk}
\index{N@$N_k(t)$} 
N_k(t) := \frac12 \int e^{-ikt \cdot \hv}  (ik\cdot \hv) |\mP_k\hv|^2\varphi'(\langle v\rangle)  dv .
\end{equation}
\begin{lem}
\label{lem:Nk} 
For all $k \neq 0$, the function $N_k(t)$ is $\mathcal{C}^\infty$ and for all $ N\leq (N_{0}-5)/2$ we have for all $n\ge 0$, for all $k \in \RR^3$ and all $t \geq 0$,
\begin{equation}\label{bounds-Nkt}
|\partial_t^n N_k(t)|\le C_{n,N} |k|^{n+1} \langle kt \rangle^{-N} ,
\end{equation}
 where  $C_{n,N}$ depends only on $n$ and $N$.
\end{lem}
\begin{proof} The proof follows is similar to that of Lemma \ref{lem:Kk} and is therefore skipped.
\end{proof}

Exactly as for Lemma~\ref{lem:Kk} and Corollary~\ref{coro:LKK}, the following corollary is a consequence of Lemma~\ref{lem:Nk}.  

\begin{cor}
\label{coro:LNK}
For all  $ n  \leq (N_{0} - 9)/2$, there exists $C_n>0$ such that, for all $\lambda \in \mathbb{C}$ with $\Re \lambda \geq 0$,  for all $ k \neq 0$,
\begin{equation}
\label{eq:LNK}
|\partial_\lambda^{n} \cL[N_k](\lambda)| \leq  C_n  \frac{|k|^2}{(|k|^2 + |\Im \lambda|^2)^{ n/2 +1}}.
\end{equation}
\end{cor}

We finally study the magnetic dispersion function $M(\lambda,k)$.

\begin{prop}
\label{prop:M} Let $M(\lambda,k)$ be defined as in \eqref{rew-DMlambda}. Then, the following holds.
\begin{enumerate}
\item For all $k \neq 0$, the function $\lambda \mapsto M(\cdot, k)$ is holomorphic on $\{\Re \lambda > 0\}$ with a continuous extension to  $\{\Re \lambda=0\}$ and the function $\tau \mapsto M(i \tau , k)$ is  $\mathcal{C}^K$ on $\mathbb{R}$ for $ K\leq (N_{0}-9)/2$. Furthermore for all $m \in \mathbb{N}$ we have the expansion for all $k \in \RR^3$ and all $\Re \lambda > 0$
\begin{equation}
\label{eq:expansionM}
\begin{aligned}
M(\lambda,k)
 &=\lambda^2 + |k|^2  + \tau_{0}^2+   \sum_{j = 0}^{m}\frac{1}{\lambda^{2j+2}} \partial_t^{2j+1}[N_k](0)
 +  R_{m}^M(\lambda, k),
\end{aligned}\end{equation}
where 
\begin{equation}\label{comp-dtjMprop}
\begin{aligned}
\partial_t^{2j+1}N_k(0)&= \frac{(-1)^j}{2} \int (k\cdot \hv)^{2j+2} |\mP_k\hv|^2\varphi'(\langle v\rangle)  dv,
\\
  R_{m}^M(\lambda, k)&= \frac{1}{\lambda^{2m+2}}\int_0^\infty e^{- \lambda  t} \partial_t^{2m+2}K_k(t) \; dt .\end{aligned}
\end{equation}
In addition, the remainder $R_{m}^M(\lambda, k)$ satisfies for all $k \in \RR^3$ and all $\Re \lambda > 0$,
\begin{equation}
\label{eq:remainderM}   
\Big|R_{m}^M(\lambda, k) \Big| \lesssim_m \frac{|k|^{2m+2}}{|\lambda|^{2m+2}}.
\end{equation}

\item For all $k \neq 0$, for all $|\tau|>  |k|$, there holds:
\begin{equation}
\label{eq:formulaMtauleqk}
\index{p@$\psi(y)$} 
M(i\tau,k) = -\tau^2 + |k|^2  + \psi(|k|^2/\tau^2), \qquad 
\psi(y) := - \frac12 \int_{-1}^1 \frac{q(u)}{1- y u^2}\;du,
\end{equation}
for $q(u)$ as in \eqref{def-kqu}.
Furthermore, $\psi$ is a $\mathcal{C}^{K}$  function on $[0,1]$, for $ K\leq N_{0}/2$ with $ \sup_{[0,1]} |\psi^{(l)}| \lesssim _{l}1$, with $\psi(0) = \tau_0^2$. 

\item For all $\lambda \notin i[-|k|,|k|]$, $\Re \lambda \geq 0$, there holds
\begin{equation}\label{eq-formulaM}
\index{L@$\Lambda(z)$} 
M(\lambda,k) = \lambda^2 +|k|^2 - \Lambda(-i \lambda/|k|), \qquad \Lambda(z) := \frac{z}{2} \int_{-1}^1  \frac{q(u)}{ z+  u}  \; du, 
\end{equation}
in which the function $\Lambda$ is holomorphic on $\mathbb{C}\setminus [-1,1]$.
\end{enumerate}
\end{prop}

As a straightforward application of the last item of Proposition~\ref{prop:M}, we have
\begin{cor}
\label{coro:extenM}
For all $k\neq 0$, the magnetic dispersion function $M(\cdot,k)$ admits a holomorphic extension (still denoted by ${M}$) on  $\{z \in \mathbb{C},  |\Im z |> |k|\}\cup\{\Re z >0\}$  which coincides with $M(\cdot,k)$ on $\{z \in \mathbb{C}, \Re z \geq 0,  |\Im z |>|k|\}$. 
\end{cor}

\begin{proof}[Proof of Proposition \ref{prop:M}]

\noindent $1.$ 
By Lemma~\ref{lem:Nk}, the Laplace transform $\cL[N_k(t)](\lambda)$ is holomorphic in $\Re \lambda >0$ and well-defined up to the imaginary axis $\Re \lambda \ge 0$. In addition, using \eqref{bounds-Nkt}, there holds
\begin{equation}\label{bdLNk0}
|\cL[N_k](\lambda)| \le C_0 \int_0^\infty |k| \langle kt \rangle^{-2} \; dt \le C_0 ,
\end{equation}
uniformly in $k$ and $\Re \lambda \ge 0$. Next, integrating by parts exactly as we have proceeded for the electric dispersion function $D$, we obtain for $\Re \lambda >0$, 
\begin{equation}\label{expansion-ML}
\begin{aligned}
 \cL[N_k](\lambda  ) 
 &= \sum_{j = 0}^{m}\frac{1}{\lambda^{2j+2}} \partial_t^{2j+1}[N_k(t)](0)
 + \frac{1}{\lambda^{2m+2}}\int_0^\infty e^{- \lambda  t} \partial_t^{2m+2}N_k(t) \; dt ,
\end{aligned}\end{equation}
where
$$
\partial_t^{2j+1}N_k(0)= \frac{(-1)^j}{2} \int (k\cdot \hv)^{2j+2} |\mP_k\hv|^2\varphi'(\langle v\rangle)  dv.
$$
In deriving \eqref{expansion-ML}, we have used $\partial^{2j}_t[N_k(t)](0) =0$ for $j\ge 0$. Next, using again \eqref{bounds-Nkt}, the last integral term satisfies
\begin{equation}\label{bd-RM}
\Big|\int_0^\infty e^{-\lambda  t} \partial_t^{2m+2}N_k(t) \; dt \Big| \le C_{m} |k|^{2m+2}
\end{equation}
uniformly in $k$, for $m\ge 0$ and $\Re \lambda \ge 0$, where $C_m = C_0 m^{2}$. Finally, the fact that $\tau \mapsto M(i \tau , k)$ is  $\mathcal{C}^K$ on $\mathbb{R}$ for $K\leq (N_{0}-9)/2$ follows from Corollary~\ref{coro:LNK} and Lebesgue's domination theorem.

\bigskip

\noindent $2.$
  For $\lambda= i \tau$ and $|\tau|> |k|$, we recall from \eqref{rew-DMlambda} that 
$$
\begin{aligned}
M(i\tau,k) 
&= -\tau^2 + |k|^2 -  \frac12\int_{-1}^1 \frac{ q(u) }{1+ |k| u/\tau} \; du
\\
&= -\tau^2 + |k|^2 -  \frac12\int_{-1}^1 \frac{ q(u) }{1 - |k|^2 u^2/\tau^2} \; du,
\end{aligned}$$
in which the last identity used the fact that $q(u)$ is even in $u$. This yields \eqref{eq:formulaMtauleqk}. In addition, since $q(u)$ decays rapidly to zero as $u\to \pm1$, the function $\psi(y)$ is a  $  \mathcal{C}^{\lfloor N_0/2 \rfloor}$  function on $[0,1]$ with bounded derivatives.

\bigskip

\noindent $3.$ Finally, the representation \eqref{eq-formulaM} reads off from \eqref{rew-DMlambda}. Clearly, $\Lambda$ is holomorphic on $\mathbb{C} \setminus [-1,1]$.
This completes the proof of the proposition.

\end{proof}

\subsection{Electric dispersion relation}\label{sec-D}

In view of the spectral stability result of Proposition~\ref{prop-nogrowth} , we are led to study the dispersion relation $D(\lambda,k) =0$ for $\lambda$ lying on the imaginary axis. We obtain the following. 

\begin{theorem}\label{theo-LangmuirE} 
Let $\kappa(u), q(u)$ be defined as in \eqref{def-kqu}, and set 
\begin{equation}\label{def-taukappa0}
\index{k@$\kappa_0$} 
 \kappa_0^2:=\Omega(1)= - \int_{-1}^1 \frac{ u^2}{1 - u^2}  \kappa(u) \;du.
\end{equation}
Then, for any $0\le |k| \le \kappa_0$, there are exactly two zeros $\index{l@$\lambda_\pm^{\text{elec}}$: electric dispersion relation} \index{t@$\tau_*$}
\lambda_\pm^{\text{elec}}= \pm i \tau_{*}(|k|)$ of the electric dispersion relation $ D(\lambda,k) = 0$
that lie on the imaginary axis $\{\Re \lambda =0\}$, where $\tau_*$ is a $ \mathcal{C}^{\lfloor \frac{N_0 -2} {2}\rfloor}$  function, strictly increasing with respect to $|k|$, with $\tau_*(0) = \tau_0$ and $\tau_*(\kappa_0) = \kappa_0$. Moreover, we have  for $0 < |k|  <\kappa_{0}$ that 
\begin{equation}\label{lowerbound-taustarE} 
 \tau_0< \tau_*(|k|) <\kappa_0, \qquad  |k|< \tau_*(|k|)< \sqrt{\tau_0^2 + |k|^2} , 
\end{equation}
and for some constants $c_0, c_1,C_0>0$,
\begin{equation}\label{lowerbound-taustarDE} 
c_0 |k|\le    \tau'_*(|k|) \le C_0|k|, \qquad  c_1 \le  \tau''_*(|k|),
\end{equation}
for all $0\le |k|\le \kappa_0$.

Moreover, we have the following expansion for $|k|\ll 1$:
\begin{equation}\label{small-tauk}
\tau_*(|k|)  = \tau_0 + \frac{\tau_1^2}{2\tau_0^3}|k|^2 + \cO_{k\to 0}(|k|^4),
\end{equation}
where $\tau_1^2:=- \int_{-1}^1 u^4 \kappa(u) \, du >0$.

Finally, there exists $\delta_0>0$ small enough so that the zero curves $\lambda^{\text{elec}}_\pm(k)$ of  $ D(\lambda,k) = 0$ can be  extended as  $ \mathcal{C}^{\lfloor \frac{N_0 -9} {2}\rfloor}$  functions on $[\pm\kappa_0- \delta_0,\pm\kappa_0+ \delta_0]$, such that the following holds.
\begin{itemize}
\item We can write $\lambda^{\text{elec}}_\pm(k)= \Re \lambda_*(|k|) \pm i \Im \lambda_*(|k|)$  and $ \Re \lambda_*(|k|) < 0$, 
 $|\Im \lambda_*(|k|)| <|k|$  for all  $\kappa_0 < |k|\le \kappa_0+\delta_0$. 
 
\item The only zeros of $\widetilde{\widetilde{D}}(\lambda,k)$ (the extension of $D$ defined in Corollary~\ref{coro:extenD}) 
for $\kappa_{0} \leq |k| \leq \kappa_{0}+ \delta$ and $\lambda$ in a (complex) vicinity of $\pm i \kappa_{0} $ are exactly the $\lambda^{\text{elec}}_\pm(k)$.

\item  There are positive constants $c_0, C_0$ such that for $ \kappa_{0} \leq |k| \leq \kappa_{0}+ \delta$, we have  
\begin{equation}\label{lowerbound-taustarDEextend} 
c_0 |k|\le  \Im \lambda'_*(k) \le C_0|k|.
\end{equation}

\end{itemize}

\end{theorem}
\begin{remark}
\label{remarkholo}
Note that we have in particular  that the extensions $ \lambda^{\text{elec}}_\pm(k)$ that we obtain above  for $| k| \geq \kappa_{0}$
have the property that $|\Im \lambda_*(|k|)| <|k|$  for   $\kappa_0 < |k|\le \kappa_0+\delta_0$ so that even if they are originally
defined as zero curves of $\widetilde{\widetilde{D}}$ which was  a non unique smooth  extension of $D$ they lie in the domain 
where $\widetilde{\widetilde{D}}$ coincides with $\tilde D$ the holomorphic extension of $D$ (see Corollary \ref{coro:extenD}).
This yields that these curves which are the objects we are really interested in  actually do not depend on the choice of the smooth extension.
 \end{remark}

\begin{proof} We start with the first part of the statement.  According to Proposition~\ref{prop-nogrowth}, we can  restrict to the imaginary axis, that is we consider $\lambda = i\tau$ for $\tau \in \RR$. We distinguish between the two cases: $|\tau| \ge |k|$ and $|\tau| < |k|$. 

\bigskip

\noindent {\bf Case 1: $|\tau| \ge |k|$.}
In this case we apply the second item of Proposition~\ref{prop:D}, which gives that
$$
D(i\tau ,k) = 1 - \frac{1}{|k|^{2}} \Omega(|k|^2/\tau^2),
$$
where $\Omega$ is defined in~\eqref{eq:formulaDtauleqk}. From the sign of $\kappa$ in Lemma \ref{lem-rewDM}, we have that $\Omega'(y)  > 0$, and so $\Omega(y)$ is strictly increasing on $[0,1]$. That is, for any $0<y<1$, there holds
\begin{equation}\label{ineq-omegay}
0 =\Omega(0) < \Omega(y) < \Omega(1) = \kappa_0^2,
\end{equation}
where $\kappa_0$ is defined in~\eqref{def-taukappa0}.
As a result, by the strict monotonicity, the inverse map $\Omega^{-1}(\cdot)$ is well-defined from $[0,\kappa_0^2]$ to $[0,1]$. 

We are now ready to study the dispersion relation $D(i\tau,k) =0$, that corresponds  for $|\tau|\ge |k|$ to the equation
\begin{equation}\label{dispersion1}
\begin{aligned}
\Omega(|k|^2/\tau^2) =|k|^2.
\end{aligned}\end{equation}
 Clearly, there are no zeros of $D(i\tau,k)$ for $|k| > \kappa_0$ due to \eqref{ineq-omegay}. On the other hand, for $0<|k| \le \kappa_0$, $|k|^2$ belongs to the range of $\Omega(\cdot)$. 
 It thus follows that there are zeros $\tau = \pm \tau_*(|k|)$ of the form 
\begin{equation}\label{form-taustar} 
\tau_*(|k|) =\Big(\frac{|k|^2}{\Omega^{-1}(|k|^2)}\Big)^{1/2},
\end{equation}
for any $k\not =0$ such that $|k| \le \kappa_0$. Note that $\Omega^{-1}(|k|^2) >0$, since $k\not =0$, so $\tau_*(|k|)$ is a well-defined  function  for $k\not =0$.  From the second item of Proposition~\ref{prop:D}, we know that $\Omega$ is $\ \mathcal{C}^{\lfloor N_{0}/2-1\rfloor}$ on $[0,1]$ with bounded derivatives. Therefore, $\Omega^{-1}$ and thus $\tau_*$ enjoy the same regularity. By~\eqref{comp-tau0-2}, we have $\Omega'(0) = \tau^2$ and thus  $\Omega(0) =0$, $\Omega'(0) = \tau_0^2>0$, and $\Omega''(0) = 2\tau_1^2 >0$ imply the Taylor expansion
$$ \Omega^{-1}(|k|^2) = \frac{1}{\tau_0^2}|k|^2 - \frac{\tau_1^2}{\tau_0^6}|k|^4 + \cO(|k|^6)$$
for $|k|\ll1$. This yields 
\begin{equation}\label{small-taukproof}
\tau_*(|k|)  = \Big( \tau_0^2 + \frac{\tau_1^2}{\tau_0^2}|k|^2 + \cO(|k|^4)\Big)^{1/2} = \tau_0 + \frac{\tau_1^2}{2\tau_0^3}|k|^2 + \cO(|k|^4)
\end{equation}
as claimed for $|k|\ll1$. On the other hand, at $|k| = \kappa_0$, we have by \eqref{ineq-omegay} and \eqref{form-taustar}  that
 $\tau_*(\kappa_0) =  \kappa_0$. 

\bigskip

\noindent{\bf Case 2: $|\tau|< |k|$.}
We next consider the case when $\lambda = i\tau$ for $|\tau| < |k|$. From \eqref{rew-DMlambda}, 
for $\lambda = (\tilde\gamma + i\tilde\tau)|k|$, with $|\tilde \tau|<1$, we compute 
\begin{equation}\label{small-tau1}
\begin{aligned}
D(\lambda,k) 
& = 1  - \frac{1}{|k|^2} \int_{-1}^1 \frac{ u \kappa(u)}{-i\lambda/|k| +  u}\; du = 1 - \frac{1}{|k|^2}  \int_{-1}^1  \frac{u\kappa(u)}{-i\tilde\gamma + \tilde\tau +  u}  \; du,
\end{aligned}
\end{equation}
for $\kappa(u)$ defined as in \eqref{def-kqu}. 
Since $\varphi$ decays rapidly, $\kappa(u)$ is a well defined function on $[-1,1]$ and strictly negative for each $u\in (-1,1)$, see Lemma \ref{lem-rewDM}. Note in particular that $\kappa(u)\to 0$ fast enough  as $u\to \pm 1$. Using  Plemelj's formula we thus obtain
\begin{equation}\label{PV-Dk}
\lim_{\tilde\gamma\to 0^+}  \int_{-1}^1  \frac{u\kappa(u)}{-i\tilde\gamma + \tilde\tau +  u}  \; du =  P.V. \int_{-1}^1 \frac{u \kappa(u)}{u+\tilde\tau} du - i \pi \tilde \tau  \kappa(\tilde\tau)  
\end{equation}
for $|\tilde\tau|<1$. 
Therefore, for $\lambda = i\tilde\tau |k|$, with $|\tilde\tau|<1$, we have 
\begin{equation}\label{form-Ditau2}
D(i\tilde\tau|k|,k) = 1 - \frac{1}{|k|^2} P.V. \int_{-1}^1 \frac{u \kappa(u)}{u+\tilde\tau} du + \frac{i \pi }{|k|^2}  \tilde \tau  \kappa(\tilde\tau)  
.\end{equation}
Since $\kappa(u) \not =0$ for $u\in (-1,1)$, the imaginary part of $D(i\tilde\tau|k|,k)$ never vanishes for $0<|\tilde\tau|<1$, while at $\tilde\tau=0$, we have 
$$
D(0,k) = 1 -\frac{1}{|k|^2} \int_{-1}^1 \kappa(u)\; du.
$$
Since $\kappa(u) \le 0$, $D(0,k)\not =0$, and in fact $D(0,k) \gtrsim 1 + |k|^{-2}$. This and \eqref{form-Ditau2} imply that for any $\delta>0$, there is a positive constant $c_\delta$ so that
\begin{equation}\label{lowbd-Ditau2}
|D(i\tilde\tau|k|,k)| \ge c_\delta \Big(1 + \frac{1}{|k|^2}\Big), \qquad \forall |\tilde \tau|\le 1-\delta,
\end{equation}
for any $k\not=0$. 

\bigskip

\noindent {\bf Dispersive properties of $\tau_*(|k|)$}

Let us now study the property of $\tau_*(|k|)$ for any $0\le |k|\le \kappa_0$. Set $\index{x@$x_*$} x_*(|k|) = \tau_*(|k|)^2$. Recalling \eqref{eq:formulaDtauleqk}, it turns out to be convenient to write 
\begin{equation}\label{def-smallomega} 
\index{o@$\omega$}
\Omega(y) = y \omega(y), \qquad \omega(y) := -\int_{-1}^1 \frac{u^2\kappa(u)}{1- y u^2}\;du.
\end{equation}
Hence, $x_*(|k|)$ is defined through  
\begin{equation}\label{eqs-xstarE}
x_* = \omega(|k|^2/x_*) .\end{equation}
Note that $\omega(0) = \tau_0^2$ and $\omega(1) = \kappa_0^2$, using \eqref{ineq-omegay}. Hence, by monotonicity, we have
\begin{equation}\label{x-bounds} \tau_0^2 \le x_*(|k|) \le \kappa_0^2.\end{equation}
Now, taking the derivative of the equation $x_* = \omega(|k|^2/x_*)$, we get
\begin{equation}\label{cal-D2x}
\begin{aligned}
\Big[1+  \frac{|k|^2}{x_*^2}  \omega'(\frac{|k|^2}{x_*})\Big] x_*'(|k|) &=  \frac{ 2|k|}{x_*} \omega'(\frac{|k|^2}{x_*}) ,
\\
\Big[1+  \frac{|k|^2}{x_*^2}  \omega'(\frac{|k|^2}{x_*}) \Big] x_*''(|k|) 
&=  \frac{2(x_* - |k| x'_*)^2}{x_*^3} \omega'(\frac{|k|^2}{x_*})  
+ \frac{|k|^2(2x_*- |k|x'_*)^2}{x_*^4} \omega''(\frac{|k|^2}{x_*})  ,
\end{aligned}
\end{equation}
noting that each term on the right-hand side is nonnegative, since $\omega'(y) \ge 0$ and $\omega''(y)\ge 0$. The above yields that both $x_*'(|k|) $ and $x_*''(|k|)$ never vanish for $0<|k|\le \kappa_0$ and satisfy the asymptotic expansion \eqref{small-tauk} for small $k$. Therefore, recalling \eqref{x-bounds} and the fact that 
\begin{equation}\label{bd-coff1}
C_0^{-1} \le 1 + \frac{|k|^2}{x_*^2}  \omega'(\frac{|k|^2}{x_*}) \le C_0,
\end{equation}
we obtain 
\begin{equation}\label{Dx-bounds}
C_0^{-1} |k| \le x_*'(|k|) \le C_0 |k| , \qquad C_0^{-1}\le x_*''(|k|)  \le C_0 , \qquad \forall ~0\le |k|\le \kappa_0,\end{equation}
for some positive constant $C_0$.

Similarly, we compute 
$$ \tau_*'(|k|) = \frac{x_*'(|k|)}{2\sqrt{x_*(|k|)}} , \qquad \tau_*''(|k|) = \frac{2 x''_*(|k|) x_*(|k|) - x_*'(|k|)^2}{4 x_*(|k|)^{3/2}}.$$ 
The estimates on $\tau_*'(|k|)$ follow at once from those on $x_*(|k|), x_*'(|k|)$, see \eqref{x-bounds} and \eqref{Dx-bounds}. We shall now prove that 
\begin{equation}\label{lowbd-tauk}
\tau_*''(|k|) \ge c_0, \qquad ~\forall~0\le |k|\le \kappa_0,
\end{equation}
for some positive constant $c_0$. In view of \eqref{x-bounds}, it suffices to obtain a lower bound for $2x_*''(|k|) x_*(|k|) - x_*'(|k|)^2 $. Using \eqref{cal-D2x}, we compute 
$$
\begin{aligned} 
\Big[ 1 + \frac{|k|^2}{x_*^2}  \omega'(|k|^2/x_*) \Big]x_*'(|k|)^2 &= \frac{ 2|k| x'_*}{x_*} \omega'(|k|^2/x_*) ,
\\
2\Big[ 1 + \frac{|k|^2}{x_*^2}  \omega'(\frac{|k|^2}{x_*}) \Big]x_*''(|k|) x_*(|k|) 
&=  \frac{4(x_* - |k| x'_*)^2}{x_*^2} \omega'(\frac{|k|^2}{x_*})  
+ \frac{2|k|^2(2x_*- |k|x'_*)^2}{x_*^3} \omega''(\frac{|k|^2}{x_*}) ,
 \end{aligned}$$
and so 
$$
\begin{aligned}
\Big[ &1 + \frac{|k|^2}{x_*^2}  \omega'(\frac{|k|^2}{x_*}) \Big] \Big (2x_*''(|k|) x_*(|k|) - x_*'(|k|)^2 \Big) 
\\& = \frac{2(2x_* - |k| x'_*)}{x_*^2} 
 \Big[ (x_* - 2|k| x'_*) \omega'(\frac{|k|^2}{x_*})  
+ (2x_*- |k|x'_*)\frac{|k|^2}{x_*} \omega''(\frac{|k|^2}{x_*}) \Big] 
\end{aligned}$$
in which recalling \eqref{bd-coff1}, the factor $1 + \frac{|k|^2}{x_*^2}  \omega'(\frac{|k|^2}{x_*}) $ is harmless.  
Using again \eqref{cal-D2x}, we note that 
$$
\begin{aligned}
2 x_* - |k|x'_* = 2 x_* - 2|k|^2 \frac{ \frac{1}{ x_*} \omega'(|k|^2/x_*) }{1 + \frac{|k|^2}{x_*^2}  \omega'(|k|^2/x_*) } =  \frac{ 2 x_* }{1 + \frac{|k|^2}{x_*^2}  \omega'(|k|^2/x_*) } ,
\\
x_* - 2 |k|x'_* = x_* - 4|k|^2 \frac{ \frac{1}{ x_*} \omega'(|k|^2/x_*) }{1 + \frac{|k|^2}{x_*^2}  \omega'(|k|^2/x_*) } =  \frac{  x_* - \frac{3|k|^2}{ x_*} \omega'(|k|^2/x_*) }{1 + \frac{|k|^2}{x_*^2}  \omega'(|k|^2/x_*) },
\end{aligned}
$$
which in particular yields that $2 x_* - |k|x'_* \gtrsim 1$ on $[0,\kappa_0]$, recalling  \eqref{x-bounds}. 
Therefore, 
$$
\begin{aligned}
\Big[ &1 + \frac{|k|^2}{x_*^2}  \omega'(\frac{|k|^2}{x_*}) \Big] \Big[ (x_* - 2|k| x'_*) \omega'(\frac{|k|^2}{x_*})  
+ (2x_*- |k|x'_*)\frac{|k|^2}{x_*} \omega''(\frac{|k|^2}{x_*}) \Big] 
\\
&= x_* \omega'(\frac{|k|^2}{x_*}) - \frac{3|k|^2}{x_*}[\omega'(\frac{|k|^2}{x_*})]^2 + 2|k|^2 \omega''(\frac{|k|^2}{x_*}) ,
\end{aligned}$$
and so $$
\begin{aligned}
\Big( &1 + \frac{|k|^2}{x_*^2}  \omega'(\frac{|k|^2}{x_*}) \Big)^2 \Big (2x_*''(|k|) x_*(|k|) - x_*'(|k|)^2 \Big) 
\\& = \frac{2(2x_* - |k| x'_*)}{x_*^2} 
 \Big[  x_* \omega'(\frac{|k|^2}{x_*}) - \frac{3|k|^2}{x_*}[\omega'(\frac{|k|^2}{x_*})]^2 + 2|k|^2 \omega''(\frac{|k|^2}{x_*}) \Big]. 
\end{aligned}$$
Recalling \eqref{bd-coff1} and the fact that $2 x_* - |k|x'_* \gtrsim 1$, it suffices to study the terms in the bracket. Let $y_* = |k|^2/x_*$.  Recalling that $x_* = \omega(y_*)$, we consider 
$$A_* = \omega(y_*) \omega'(y_*) - 3y_* \omega'(y_*)^2 + 2  y_* \omega(y_*)\omega''(y_*) .$$
We have  
$$
\omega'(y) =  -\int_{-1}^1 \frac{u^4\kappa(u)}{(1- y u^2)^2}\;du 
, \qquad \omega''(y) = -2 \int_{-1}^1 \frac{u^6\kappa(u)}{(1- y u^2)^3}\;du .
$$
By the Cauchy-Schwarz inequality, we have 
$$
\int_{-1}^1 \frac{u^4|\kappa(u)|}{(1- y u^2)^2}\;du  \le \Big(  \int_{-1}^1 \frac{u^2|\kappa(u)|}{1- y u^2}\;du \Big)^{1/2} \Big( \int_{-1}^1 \frac{u^6|\kappa(u)|}{(1- y u^2)^3}\;du \Big)^{1/2}.
$$
That is, $\omega'(y)^2 \le \frac12 \omega(y) \omega''(y)$. This yields 
$$A_* = \omega(y_*) \omega'(y_*) - 3y_* \omega'(y_*)^2 + 2  y_* \omega(y_*)\omega''(y_*) \ge  \omega(y_*) \omega'(y_*) \ge  \omega(0) \omega'(0) >0,$$
in which we used the monotonicity of $\omega(y)$ and $\omega'(y)$. Hence, $2x_*''(|k|) x_*(|k|) - x_*'(|k|)^2  \gtrsim 1$, and \eqref{lowbd-tauk} follows.

\bigskip

\noindent {\bf Extension and dispersion relation around $|k|=\kappa_0$.}

We use the third item of Proposition~\ref{prop:D} and consider the  extension $\widetilde{\widetilde{\Phi}}$ of $\Phi$.
By~\eqref{eq:diffpm1}, we recall that
$$\mathrm{d} \widetilde{\widetilde{\Phi}}(\pm 1)   = \left(\int_{-1}^1 \frac{u\kappa(u)}{(u\pm1)^2} \, du\right)\operatorname{I}  \neq 0.
$$
We can  therefore apply the implicit function theorem and get $U$ a neighborhood of $\pm 1$ and a $ \mathcal{C}^{\lfloor N_{0}/2 -3\rfloor}$ curve $\widetilde{\widetilde{z}}_\pm(r)$ such that, on $\{ |r-\kappa_0|< \eps\} \times  U$, the unique zeros of the function $(r, z) \mapsto r^2- \widetilde{\widetilde{\Phi}}(z)$ are described by $r \mapsto \widetilde{\widetilde{z}}_\pm (r)$, with $\widetilde{\widetilde{z}}_\pm (\kappa_0)=\pm1$. Setting $\widehat{z}_\pm(r) = i r \widetilde{\widetilde{z}}(r)$, we thus obtain, denoting $|k|=r$, that the unique zeros of the function 
$$\widetilde{\widetilde{D}}(z,k) \equiv \widetilde{\widetilde{D}}(z,r) := 1 - \frac{1}{|k|^2}  \widetilde{\widetilde{\Phi}}(-iz /|k|)$$ 
on a small neighborhood of $(\kappa_0, \pm i \kappa_0)$
are described by the curve $r \mapsto \widehat{z}_\pm (r)$.
We already know that
$
\widehat{z}_\pm(\kappa_0) = \pm i \kappa_0.
$
There remains to check that the zeros for $\kappa_0<r<\kappa_0+\delta$ belong to the region $\{ z \in \mathbb{C}, |\Im z| < r\}$, or in other words that $$|\Re  \widetilde{\widetilde{z}}_\pm(r)| <1.
$$ Let us for instance focus on $z_+$, the case of $z_-$ being symmetric. To this end we differentiate the identity $0=r^2- \widetilde{\widetilde{\Phi}}(\widetilde{\widetilde{z}}_+(r))$ 
which yields the formula
$$
\widetilde{\widetilde{z}}_+'(\kappa_0)=(\mathrm{d} \widetilde{\widetilde\Phi}(1))^{-1}{2 \kappa_0}.
$$
Therefore, $\widetilde{\widetilde{z}}_+'(\kappa_0)$ is real with  $ \widetilde{\widetilde{z}}_+'(\kappa_0)< 0$; since $\Re \widetilde{\widetilde{z}}_+(\kappa_0)=1$, we deduce our claim. Likewise, $\widetilde{\widetilde z}_-'(\kappa_0)> 0$.
By Lemma~\ref{resolvent} and the analysis in the case $|\tau| <|k|$, for $|k|>r_0$, $\widetilde{z}_\pm(\kappa_0)$  must satisfy $\Re \widetilde{z}_\pm(\kappa_0)<0$. 

Now, taking $\delta>0$ small enough, we define for all $|k|<\kappa_0+ \delta$,
\begin{equation}\label{eq-defz}
\lambda_\pm(k)
= \left \{ \begin{aligned} \pm i  \tau_*(|k|) \quad& \quad \mbox{if}\quad 0\leq |k|\leq \kappa_0,\\
\widehat{z}_\pm(|k|) \quad& \quad \mbox{if}\quad \kappa_0<  |k|< \kappa_0 +\delta.
  \end{aligned}\right.
\end{equation}
Clearly $\lambda_\pm$ is a $ \mathcal{C}^{{\lfloor \frac{N_0-9}{2} \rfloor}}$ function on $\{0\leq |k|< \kappa_0\}$ and on $\{\kappa_0<  |k|< \kappa_0 +\delta\}$. To prove that $\lambda_\pm$ is globally $ \mathcal{C}^{{\lfloor \frac{N_0-9}{2} \rfloor}}$, since it is radial, it suffices to check that the derivatives coincide at $|k|=\kappa_0$. This is the case by construction of $\pm i  \tau_*(|k|)$ and $\widetilde{z}_\pm(|k|)$, since $\tau \mapsto D(i\tau,k)$ is $ \mathcal{C}^{\lfloor \frac{N_0-9}{2} \rfloor}$.
Moreover, still by continuity, imposing $\delta$ small enough, the estimates~\eqref{Dx-bounds} still hold on $B(0, \kappa_0 + \delta)$.

This concludes the description of the zeros in a small neighborhood of $|k|= \kappa_0$ 
and the proof of Theorem \ref{theo-LangmuirE} is finally complete. 
\end{proof}

\begin{rem} In addition, we can observe  that 
\begin{equation}\label{vanishingImz}  
|\Re \lambda_\pm(k)|\lesssim |k-\kappa_0|^{K}, \quad K \leq (N_{0}-9)/2
\end{equation}
for $|k|$ sufficiently close to $\kappa_0$, where $N_0$ is given as in \eqref{decay-mu}. Namely, the faster the stationary state  $\varphi$ decays, the slower the damping rate is. 
In view of \eqref{laderiveedordreldephi} and  \eqref{annulkappa}, $\widetilde\Phi^{(n)}(\pm1)$ is real valued for all $0\le n \leq N_0/2-2$. It thus follows by induction, together with a use of the Fa\`a di Bruno formula, that  the $\partial_k^\alpha \widetilde{z}_\pm(\kappa_0)$ are real valued for all $ 0\le |\alpha| \leq K$. This proves the claim \eqref{vanishingImz} for $|k|$ sufficiently close to $\kappa_0$ from a Taylor expansion. 
\end{rem}

\subsection{Magnetic dispersion relation}
\label{sec-M}

Again thanks to the stability result of Proposition~\ref{prop-nogrowth}, we are led to study the magnetic dispersion relation $M(\lambda,k)=0$ for $\lambda$ lying on the imaginary axis. We obtain the following result. 

\begin{theorem}\label{theo-LangmuirB} 
For each $k\in \RR^3$, there are exactly two zeros $\index{l@$\lambda_\pm^{\text{mag}}$: magnetic dispersion relation} \index{n@$\nu_{*}$}  \lambda_\pm^{\text{mag}} = \pm i \nu_*(|k|)$ of the magnetic dispersion relation 
$$ M(\lambda,k) = 0$$
that lie on the imaginary axis $\{\Re \lambda =0\}$. Moreover, we have $\nu_*(|k|)>|k|$.
Finally, $\nu_*$ is a $  \mathcal{C}^{\lfloor{N_{0}\over 2} \rfloor}$ function   satisfying the following Klein-Gordon type estimates:  
\begin{equation}\label{KG-behave1} 
c_0 \sqrt{1+|k|^2} \le \nu_*(|k|)\le C_0 \sqrt{1+|k|^2} ,
\end{equation}
\begin{equation}\label{KG-behave2}
c_0\frac{|k|}{ \sqrt{1+|k|^2}} \le \nu_*'(|k|) \le C_0 \frac{|k|}{ \sqrt{1+|k|^2}} , 
\end{equation}
\begin{equation}\label{KG-behave3}
 c_0 (1+|k|^2)^{-3/2}  \le \nu_*''(|k|)  \le C_0.
\end{equation}
uniformly with respect to $k\in \RR^3$, for some positive constants $c_0, C_0$.

\end{theorem}

\begin{proof}

 According to Proposition~\ref{prop-nogrowth}, we can  restrict to the imaginary axis, that is we consider $\lambda = i\tau$ for $\tau \in \RR$. We distinguish between the two cases: $|\tau| < |k|$ and $|\tau| \ge |k|$.

\bigskip

\noindent {\bf Case 1: $|\tau|\ge |k|$.}
We first consider the case when $|\tau| \ge |k|$. 
Using the formulation \eqref{eq:formulaMtauleqk}, we may write 
$$M(i\tau,k ) = \Psi(\tau^2)$$
where 
$$ \Psi(x) := - x + |k|^2 + \psi(|k|^2/x),$$
which is well-defined for $x>|k|^2$. In addition, recalling $\psi(y) = - \frac12 \int_{-1}^1 \frac{q(u)}{1- y u^2}\;du$ and the fact that $q(u)\le 0$, we compute 
\begin{equation}
\label{laderiveedePsi} \Psi'(x) = -1  - \frac{|k|^2}{x^2}\psi'(\frac{|k|^2}{x}) \le -1.
\end{equation}
Next, since $\psi(|k|^2/x)\le C_0$, the function $\Psi(x)$ is strictly negative for $x\gg1$. On the other hand, we observe that $\Psi(|k|^2) = \psi(1) = \tau_0^2>0$, see \eqref{comp-tau0}. By the strict monotonicity of $\Psi(x)$, there is a unique solution $x_* > |k|^2$ of $\Psi(x) =0$, or equivalently,
\begin{equation}\label{eqs-xstar} x_* = |k|^2 + \psi(|k|^2/x_*) ,\end{equation}
yielding the existence of purely imaginary modes $\lambda_\pm = \pm i \sqrt{x_*(|k|)}$ in the region where $|\tau|> |k|$.  The smoothness of $x_*(|k|)$ follows from that of $\psi$, as described in the second item of Proposition~\ref{prop:M}
and \eqref{laderiveedePsi}.

\bigskip

\noindent {\bf Case 2: $|\tau| < |k|$.}
We next consider the case when $\lambda = i\tau$ for $|\tau| < |k|$. From \eqref{rew-DMlambda}, 
for $\lambda = (\tilde\gamma + i\tilde\tau)|k|$, with $|\tilde \tau|<1$, we compute 
\begin{equation}\label{new-Msmall}
M(\lambda,k) =   |k|^2(1 + \tilde\gamma^2-\tilde\tau^2+2i \tilde \gamma \tilde \tau) + \frac{i(\tilde\gamma + i\tilde\tau)}{2}  \int_{-1}^1 \frac{ q(u)}{-i\tilde \gamma + \tilde\tau+ u} 
\; du .
\end{equation}
Therefore, by Plemelj's formula,  for $\tilde \gamma \to 0^+$ and $|\tilde\tau|<1$, we have 
\begin{equation}\label{comp-Mstau}
M(i\tilde\tau|k|,k) =  |k|^2(1-\tilde\tau^2) - \tilde\tau  P.V. \int_{-1}^1 \frac{1}{\tilde\tau + u} q(u) du  - i\pi  \tilde\tau q(\tilde\tau) 
.\end{equation}
Since $q(u) \not =0$ for $u\in (-1,1)$, the imaginary part of $M(i\tilde\tau|k|,k)$ never vanishes for $0<|\tilde\tau|<1$, while at $\tilde\tau=0$ the real part is equal to $|k|^2$. On the other hand, at $\tilde\tau = \pm 1$, we have 
$$
M( \pm i|k|,k) = - \int_{-1}^1 \frac{q(u)}{1 \pm u} \; du = - \int_{-1}^1 \frac{q(u)}{1- u^2}  \; du , $$
recalling that $q(u)$ is negative and even in $u$. 
Therefore, $M( \lambda ,k) \not =0$ for $\lambda = i\tilde\tau |k|$ with $|\tilde\tau|\le 1$. In particular, we have 
\begin{equation}\label{lower-Mstau0} 
|M(i\tilde\tau|k|,k) | \ge \theta_0 |\tilde\tau|, \qquad \forall ~|\tilde\tau|\le 1,
\end{equation}
for some positive $\theta_0$, which in particular proves that there are no zeros in this region when $|\tau| \le |k|$.
In fact, using the fact that the P.V. integral $ \int_{-1}^1 \frac{1}{\tilde\tau + u} q(u) du$ is finite, we obtain from \eqref{comp-Mstau} that 
$$ 
|M(i\tilde\tau|k|,k) | \ge  |k|^2(1-\tilde\tau^2)  - C_0 |\tilde \tau|.
$$
Combining this with \eqref{lower-Mstau0}, we get  
$$
\begin{aligned}
|M(i\tilde\tau|k|,k) | &= (1- \delta) |M(i\tilde\tau|k|,k) | + \delta |M(i\tilde\tau|k|,k) | 
\\& 
\ge (1- \delta) \theta_0 |\tilde\tau| +  \delta |k|^2(1-\tilde\tau^2) - C_0 \delta |\tilde \tau|,.
\end{aligned}
$$
Taking $\delta$ sufficiently small, we get 
\begin{equation}\label{lower-Mstau} 
\begin{aligned}
|M(i\tilde\tau|k|,k) | 
\gtrsim |\tilde\tau| +  |k|^2(1-\tilde\tau^2), \qquad \forall ~|\tilde\tau|\le 1.
\end{aligned}\end{equation}

\bigskip

\noindent {\bf Dispersive properties of $\nu_*(|k|)$.}

To conclude, we have shown that $\lambda_\pm^{\text{mag}}= \pm i \nu_*(|k|) $ are the only purely imaginary zeros of the symbol $M( \lambda ,k) $ for each $k\in \RR^3$. It remains to give the desired estimates on the dispersion relation $\nu_*(|k|) = \sqrt{x_*(|k|)}$. 
Let us start with some estimates on $x_* = x_*(|k|)$, as a function of $|k|$, that is defined through the relation \eqref{eqs-xstar}, namely
$x_* = |k|^2 + \psi({|k|^2}/{x_*}) ,$
thanks to the second item of Proposition~\ref{prop:M}. Recall that 
\begin{equation}\label{redef-psiy}
\psi(y) = - \frac12 \int_{-1}^1 \frac{q(u)}{1- y u^2}\;du
\end{equation}
where $q(u)\le 0$, with $\psi(0) =\tau_0^2$ and $\psi(1) = q_0^2$ being well-defined. By monotonicity, we have $\tau_0^2 \le \psi(y)\le q_0^2$ for $y\in [0,1]$. It thus follows from the equation $x_* = |k|^2 + \psi({|k|^2}/{x_*}) $ that 
\begin{equation}\label{rangex} 
|k|^2 + \tau_0^2 \le x_* (|k|)\le  |k|^2 + q_0^2,\end{equation}
for all $k \in \RR^3$. Next, we claim that there is a universal constant $C_0$ such that 
\begin{equation}\label{bounds-DxstarB}
C_0^{-1} |k| \le x_*'(|k|) \le C_0 |k| , \qquad C_0^{-1}\le x_*''(|k|)  \le C_0 \end{equation}
for all $k\in \RR^3$. Indeed, we compute
\begin{equation}\label{cal-D2xB}
\begin{aligned}
 \Big[ 1 + \frac{|k|^2}{x_*^2}  \psi'(|k|^2/x_*) \Big]x_*'(|k|) &= 2|k| + \frac{ 2|k|}{x_*} \psi'(|k|^2/x_*) ,
 \\
  \Big[ 1 + \frac{|k|^2}{x_*^2}  \psi'(\frac{|k|^2}{x_*}) \Big]x_*''(|k|) 
&= 2 + \frac{2(x_* - |k| x'_*)^2}{x_*^3} \psi'(\frac{|k|^2}{x_*})  
+ \frac{|k|^2(2x_*- |k|x'_*)^2}{x_*^4} \psi''(\frac{|k|^2}{x_*})  ,
\end{aligned}\end{equation}
noting that each term on the right-hand side is nonnegative, since $\psi'(y)$ and $\psi''(y)$ are nonnegative. In fact, we have
\begin{equation}\label{cal-Dpsiy}
\psi'(y) = -\frac12 \int_{-1}^1 \frac{u^2q(u)}{(1- y u^2)^2}\;du 
, \qquad \psi''(y) =  -\int_{-1}^1 \frac{u^4 q(u)}{(1- y u^2)^3}\;du ,
\end{equation}
which imply that $0<\psi'(0) \le \psi'(y) \le \psi'(1)$ and $0<\psi''(0) \le \psi''(y) \le \psi''(1)$ for $y\in [0,1]$. In particular, together with \eqref{rangex}, we have
\begin{equation}\label{temp-bdc0} 1\le 1 + \frac{|k|^2}{x_*^2}  \psi'(|k|^2/x_*) \le C_0,\end{equation}
which holds uniformly for all $k\in \RR^3$. Therefore, 
$$x_*'(|k|) = 2|k| \frac{ 1+\frac{1}{ x_*} \psi'(|k|^2/x_*) }{1 + \frac{|k|^2}{x_*^2}  \psi'(|k|^2/x_*) } \quad \approx \quad |k|,$$
as claimed in \eqref{bounds-DxstarB}. Similarly, since $\psi'(y)$ and $\psi''(y)$ are nonnegative, we have 
$$
\Big[ 1 + \frac{|k|^2}{x_*^2}  \psi'(\frac{|k|^2}{x_*}) \Big]x_*''(|k|)  \ge 2$$
which gives the desired lower bound for $x_*''(|k|)$, upon using \eqref{temp-bdc0}. On the other hand, using \eqref{rangex}, we bound 
$$
\begin{aligned}
\Big[ 1 + \frac{|k|^2}{x_*^2}  \psi'(\frac{|k|^2}{x_*}) \Big]x_*''(|k|) 
&= 2 + \frac{2(x_* - |k| x'_*)^2}{x_*^3} \psi'(\frac{|k|^2}{x_*})  
+  \frac{|k|^2(2x_*- |k|x'_*)^2}{x_*^4}  \psi''(\frac{|k|^2}{x_*})  
\\&\le  2 + \frac{C_0}{x_*} \psi'(\frac{|k|^2}{x_*})  
+ \frac{C_0}{x_*}  \psi''(\frac{|k|^2}{x_*})  
\\& \lesssim 1,
\end{aligned}$$
which gives \eqref{bounds-DxstarB}. 

Finally, we study the properties of the dispersion relation $\nu_*(|k|) = \sqrt{x_*(|k|)}$.  We compute 
$$ \nu_*'(|k|) = \frac{x_*'(|k|)}{2\sqrt{x_*(|k|)}} , \qquad \nu_*''(|k|) = \frac{2 x''_*(|k|) x_*(|k|) - x_*'(|k|)^2}{4 x_*(|k|)^{3/2}}.$$ 
The estimates on $\nu_*'(|k|)$ follow at once from those on $x_*(|k|), x_*'(|k|)$, see \eqref{rangex} and \eqref{bounds-DxstarB}. We now study $\nu_*''(|k|)$. In view of \eqref{rangex}, it suffices to prove that 
\begin{equation}\label{convex-nuxk}
2 x''_*(|k|) x_*(|k|) - x_*'(|k|)^2 \gtrsim 1,
\end{equation}
for any $k\in \RR^3$. Using \eqref{cal-D2xB}, we compute 
$$
\begin{aligned} 
\Big[ 1 + \frac{|k|^2}{x_*^2}  \psi'(|k|^2/x_*) \Big]x_*'(|k|)^2 &= 2|k|x'_* + \frac{ 2|k| x'_*}{x_*} \psi'(|k|^2/x_*)
\\
2\Big[ 1 + \frac{|k|^2}{x_*^2}  \psi'(\frac{|k|^2}{x_*}) \Big]x_*''(|k|) x_*(|k|) 
&= 4 x_* + \frac{4(x_* - |k| x'_*)^2}{x_*^2} \psi'(\frac{|k|^2}{x_*})  
+ \frac{2|k|^2(2x_*- |k|x'_*)^2}{x_*^3} \psi''(\frac{|k|^2}{x_*}) 
 \end{aligned}$$
and so 
$$
\begin{aligned}
\Big[ &1 + \frac{|k|^2}{x_*^2}  \psi'(\frac{|k|^2}{x_*}) \Big] \Big (2x_*''(|k|) x_*(|k|) - x_*'(|k|)^2 \Big) 
\\& =  \frac{2(2x_* - |k| x'_*)}{x_*^2} 
 \Big[x_*^2 + (x_* - 2|k| x'_*) \psi'(\frac{|k|^2}{x_*})  
+ (2x_*- |k|x'_*)\frac{|k|^2}{x_*} \psi''(\frac{|k|^2}{x_*}) \Big] 
\end{aligned}$$
in which recalling \eqref{temp-bdc0}, the factor $1 + \frac{|k|^2}{x_*^2}  \psi'(\frac{|k|^2}{x_*}) $ is harmless. Set 
$$ A_* = x_*^2 + (x_* - 2|k| x'_*) \psi'(\frac{|k|^2}{x_*})  
+ (2x_*- |k|x'_*)\frac{|k|^2}{x_*} \psi''(\frac{|k|^2}{x_*}) .$$ 
We shall prove that  
\begin{equation}\label{good-DxstarB}
2 x_* - |k|x'_* \gtrsim 1 ,\qquad A_* \gtrsim (1+|k|^2)^2.
\end{equation}
This would yield $2x_*''(|k|) x_*(|k|) - x_*'(|k|)^2 \gtrsim 1$, and so $\nu_*''(|k|) \gtrsim 1/x_*^{3/2} \gtrsim \langle k\rangle^{-3}$ as claimed. Indeed, using again \eqref{cal-D2xB}, we note that 
$$
\begin{aligned}
2 x_* - |k|x'_* &= 2 x_* - 2|k|^2 \frac{ 1+\frac{1}{ x_*} \psi'(|k|^2/x_*) }{1 + \frac{|k|^2}{x_*^2}  \psi'(|k|^2/x_*) } =  \frac{ 2 (x_* -|k|^2)}{1 + \frac{|k|^2}{x_*^2}  \psi'(|k|^2/x_*) },
\\
x_* - 2 |k|x'_* &= x_* - 4|k|^2 \frac{ 1+\frac{1}{ x_*} \psi'(|k|^2/x_*) }{1 + \frac{|k|^2}{x_*^2}  \psi'(|k|^2/x_*) } =  \frac{  x_* -4|k|^2 - \frac{3|k|^2}{ x_*} \psi'(|k|^2/x_*) }{1 + \frac{|k|^2}{x_*^2}  \psi'(|k|^2/x_*) },
\end{aligned}
$$
which in particular yields $2 x_* - |k|x'_* \gtrsim x_* - |k|^2 \gtrsim 1$ as stated in \eqref{good-DxstarB}, 
upon recalling  \eqref{rangex} and \eqref{temp-bdc0}. On the other hand, using the above identities, we have 
$$
\begin{aligned}
\Big[ 1 + \frac{|k|^2}{x_*^2}  \psi'(\frac{|k|^2}{x_*}) \Big] A_* 
&=x_*^2 + (x_* - 3 |k|^2)\psi'(\frac{|k|^2}{x_*}) - \frac{3|k|^2}{x_*}[\psi'(\frac{|k|^2}{x_*})]^2 + 2(x_*-|k|^2)\frac{|k|^2}{x_*} \psi''(\frac{|k|^2}{x_*}) .
\end{aligned}$$
Let $y_* = |k|^2/x_*$.  Recalling $x_* = |k|^2 + \psi(y_*)$, we thus have
$$\Big[ 1 + \frac{|k|^2}{x_*^2}  \psi'(\frac{|k|^2}{x_*}) \Big] A_* = x_*^2  + \psi(y_*)\psi'(y_*) - 2|k|^2 \psi'(y_*) - 3y_* \psi'(y_*)^2 + 2 y_* \psi(y_*)\psi''(y_*) .$$
First, recalling \eqref{redef-psiy} and \eqref{cal-Dpsiy}, and using the H\"older's inequality, we have 
$$
\int_{-1}^1 \frac{u^2|q(u)|}{(1- y u^2)^2}\;du  \le \Big(  \int_{-1}^1 \frac{|q(u)|}{1- y u^2}\;du \Big)^{1/2} \Big( \int_{-1}^1 \frac{u^4|q(u)|}{(1- y u^2)^3}\;du \Big)^{1/2}.
$$
That is, $\psi'(y)^2 \le \frac12 \psi(y) \psi''(y)$. This yields 
$$\begin{aligned}
\Big[ 1 + \frac{|k|^2}{x_*^2}  \psi'(\frac{|k|^2}{x_*}) \Big] A_* 
\ge  x_*^2  + \psi(y_*)\psi'(y_*) - 2x_* y_* \psi'(y_*) + y_* \psi'(y_*)^2 ,
\end{aligned}$$
in which we have used $|k|^2 = x_* y_*$ to get the third term on the right hand side. Finally, completing the square to treat the only negative term in the above expression, we get  
$$\begin{aligned}
\Big[ 1 + \frac{|k|^2}{x_*^2}  \psi'(\frac{|k|^2}{x_*}) \Big] A_* 
\ge  x_*^2 (1-y_*) + y_* (x_* - \psi'(y_*))^2 + \psi(y_*)\psi'(y_*).
\end{aligned}$$
Note again that the factor $1 + \frac{|k|^2}{x_*^2}  \psi'(\frac{|k|^2}{x_*}) $ is harmless, thanks to \eqref{temp-bdc0}. Now, since $y_* \in [0,1]$, $\psi(y_*) \ge \tau_0^2$, and $\psi'(y_*)\ge \psi'(0)>0$, we get $A_*\gtrsim 1$, which proves the lower bound \eqref{good-DxstarB} on $A_*$ for bounded $k$. In the case when $|k|\gg1$, we note that $y_* \ge 1/2$ and $\psi'(y_*) \le \psi'(1)$, while $x_* \ge \tau_0^2 + |k|^2$. The second term in the above expression gives the lower bound \eqref{good-DxstarB} as claimed.    

This ends the proof of Theorem \ref{theo-LangmuirB}.
\end{proof}


\section{Electric and magnetic Green functions}
\label{sec:Green}

In view of the resolvent equations \eqref{resolvent} for the electric and magnetic potentials $\phi, A$, we introduce the resolvent kernels 
\begin{equation}\label{def-THk}
\TG_{k}(\lambda) :=  \frac{1}{D(\lambda,k)} , \qquad \TH_{k}(\lambda) :=  \frac{1}{M(\lambda,k)} ,
\end{equation}
and the corresponding temporal Green functions 
\begin{equation}\label{def-FHk}
\begin{aligned}
\FG_{k}(t) &=  \frac{1}{2\pi i}\int_{\{\Re \lambda = \gamma_0\}}e^{\lambda t} \TG_{k}(\lambda)\; d\lambda,
\\
\FH_{k}(t) &=  \frac{1}{2\pi i}\int_{\{\Re \lambda = \gamma_0\}}e^{\lambda t} \TH_{k}(\lambda)\; d\lambda, 
\end{aligned}\end{equation}
which are well-defined as oscillatory integrals  for $\gamma_0 >0$, recalling the resolvent kernels are holomorphic in $\{\Re \lambda>0\}$, see Theorem \ref{theo-LangmuirE} and Theorem \ref{theo-LangmuirB}, respectively.

The main goal of this section is to establish decay estimates for the Green functions through the representation \eqref{def-FHk}.

\subsection{Electric Green function}

We first study the electric Green function $\FG_k(t)$. We obtain the following key result. 

\begin{proposition}\label{prop-GreenG} Let $\FG_k(t)$ be defined as in \eqref{def-FHk}, and let $\lambda_\pm^{\text{elec}}(k)$
be the electric dispersion relation constructed in Theorem \ref{theo-LangmuirE}. Then, we can write 
\begin{equation}\label{decomp-FG} 
\index{G@$\FG_k(t)$: electric Green function}
\index{G@$\FG^{osc}_{k,\pm}(t)$: oscillatory part of the electric Green function}
\index{G@$\FG^{r}_{k}(t)$: regular part of the electric Green function}
 \FG_k(t) = \delta_{t=0} + \sum_\pm \FG^{osc}_{k,\pm}(t)  +   \FG^{r}_k(t) ,
\end{equation}
with 
$$
\FG^{osc}_{k,\pm}(t) = e^{\lambda_\pm^{\text{elec}}(k) t} a_\pm(k),
$$
for a smooth  function $a_\pm(k)= \alpha_\pm(|k|^2)$ whose support is contained in $B(0, \kappa_0 + \delta )$ for some small $\delta>0$, and which satisfies 
\begin{equation}
\label{eq-apm0}
a_\pm(0) =  \pm { i \tau_0 \over 2}.
\end{equation} 
The regular part satisfies for all $ N \leq \lfloor{(N_{0} -9)/2}\rfloor$, for all $0\leq |\alpha|< N$, for all $k\in \RR^3$, and all $t \geq 0$ the estimate
\begin{equation}\label{bd-Gr}
|  |k|^\alpha \partial_k^\alpha \FG^{r}_k(t)| \le C_0 |k|^3 \langle k\rangle^{-4} \langle kt\rangle^{|\alpha|-N} ,
\end{equation}
for some universal constant $C_0>0$. 
\end{proposition}

The proof is postponed to Appendix~\ref{sec:proof1}.

\subsection{Magnetic Green function}\label{sec-GreenB}

Next, we derive decay estimates for the magnetic Green function $\FH_k(t)$. We obtain the following key result. 

\begin{proposition}\label{prop-Green}  Let $\FH_k(t)$ be defined as in \eqref{def-FHk}, and let $\lambda_\pm^{\text{mag}} = \pm i \nu_*(|k|)$
be the magnetic dispersion relation constructed in Theorem \ref{theo-LangmuirB}. Then, we can write 
\begin{equation}\label{decomp-FH} 
\index{G@$\FH_k(t)$: magnetic Green function}
\index{G@$\FH^{osc}_{k,\pm}(t)$: oscillatory part of the magnetic Green function}
\index{G@$\FH^{r}_{k}(t)$: regular part of the magnetic Green function}
\FH_k(t) = \sum_\pm \FH^{osc}_{k,\pm}(t)  +   \FH^{r}_k(t) ,
\end{equation}
where 
$$
\FH^{osc}_{k,\pm}(t) = e^{\pm i\nu_*(|k|) t} b_\pm(k),
$$
for some smooth functions $b_\pm(k)$ that satisfy $|b_\pm(k)|\lesssim \langle k\rangle^{-1}$. In addition, the regular part satisfies for all $N \leq \lfloor (N_{0}- 9)/2\rfloor$,  for all $0\leq |\alpha|<N$, for all $k\in \RR^3$, and all $t \geq 0$  the estimate
\begin{equation}\label{bound-Hrk}
| |k|^\alpha \partial_k^\alpha \FH^r_k(t)|  \lesssim |k| \langle k\rangle^{-2} \langle kt\rangle^{|\alpha|-N}  + |k| \langle k^3t\rangle^{|\alpha|-N} \chi_{\{|\partial_t| \ll |k| \ll1\}}
\end{equation}
in which $ \chi_{\{|\partial_t| \ll |k|\ll1\}}$ denotes a Fourier multiplier whose support is contained in $\{|\tau | \ll |k|\ll1\}$.

\end{proposition}

The proof is postponed to Appendix~\ref{sec:proof2}.


\subsection{Representation of the electromagnetic field}\label{sec-fields}


In this section, we give a complete representation of the electric and magnetic fields $E,B$ in terms of the initial data. Recall from \eqref{def-phiA} that the fields are obtained through 
\begin{equation}\label{recover-EB}
E = -\nabla_x \phi - \partial_t A , \qquad B = \nabla_x \times A,
\end{equation}
where $\phi, A$ solve the Vlasov-Maxwell system \eqref{VM-g}-\eqref{new-data}. We first derive the representation for electric and magnetic potentials. Precisely, we obtain the following. 

\begin{proposition} Let $\FG_k(t), \FH_k(t)$ be the Green functions defined as in \eqref{def-FHk}. Then, the electric and magnetic potentials $\Fphi_k, \FA_k$ of the system \eqref{VM-g}-\eqref{new-data} can be solved via the representation 
\begin{equation}
\label{rep-phiA} 
\begin{aligned}
\Fphi_k(t) &= \frac{1}{|k|^2}\int_0^t \FG_k(t-s) \FS_k(s)\; ds
\\
\FA_k(t) &=\partial_t  \FH_k(t) \FA_k^0 + \FH_k(t) \FA_k^1 +  \int_0^t \FH_k(t-s)  \mP_k\FS^\bj_k(s)\; ds
\end{aligned}
\end{equation}
where $\index{S@$S(t,x), S^\bj(t,x)$}
\FS_k(t), \FS_k^\bj(t)$ denote the charge and current densities generated by the free transport dynamics: namely
\begin{equation}\label{free-density} \FS_k(t) =\int e^{-ikt\cdot \hv}\Fg_k^0(v) \;dv , \qquad  \FS^\bj_k(t) =\int e^{-ikt\cdot \hv} \hv \Fg_k^0(v) \;dv .
\end{equation}

\end{proposition}

\begin{proof} Recall from Lemma \ref{lem-resolvent} that the Fourier-Laplace transform of the electric and magnetic potentials $\Tphi_k(\lambda), \TA_k(\lambda)$ solves the resolvent equations \eqref{resolvent}, namely 
\begin{equation}\label{resolvent1}
\begin{aligned}
\Tphi_k(\lambda) &= \frac{1}{|k|^2} \frac{1}{D(\lambda,k) }\int \frac{\Fg_k^0}{\lambda +  ik \cdot \hat v} dv
\\
\TA_k(\lambda) &=  \frac{1}{M(\lambda,k) }\int \frac{  \mP_k\hv \Fg_k^0 }{\lambda +  ik \cdot \hat v} dv +  \frac{\lambda }{M(\lambda,k) }\FA^0_k +  \frac{1}{M(\lambda,k) }  \FA_k^1. 
\end{aligned}
\end{equation}
Note that initial data $A^0_k, A^1_k$ are constants in $\lambda$. Therefore, in view of \eqref{def-THk} and \eqref{def-FHk}, the representation \eqref{rep-phiA} follows directly from \eqref{resolvent1}, upon taking the inverse Laplace transform. Finally, to compute the inverse of the Laplace transform of the right hand side of \eqref{resolvent1},  we note that $\lambda \mapsto \frac{1}{\lambda + ik\cdot \hv}$ is meromorphic in $\CC$, and therefore, by Cauchy's residue theorem,
$$
\begin{aligned}
\FS_k(t) &=   \frac{1}{2\pi i}\int_{\{\Re \lambda = \gamma_0\}}e^{\lambda t} \Big(\int \frac{\Fg_k^0(v)}{\lambda +  ik \cdot \hat v} \; dv\Big)\; d\lambda =\int e^{-ikt\cdot \hv}\Fg_k^0(v) \;dv.
\end{aligned}
$$
The same calculation holds for $\FS^\bj_k(t)$. The proposition follows. 
\end{proof}

\subsubsection{Electric potential}

We now derive the representation for the electric potential. Let $S$ be the charge density generated by the free transport, whose Fourier transform is equal to $\FS_k(t)$, see \eqref{free-density}. Then, plugging the decomposition \eqref{decomp-FG} into \eqref{rep-phiA}, we obtain  
\begin{equation}\label{rep-electric}
\begin{aligned}
\phi & = (-\Delta_x )^{-1}\Big[ S + \sum_\pm G^{osc}_\pm \star_{t,x} S + G^{r} \star_{t,x} S\Big] 
\end{aligned}
\end{equation}
From the dispersion of the free transport, the electric field component $\nabla_x \Delta_x^{-1}S(t,x)$ decays only at rate $t^{-2}$, 
which is far from being sufficient to get decay for the oscillatory electric field $G^{osc}_\pm \star_{t,x} \nabla_x \Delta_x^{-1}S$ 
through the space-time convolution. Formally, one would need $\nabla_x \Delta_x^{-1}S$ to decay at order $t^{-4}$ in order to establish a $t^{-3/2}$ 
decay for $G^{osc}_\pm \star_{t,x} \nabla_x \Delta_x^{-1}S$. 
However it turns out that, by integrating by parts in time, the low frequency part of the first term $S$ may be absorbed by the
second one, leading to a good  time decay, as we will now detail. We obtain the following.

\begin{proposition}[Electric potential decomposition]\label{prop-decompE} 
Let $S$ be defined as in \eqref{free-density}. 
Then, the corresponding electric potential $\phi$ satisfies
 \begin{equation}\label{rep-electricE}
\begin{aligned}
\index{p@$\phi^{osc}_\pm(t,x) $: oscillatory part of the electric scalar potential}
\index{p@$\phi^r(t,x) $: regular part of the electric scalar potential}
\phi &= \sum_\pm \phi^{osc}_\pm(t,x) + \phi^r(t,x) 
\end{aligned}\end{equation}
with 
 \begin{equation}\label{rep-phiosc}
\begin{aligned}
-\Delta_x\phi^{osc}_\pm(t,x) &= G_{\pm}^{osc}(t) \star_{x} \Big[  \frac{1 }{\lambda_\pm(i\partial_x)} S(0) +  \frac{1}{\lambda_\pm(i\partial_x)^2}\partial_tS(0)\Big] 
+ G^{osc}_{\pm} \star_{t,x}  \frac{1}{\lambda_\pm(i\partial_x)^2}\partial^2_t S 
\\
-\Delta_x\phi^r(t,x) &= G^r \star_{t,x} S+  \cP_2(i\partial_x)S  + \cP_4(i \partial_x) \partial_t S,
\end{aligned}\end{equation}
where $G_\pm^{osc}(t,x), G^r(t,x)$ are defined as in Proposition \ref{prop-GreenG}, with 
$\lambda_\pm (k)= \pm i \tau_*(|k|)$ as in  Theorem \ref{theo-LangmuirE}. In addition, $\cP_2(i\partial_x)$,  $\cP_4(i\partial_x)$  
denote smooth Fourier multipliers, which are sufficiently smooth and satisfy 
\begin{equation}\label{bounds-P24k}
\begin{aligned}
|  \cP_2(k)| + |  \cP_4(k)| & \lesssim |k|^2 \langle k\rangle^{-2}, \qquad \forall~k\in \RR^3,
\end{aligned}\end{equation}
and $\cP_4(k)$ is compactly supported in $\{ |k|< \kappa_0 + \delta\}$.

\end{proposition}

\begin{proof} 
We only focus on the low frequency regime: 
namely, the region $\{ |k|< \kappa_0 + \delta\}$, where we recall $k\in \RR^3$ is the Fourier frequency. 
 Indeed, in the high frequency regime, the result is  straightforward, as no oscillatory part is involved.
In the region $\{ |k|< \kappa_0 + \delta\}$, we shall exploit the time oscillations of $G^{osc}_\pm(t,x)$; with this idea in mind, we integrate by parts in time the second term of \eqref{rep-electric}, 
which gives
$$
G^{osc}_\pm \star_{t,x} S=   \frac{1}{\lambda_\pm(i\partial_x)} \Big[ G^{osc}_{\pm}(t,\cdot)\star_x S(0,x) 
-  G^{osc}_{\pm}(0,\cdot)\star_x S(t,x)  + G^{osc}_{\pm} \star_{t,x} \partial_tS(t,x) \Big],
$$
where
$\lambda_\pm (k)= \pm i \tau_*(|k|)$ are defined as in  Theorem \ref{theo-LangmuirE}. 
Using \eqref{small-tauk} and the fact that $a_\pm(0) = \pm \frac{i\tau_0}{2}$, we have
$$
\sum_\pm  \frac{1}{\lambda_\pm(i\partial_x)} G^{osc}_{\pm}(0,x) 
= \sum_\pm \int e^{i k\cdot x}  \frac{a_\pm(k)}{\lambda_\pm(k)} dk = \int_{\{ |k|< \kappa_0 + \delta\}} e^{i k\cdot x} \Big(1 + \cO(|k|^2)\Big)dk.  
$$
Therefore
\begin{equation}\label{cal-GoscE}
 \sum_\pm  \frac{1}{\lambda_\pm(i\partial_x)}G^{osc}_{\pm}(0,\cdot)\star_x S =  \widetilde \cP_3(i \partial_x) S - \widetilde \cP_2(i\partial_x) S  ,
\end{equation}
in which $\widetilde \cP_3(i\partial_x)$ (resp. $\widetilde \cP_2(i\partial_x)$) denotes the Fourier multiplier denotes the Fourier multiplier with $\widetilde \cP_3(k)$  (resp. $\widetilde \cP_2(k)$), being sufficiently smooth, with $\widetilde \cP_3(k)=1$ for $|k|\le 1$ and $\widetilde \cP_3(k)\lesssim 1$ (resp.
$\widetilde \cP_2(k)=0$ for $|k|\ge 1$ and $|\widetilde \cP_2(k)| \lesssim |k|^2$). 
Note that $S$ in (\ref{rep-electric}) cancels with the $S$ in (\ref{cal-GoscE}). Precisely, we have 
$$ 
\begin{aligned}
S + G^{osc}_\pm \star_{t,x} S &=  \frac{1}{\lambda_\pm(i\partial_x)} G^{osc}_{\pm}(t,\cdot)\star_x S(0,x) +  \frac{1}{\lambda_\pm(i\partial_x)}G^{osc}_{\pm} \star_{t,x} \partial_tS(t,x) 
\\&\quad + \widetilde \cP_2(i\partial_x) S + (1- \widetilde \cP_3(i \partial_x)) S.
\end{aligned}$$

Finally, as it turns out that $\partial_tS(t,x) $ decays not sufficiently fast for the convolution $G^{osc}_{\pm} \star_{t,x} \partial_tS(t,x) $, we further integrate by parts in $t$, yielding 
$$
G^{osc}_{\pm} \star_{t,x} \partial_tS =   \frac{1}{\lambda_\pm(i\partial_x)} \Big[ G^{osc}_{\pm}(t,\cdot)\star_x \partial_tS(0,x) 
- G^{osc}_{\pm}(0,\cdot)\star_x \partial_t S(t,x) +  G^{osc}_{\pm} \star_{t,x} \partial^2_tS(t,x) \Big],
$$
Using \eqref{small-tauk} and \eqref{eq-apm0}, we have
$$
\sum_\pm  \frac{1}{\lambda_\pm(i\partial_x)^2}G^{osc}_{\pm}(0,x) 
= \sum_\pm \int e^{i k\cdot x}  \frac{a_\pm(k)}{\lambda_\pm(k)^2} dk = \int_{\{ |k|< \kappa_0 + \delta\}} e^{i k\cdot x} \cO(|k|^2)dk,
$$
in which we stress that there is a cancellation that takes place at the leading order in $k$ for $k$ small. This proves that 
$$ 
\begin{aligned}
S + \sum_\pm G^{osc}_\pm \star_{t,x} S &= \sum_\pm  \frac{1}{\lambda_\pm(i\partial_x)} G^{osc}_{\pm}(t,\cdot)\star_x S(0,x)  - \sum_\pm  \frac{1}{\lambda_\pm(i\partial_x)^2}G^{osc}_{\pm}(t,\cdot)\star_x \partial_tS(0,x) 
\\& \quad + \widetilde \cP_2(i\partial_x) S + (1- \widetilde \cP_3(i \partial_x)) S + \cP_4(i \partial_x) \partial_t S.
\\& \quad + \sum_\pm  \frac{1}{\lambda_\pm(i\partial_x)^2}G^{osc}_{\pm} \star_{t,x} \partial^2_tS(t,x),
\end{aligned}$$
where $ \cP_4(i \partial_x) $ denotes a smooth Fourier multiplier with  $|\cP_4(k)\lesssim |k|^2$ that is compactly supported in $\{ |k|< \kappa_0 + \delta\}$. 
The proposition thus follows. 
\end{proof}


\subsubsection{Magnetic potential}


We now derive the representation for the magnetic potential. We obtain the following.

\begin{proposition}[Magnetic potential decomposition]\label{prop-decompB} 
Let $S^\bj$ be defined as in  \eqref{free-density}. Then, the corresponding magnetic potential $A$ satisfies
 \begin{equation}\label{rep-magneticA}
 \index{A@$A^{osc}_\pm(t,x) $: oscillatory part of the magnetic vector potential}
\index{A@$A^r(t,x) $: regular part of the magnetic vector potential}
\begin{aligned}
A &= \sum_\pm A^{osc}_\pm(t,x) + A^r(t,x) 
\end{aligned}\end{equation}
with 
 \begin{equation}\label{rep-Aosc}
\begin{aligned}
A^{osc}_\pm(t,x) &= H_{\pm}^{osc}(t) \star_{x} \Big[A^1 \pm i \nu_*(i\partial_x)A^0 +  \frac{1}{\nu_\pm(i\partial_x)}  \mP S^\bj(0) \Big]
+ H^{osc}_{\pm} \star_{t,x}  \frac{1}{\nu_\pm(i\partial_x)} \partial_t\mP S^\bj(t)
\\
A^r(t,x) &= \partial_t H^r(t) \star_x A^0 + H^r(t) \star_x A^1 + H^r \star_{t,x} \mP S^\bj(t)  -  \sum_\pm \frac{b_\pm(i\partial_x)}{\nu_\pm(i\partial_x)} \mP S^\bj(t)  ,
\end{aligned}\end{equation}
where $H_\pm^{osc}(t,x), H^r(t,x)$ are defined as in Proposition \ref{prop-Green}, with $\nu_\pm(k) = \pm i \nu_*(|k|)$ as in  Theorem \ref{theo-LangmuirB}.

\end{proposition}

\begin{proof}

From
 \eqref{rep-phiA}, we obtain  
\begin{equation}\label{rep-magnetic}
\begin{aligned}
A
&=\partial_t H(t) \star_x A^0 + H(t) \star_x A^1 + H \star_{t,x} \mP S^\bj(t) .
\end{aligned}
\end{equation}
Next, in view of the decomposition \eqref{decomp-FH}, we shall make use of oscillations and integrate by parts in time to deduce 
 $$ 
 \begin{aligned}
 H^{osc}_\pm \star_{t,x} \mP S^\bj(t) =  \frac{1}{\nu_\pm(i\partial_x)} \Big[
 H^{osc}_{\pm}(t,\cdot)\star_x \mP S^\bj(0) 
- H^{osc}_{\pm}(0,\cdot)\star_x \mP S^\bj(t)  + H^{osc}_{\pm} \star_{t,x} \partial_t\mP S^\bj(t) \Big]
\end{aligned}$$
where $\nu_\pm(k)= \pm i \nu_*(|k|)$ are defined as in  Theorem \ref{theo-LangmuirB}. Collecting all the terms, we obtain the proposition. 
\end{proof}

\section{Decay estimates}\label{sec-decay}

We are now ready to prove the main result of this paper, Theorem \ref{theo-main}. Recall that the fields are obtained through 
$$
E = -\nabla_x \phi - \partial_t A , \qquad B = \nabla_x \times A,
$$
where the potentials satisfy the representations \eqref{rep-electricE} and \eqref{rep-magneticA}. We will study the contribution of  each term in the above expression.  Beforehand, we establish decay for the electric and magnetic Green functions in the physical space, and we state dispersion estimates related to the relativistic free transport operator.

\subsection{Electric and magnetic Green functions in the physical space}

In this section, we bound the Green functions $G(t,x)$ and $H(t,x)$ in the physical space. To this end, we will use the homogeneous Littlewood-Paley decomposition of $\mathbb{R}^{3}$, as recalled in Appendix~\ref{sec:LittlewoodPaley}. 
Namely, we shall decompose functions as
\begin{equation}\label{def-LP-Green}
h(x)= \sum_{q \in \mathbb{Z}} P_q h(x), 
\end{equation}
where $P_q$ denotes the Littlewood-Paley projection on the dyadic interval $[2^{q-1}, 2^{q+1}]$.

We shall prove the following proposition. 

\begin{proposition}\label{prop-Greenphysical} Let $G(t,x)$ and $H(t,x)$ be the Green functions constructed in Proposition \ref{prop-GreenG} and Proposition \ref{prop-Green}, respectively. We have 
\begin{equation}\label{est-HoscL2inf}
\begin{aligned}
\|G^{osc}_\pm\star_x f\|_{L^2_x} &\lesssim \| f\|_{L^{2}_x}, \qquad&\|G^{osc}_\pm\star_x f\|_{L^\infty_x} \lesssim \langle t\rangle^{-\frac 32}\|  f\|_{L^{1}_x},
\\
\|H^{osc}_\pm\star_x f\|_{L^2_x} &\lesssim \| f\|_{L^{2}_x}, \qquad &\|H^{osc}_\pm\star_x f\|_{B^0_{\infty,2}}\lesssim \langle t\rangle^{-\frac 32}\| b_\pm(i\partial_x)f\|_{B^3_{1,2}},
\end{aligned}
\end{equation}
where the Fourier multipliers $b_\pm(k)$ are defined in Proposition \ref{prop-Green} and satisfy $|b_\pm(k)|\lesssim \langle k\rangle^{-1}$. In addition, letting $\chi_0(k)$ be a smooth cutoff function whose support is contained in $\{|k|\le 1\}$, for all $ 0\le n\le \lfloor (N_0-19)/2\rfloor$ and $p\in [1,\infty]$, there hold 
\begin{equation}\label{est-HrLp1} 
\begin{aligned}
\|\chi_0(i\partial_x) \partial_x^n \Delta_x^{-1}G^{r}(t) \|_{L^p_x} &\lesssim \langle t\rangle^{-4+3/p - n} ,
\\
\|\chi_0(i\partial_x) \partial_x^nH^r(t) \|_{L^p_x} &\lesssim \langle t\rangle^{-4/3+1/p - n/3} ,
\end{aligned}\end{equation}
and 
\begin{equation}\label{est-HrLp2} \| (1-\chi_0(i\partial_x)) \partial_x^n G^r (t) \|_{L^p_x}  + \| (1-\chi_0(i\partial_x)) \partial_x^n H^r (t) \|_{L^p_x} \lesssim \langle t\rangle^{-N}.\end{equation}
where $N \leq \lfloor (N_{0}- 7)/2 \rfloor$
\end{proposition}

Before providing the proof of the proposition, let us draw the consequences of the decay
estimates~\eqref{est-HoscL2inf} in $L^p$ spaces.

\begin{cor}
\label{coro:decayLposc}
There holds for all $p \in [2,\infty)$
\begin{equation}
\label{est-HoscLp}
\begin{aligned}
\|G^{osc}_\pm\star_x f\|_{L^p_x} &\lesssim \langle t\rangle^{-3\left(\frac 12-\frac1p\right)}\|  f\|_{L^{p'}_x},
\\
\|H^{osc}_\pm\star_x f\|_{L^p_x} &\lesssim \langle t\rangle^{-3\left(\frac 12-\frac1p\right)}\| b_\pm(i\partial_x)f\|_{B^{3\left(1-\frac{2}{p}\right)}_{p',2}},
\end{aligned}
\end{equation}
with $\frac 1p+\frac 1{p'}=1$. 

\end{cor}

\begin{proof}[Proof of Corollary~\ref{coro:decayLposc}]
The estimate for $G^{osc}_\pm$ follows from standard interpolation in Lebesgue spaces, while that for $H^{osc}_\pm$, recalling that $L^2= B^0_{2,2}$, is a consequence of interpolation in Besov spaces, as proved in~\cite{KMM}. This yields for all $p \in [2,+\infty)$ the estimate
$$
\|H^{osc}_\pm\star_x f\|_{B^0_{p,2}} \lesssim \langle t\rangle^{-3\left(\frac 12-\frac1p\right)}\| b_\pm(i\partial_x)f\|_{B^{\left(1-\frac{2}{p}\right)}_{p',2}}.
$$
and we conclude by the continuous embedding $B^0_{p,2} \hookrightarrow L^p$ for $p\in [2,\infty)$.
\end{proof}

\begin{proof}[Proof of Proposition~\ref{prop-Greenphysical}]

We start with the oscillatory Green functions, which we recall are defined by 
\begin{equation} \label{Hsreal}
\begin{aligned}
G^{osc}_\pm(t,x) &= \int e^{z_\pm(k) t + i k \cdot x} a_\pm(k) dk 
\\
H^{osc}_\pm(t,x) &= \int e^{\pm i\nu_*(|k|) t + i k \cdot x} b_\pm(k) dk .
\end{aligned}\end{equation}
Note that $G^{osc}_\pm$ is a smoothed version of the Green's function for the Schr\"odinger equation, since for small values of $|k|$, $z_\pm (k) =\pm i \tau_*(|k|)$ and $\tau_*(|k|)$ behaves like $1 + |k|^2$,  see Theorem \ref{theo-LangmuirE}. On the other hand, recall that $\nu_*(|k|)$ satisfies the Klein-Gordon type dispersion relation with positive mass $\tau_0$, see Theorem \ref{theo-LangmuirB}. 

Let us first explain how to handle the $G^{osc}_\pm$ part. The $L^2$ bound follows from the Plancherel formula, as $\Re \lambda_\pm (k) \leq 0$. 
We distinguish between the low frequency and the high frequency part, introducing a smooth cutoff function $\chi_0$ supported in a small neighborhood of $0$ and
 writing
$$
\begin{aligned}
G^{osc}_\pm(t,x) &= \int e^{z_\pm(k) t + i k \cdot x} \chi_0(k) a_\pm(k) dk  + \int e^{z_\pm(k) t + i k \cdot x} (1-\chi_0(k)) a_\pm(k) dk \\
&=:G^{osc, (1)}_\pm(t,x) + G^{osc, (2)}_\pm(t,x).
\end{aligned}
$$
For $G^{osc, (1)}_\pm$, thanks to~\eqref{lowerbound-taustarE} and \eqref{lowerbound-taustarDE}, we can use a  stationary phase argument (see  Lemma~\ref{lem:radialfourier} or  \cite{Gra,Stein}), which yields
$$
\|G^{osc,(1)}_\pm\star_x f(t)\|_{L^\infty_x}  \lesssim \frac{1}{\langle t \rangle^{3/2}} \| f\|_{L^1_x}.
$$
For $G^{osc, (2)}_\pm$, as we are away from $0$, the phase is non degenerate thanks to~\eqref{lowerbound-taustarDEextend} in Theorem~\ref{theo-LangmuirE}  and we can use a robust version of the non-stationary phase argument, namely Lemma~\ref{lem:disp1} (with $N=2$) to obtain
$$
\|G^{osc, (2)}_\pm\star_x f(t)\|_{L^\infty_x}  \lesssim \frac{1}{\langle t \rangle^{2}} \| f\|_{L^1_x},
$$
and we eventually obtain~\eqref{est-HoscLp} for the $G^{osc}_\pm$ part.
To handle the $H^{osc}_\pm$ part, we apply Lemma~\ref{lem-disp}, using the estimates
\eqref{KG-behave1},  \eqref{KG-behave2} and \eqref{KG-behave3} of Theorem~\ref{theo-LangmuirB}, which yields
$$
\|H^{osc}_\pm\star_x f(t)\|_{B^0_{\infty,2}}  \lesssim \frac{1}{\langle t \rangle^{3/2}} \| f\|_{{B^3_{1,2}}}.
$$
Next, we study the Green function $H^r(t,x)$; the estimates for $G^r(t,x)$ follow similarly (note that the limitation on $N$ comes from the analysis of $G^r$, see~\eqref{bd-Gr}). Recall that the Fourier transform $\FH^r_k(t)$ satisfies \eqref{bound-Hrk},  that is for all $N\in \mathbb{N}$,
\begin{equation}\label{bound-Hrk3}
|\FH^r_k(t)|  \lesssim |k|\langle k\rangle^{-2} \langle kt\rangle^{-N}  + |k| \langle |k|^3t\rangle^{-N} \chi_{\{|\partial_t| \ll |k| \ll1\}},
\end{equation}
in which $ \chi_{\{|\partial_t| \ll |k|\ll1\}}$ denotes a Fourier multiplier whose support is contained in $\{|\tau \ll |k|\ll1\}$. Hence, we bound for $N$ large enough so that $ n \leq N-5$ (which yields the restriction $n \leq \lfloor (N_{0}- 19)/2 \rfloor$
$$
\begin{aligned}
| \chi_0(i\partial_x) \partial_x^n H^r(t,x)|
& \lesssim \int_{\{|k|\le 1\}} |k|^{n} |\FH_k^r(t)| \; dk
 \lesssim \int_{\{|k|\le 1\}} |k|^{n+1}\langle |k|^3t \rangle^{-N} \; dk 
\lesssim \langle t\rangle^{-(n+4)/3}
\end{aligned}$$
proving \eqref{est-HrLp1} for $p=\infty$, while for $N$ large enough so that $n \leq N-3$
$$
\begin{aligned}
| (1-\chi_0(i\partial_x)) \partial_x^n H^r(t,x)|
& \lesssim \int_{\{|k|\ge 1/2\}} |k|^{n} |\FH_k^r(t)| \; dk
 \lesssim \int_{\{|k|\ge 1/2\}} |k|^{n-1}\langle |k|t \rangle^{-N} \; dk 
\lesssim \langle t\rangle^{- N} ,
\end{aligned}$$
proving \eqref{est-HrLp2} for $p=\infty$. As for $L^1_x$ estimates, we first note that for the high frequency part, we bound 
$$
\begin{aligned}
\| (1-\chi_0(i\partial_x)) \partial_x^n H^r(t)\|_{L^1_x}^2
& \lesssim \sum_{|\alpha|\le 2}\int_{\{|k|\ge 1\}} |k|^{2n} |\partial_k^\alpha \FH_k^r(t)|^2 \; dk
\\&  \lesssim  \sum_{|\alpha|\le 2} \int_{\{|k|\ge 1\}} |k|^{2n-2} \langle t\rangle^{2|\alpha|}\langle |k|t \rangle^{-2N} \; dk 
\lesssim \langle t\rangle^{-2(N-1)} .
\end{aligned}$$
It remains to treat the low frequency part. Using the Littlewood-Paley decomposition \eqref{def-LP}, we write 
$$
\begin{aligned}
\chi_0(i\partial_x) H^r(t,x) &= \sum_{q\le 0} P_q[ H^r(t,x) ],
 \end{aligned}$$
where 
$$
\begin{aligned}
P_q[H^r(t,x)] 
&=  \int_{\{2^{q-2}\le |k|\le 2^{q+2}\}} e^{ik \cdot x} \FH^r(t,k) \varphi(k/2^q)\; dk
\\
&= 2^{3q} \int_{\{\frac14\le |\tilde k|\le 4\}} e^{i \tilde k \cdot 2^q x}\FH^r(t,2^q \tilde k) \varphi(\tilde k)\; d\tilde k
\end{aligned}$$
Using the estimates \eqref{bound-Hrk} from Proposition \ref{prop-Green} for $|k|\le1$ (and so $q\le 0$), we obtain 
$$| \partial_{\tilde k}^\alpha [\FH^r(t,2^q \tilde k)]| \le | 2^{q|\alpha|}\partial_{k}^\alpha [\FH^r(t,2^q \tilde k)]|  \lesssim 2^q \langle 2^{3q} t \rangle^{-N +|\alpha|} 
$$
for $\frac14 \le |\tilde k|\le 4$. Therefore, integrating by parts repeatedly in $\tilde k$ and taking $|\alpha| =4$ in the above estimate, we have 
$$
\begin{aligned}
|P_q[H^r(t,x)]| 
&
\lesssim  2^{3q} \Big| 
 \int e^{i \tilde k \cdot 2^q x}\FH^r(t,2^q \tilde k) \varphi(\tilde k)\; d\tilde k
 \Big|
\\&
\lesssim  2^{3q} 
 \int \langle 2^q x\rangle^{-4} |\partial_{\tilde k}^4 [\FH^r(t,2^q \tilde k) \varphi(\tilde k)]| \; d\tilde k
\\&
\lesssim  2^{4q} \langle 2^q x\rangle^{-4} \langle 2^{3q} t \rangle^{-N +4}.
\end{aligned}$$
Taking the $L^1_x$ norm, we obtain 
$$
\begin{aligned}
 \| \chi_0(i\partial_x) \partial_x^n H^r(t)\|_{L^1_x} 
 &\lesssim \sum_{q\le 0} 2^{nq}\|P_q[H^r(t)]\|_{L^1_x} 
\lesssim \sum_{q\le 0} 2^{q(n+1)} \langle 2^{3q} t \rangle^{-N + 4}\lesssim \langle t\rangle^{-(n+1)/3},
  \end{aligned}$$ 
which gives \eqref{est-HrLp1} for $p=1$. The $L^p$ estimates follow from an interpolation between the $L^1$ and $L^\infty$ estimates. This completes the proof of the proposition. 
\end{proof}

\subsection{Relativistic free transport dispersion}
Let us recall from \eqref{free-density} that the charge and current densities $S$ and $S^\bj$ by the free transport in the physical space read
$$
S(t,x) =\int g^0(x - \hv t, v) \;dv, \qquad S^\bj(t,x) =\int \hv g^0(x - \hv t, v) \;dv.
$$

\begin{lem}
\label{lem:bound-SfreeLp}For all $p \in [1,+\infty]$, we have, for all $n \in \mathbb{N}$ and $\alpha \in \mathbb{N}^3$,
\begin{equation}\label{bound-SfreeLp}
\| \partial_t^n\partial_x^\alpha S(t)\|_{L^p_x} + \| \partial_t^n\partial_x^\alpha S^\bj(t)\|_{L^p_x} \lesssim t^{-3(1-1/p)-n-|\alpha|} \sum_{|\beta|\le n + |\alpha|}\|  \langle v\rangle^{5+ n + |\alpha|} \partial_v^\beta g^0(\cdot,v)\|_{L^{1}_x L^p_v}.
\end{equation}
More generally, any high order moment $S^\bj(t,x) =\int \hv_{i_1} \cdots \hv_{i_n} g^0(x - \hv t, v) \;dv$ also satisfies \eqref{bound-SfreeLp}.
\end{lem}

Recalling the definition of the shifted distribution function in~\eqref{new-data}, we shall use in the following the bound, relying on the fact that $\mu$ and its derivatives decay fast enough, 
\begin{equation}\label{bound-SfreeLpinitialdata}
\sum_{|\beta|\le n + |\alpha|}\|  \langle v\rangle^{5+ n + |\alpha|} \partial_v^\beta g^0(\cdot,v)\|_{L^{1}_x L^p_v} \lesssim \sum_{|\beta|\le n + |\alpha|}\|  \langle v\rangle^{5+ n + |\alpha|} \partial_v^\beta f^0(\cdot,v)\|_{L^{1}_x L^p_v} + \|  \Delta_x^{-1} \nabla_x \times B^0 \|_{L^{1}_x}.
\end{equation}

\begin{proof}We refer to \cite{BHD,Big2} for closely related results.
First of all it is clear that
$$
\| S(t) \|_{L^1_x} = \| g^0\|_{L^1_{x,v}}.
$$
By the change of variables $p:=\hat{v}$, we write that
\begin{equation}
\label{eq:formuleS}
S(t,x) = \int_{B(0,1)}  g^0\left(x-t p, \frac{p}{\sqrt{1-|p|^2}}\right)  \frac{1}{\sqrt{1-|p|^2}^{5}} \, dp.
\end{equation}
We deduce that 
\begin{equation}
\label{eq:formulederivS}
\begin{aligned}
\partial_x S(t,x) &= \int_{B(0,1)}  \partial_x g^0\left(x-t p, \frac{p}{\sqrt{1-|p|^2}}\right)  \frac{1}{\sqrt{1-|p|^2}^{5}} \, dp \\
&= \frac{1}{t} \int_{B(0,1)}   \partial_p\left[ g^0\left(x-t p, \frac{p}{\sqrt{1-|p|^2}}\right)\right]  \frac{1}{\sqrt{1-|p|^2}^{5}} \, dp \\
&- \frac{1}{t} \int_{B(0,1)}  \partial_p\left(\frac{p}{\sqrt{1-|p|^2}}\right) \cdot \nabla_p  g^0\left(x-t p, \frac{p}{\sqrt{1-|p|^2}}\right)  \frac{1}{\sqrt{1-|p|^2}^{5}} \, dp \\
&= - \frac{1}{t} \int_{B(0,1)}  g^0\left(x-t p, \frac{p}{\sqrt{1-|p|^2}}\right)  \partial_p\left[\frac{1}{\sqrt{1-|p|^2}^{5}}\right] \, dp \\
&- \frac{1}{t} \int_{B(0,1)}  \partial_p\left(\frac{p}{\sqrt{1-|p|^2}}\right) \cdot \nabla_p  g^0\left(x-t p, \frac{p}{\sqrt{1-|p|^2}}\right)  \frac{1}{\sqrt{1-|p|^2}^{5}} \, dp 
\end{aligned}
\end{equation}
and
\begin{equation}
\label{eq:formulederiv2S}
\begin{aligned}
\partial_t S(t,x)&= - \frac{1}{t} \int_{B(0,1)}  g^0\left(x-t p, \frac{p}{\sqrt{1-|p|^2}}\right)  \partial_p\left[\frac{p}{\sqrt{1-|p|^2}^{5}}\right] \, dp \\
&- \frac{1}{t} \int_{B(0,1)}  \partial_p\left(\frac{p}{\sqrt{1-|p|^2}}\right) \cdot \nabla_p  g^0\left(x-t p, \frac{p}{\sqrt{1-|p|^2}}\right)  \frac{p}{\sqrt{1-|p|^2}^{5}} \, dp 
\end{aligned}
\end{equation}

By~\eqref{eq:formuleS} and the change of variables $y := x- tp$
we deduce that
$$
\begin{aligned}
| S(t,x)| \lesssim 
t^{-3} \| \sup_v \, \langle v\rangle^5g^0(\cdot,v)\|_{L^1_x}.
\end{aligned}
$$
By interpolation, we deduce that for all $p \in [1,\infty]$,
$$
\|  S(t)\|_{L^p_x}  \lesssim t^{-3(1-1/p)} \|  \langle v\rangle^5  g^0\|_{L^{1}_xL^p_v}.
$$
Similarly, iterating \eqref{eq:formulederivS} and \eqref{eq:formulederiv2S}, we have 
\begin{equation}\label{bound-Sfree}
\begin{aligned}
\| \partial_t^n\partial_x^\alpha S(t)\|_{L^1_x} &\lesssim t^{-n-|\alpha|}  \sum_{|\beta|\le n + |\alpha|}\|  \langle v\rangle^{5+ n + |\alpha|} \partial_v^\beta g^0(\cdot,v)\|_{L^1_{x,v}}, \\
\| \partial_t^n\partial_x^\alpha S(t)\|_{L^\infty} &\lesssim t^{-3-n-|\alpha|} \sum_{|\beta|\le n + |\alpha|}\| \sup_v \, \langle v\rangle^{5+ n + |\alpha|} \partial_v^\beta g^0(\cdot,v)\|_{L^1_x},
\end{aligned}
\end{equation}
for all $n\ge 0, \alpha \in \mathbb{N}^3$. 
By interpolation, we obtain \eqref{bound-SfreeLp} for $p\in [1,\infty]$. The same holds for the current density $S^\bj(t,x)$ and higher order moments.
\end{proof}

\subsection{Bounds on $\nabla_x \phi$}

We now bound each term in the representation formula \eqref{rep-electricE}
for $\phi$ obtained in Proposition~\ref{prop-decompE}. 
\begin{prop}
\label{lem:phidecay}
We have for all $p \in [2,\infty)$
\begin{equation}
\label{eq:phioscdecay}
\begin{aligned}
\| \nabla_x \phi^{osc}(t) \|_{L^p_x}  &\lesssim \epsilon_0 \langle t\rangle^{-3(1/2 - 1/p)},  \\
\| \nabla_x \phi^{osc}(t) \|_{L^\infty_x}  &\lesssim \langle t\rangle^{-3/2} \left(\epsilon_0 + \| \nabla_x \Delta_x^{-1} \rho_0\|_{L^1_x}  + \| \nabla_x \Delta_x^{-1} \nabla_x \cdot \bj[f^0]\|_{L^{1}_x} \right),   
\end{aligned}
\end{equation}
and for all $p \in [1,\infty]$ and all $\alpha \in \mathbb{N}^3$, $| \alpha | \leq \lfloor (N_{0} - 19)/2 \rfloor$, 
\begin{equation}
\label{eq:phiregdecay}
\begin{aligned}
\| \partial^\alpha_x \nabla_x \phi^r(t)\|_{L^p_x} \lesssim \langle t\rangle^{-3(1-1/p)-|\alpha|}.
\end{aligned}
\end{equation}

\end{prop}

\begin{proof}

Let us start with the regular part, namely we study $\nabla_x \phi^r$,  for which we recall from \eqref{rep-phiosc} that 
$$
\begin{aligned}
\nabla_x \phi^r(t,x) &= -\nabla_x \Delta_x^{-1} \Big[G^r \star_{t,x} S+  \cP_2(i\partial_x)S  + \cP_4(i \partial_x) \partial_t S\Big].
\end{aligned}
$$
Recall that $\cP_2$ and $\cP_4$ are defined as in Proposition \ref{prop-decompE} with $\cP_2(k) = \cO(|k|^2\langle k\rangle^{-2})$ and $\cP_4(k) = \cO(|k|^2\langle k\rangle^{-2})$ (the latter is compactly supported in $|k|\le 1$). 
Therefore, both $\cP_2(i\partial_x) \Delta_x^{-1}\nabla_x$ and $\cP_4(i\partial_x) \Delta_x^{-1}\nabla_x$ are bounded in $L^p$ for $1\le p\le \infty$, see Lemma \ref{lem-potential}. This proves  that for all $p\in [1,\infty]$,
$$ \| \cP_2(i\partial_x) \Delta_x^{-1}\nabla_x S(t)\|_{L^p_x} \lesssim \| S(t)\|_{L^p_x}  \lesssim   \langle t\rangle^{-3(1-1/p)} \left(\| \langle v\rangle^{5} \partial_v^\beta f^0(\cdot,v)\|_{L^{1}_x L^p_v}  + \|  \Delta_x^{-1} \nabla_x \times B^0 \|_{L^{1}_x} \right).$$
The estimates for $ \cP_4(i \partial_x) \Delta_x^{-1} \nabla_x \cdot \partial_t S$ also follow identically. On the other hand, using Proposition \ref{prop-Greenphysical}, we have 
\begin{equation}\label{est-Grphysical} 
\|\partial_x^\alpha \Delta_x^{-1} G^r(t)\|_{L^p_x} \lesssim \epsilon_0 \langle t\rangle^{-4+3/p-|\alpha|} ,
\end{equation}
for any $p\in [1,\infty]$, noting that the high frequency part satisfies better decay estimates. 
Therefore, we bound 
$$  
\begin{aligned}
\| \Delta_x^{-1}G^r \star_{t,x} \partial_x S\|_{L^p_x} 
&\lesssim
 \int_0^{t/2}  \|\partial_x\Delta_x^{-1}G^r(t-s)\|_{L^p_x}\|S(s)\|_{L^1_x} \; ds
+  \int_{t/2}^t  \| \Delta_x^{-1}G^r(t-s)\|_{L^1_x}\|\partial_x S(s)\|_{L^p_x} \; ds
\\
&\lesssim 
 \epsilon_0\int_0^{t/2} \langle t-s\rangle^{-5+3/p} \; ds 
+ \epsilon_0 \int_{t/2}^t  \langle t-s\rangle^{-1}\langle s\rangle^{-3(1-1/p)-1}\; ds
\\
&\lesssim \epsilon_0
\langle t\rangle^{-3+3/p} .
\end{aligned} $$
As a result, we obtain $\| \nabla_x \phi^r(t)\|_{L^p_x} \lesssim \epsilon_0 \langle t\rangle^{-3+3/p}$ for $1\le p\le \infty$. The higher derivative estimates follow similarly. 

\bigskip

Next, we tackle the singular part.
Recall first that by definition, $S(0) = \rho[g^0]$. 
Thanks to the bound \eqref{est-HoscLp} of Corollary~\ref{coro:decayLposc}, we 
first have for all $p\in [2,\infty]$
$$\begin{aligned}
 \| \frac{1}{\nu_\pm(i\partial_x)} G_{\pm}^{osc}(t) \star_{x} \nabla_x \Delta_x^{-1}S(0)\|_{L^p_x} & \lesssim  t^{- 3(1/2-1/p)} \| \nabla_x \Delta_x^{-1}\rho[g^0]\|_{L^{p'}_x} \\
 &\lesssim  t^{- 3(1/2-1/p)} \| \nabla_x \Delta_x^{-1}\rho[f^0]\|_{L^{p'}_x}.
 \end{aligned}$$
 since $ \rho[g^0] =  \rho[f^0]$ (see \eqref{new-data}). For $p \in [2,\infty)$, using the assumption that $ \int \rho[f^0]\; dx = \iint f^0 \; dxdv =0$, we can apply the estimate \eqref{aLp-elliptic} of Lemma~\ref{lem-potential} to get
  $$
 \|  \frac{1}{\nu_\pm(i\partial_x)^2} \frac{1}{\nu_\pm(i\partial_x)} G_{\pm}^{osc}(t) \star_{x} \nabla_x \Delta_x^{-1}S(0)\|_{L^p_x}  \lesssim  t^{- 3(1/2-1/p)} \| \langle x\rangle \rho[f^0]\|_{L^1_x \cap L^\infty_x}.
$$
\begin{rem}
Observe that the above estimate in general fails for $p=\infty$, unless some additional assumption on the vanishing of higher moments of $\rho[g^0]$.
\end{rem}
 Similarly, using $\partial_t S= - \nabla_x \cdot S^\bj$, we get for all $p \in [2,\infty]$,
 $$
\begin{aligned}
 \| \frac{1}{\nu_\pm(i\partial_x)} G_{\pm}^{osc}(t) \star_{x} \nabla_x \Delta_x^{-1} \partial_t S(0)\|_{L^p_x} & \lesssim  t^{- 3(1/2-1/p)} \| \nabla_x \Delta_x^{-1} \nabla_x \cdot \bj[g^0]\|_{L^{p'}_x} 
 \\&\lesssim  t^{- 3(1/2-1/p)} \| \nabla_x \Delta_x^{-1} \nabla_x \cdot \bj[f^0]\|_{L^{p'}_x}
 \end{aligned}$$
 upon noticing that $ \bj[g^0] =  \bj[f^0] + 3 n_0 A^0$ and $\nabla_x \cdot A^0 =0$ (see again \eqref{new-data}). 
 Thanks to the boundedness of the operator $\nabla^2_x \Delta^{-1}_x$ in $L^p_x$ for $p\in (1,\infty)$, we have for all $p \in [2,\infty)$  
$$
\begin{aligned}
 \| G_{\pm}^{osc}(t) \star_{x} \nabla_x \Delta_x^{-1} \partial_t S(0)\|_{L^p_x} & \lesssim  t^{- 3(1/2-1/p)} \| \nabla_x \Delta_x^{-1} \nabla_x \cdot \bj[f^0]\|_{L^{p'}_x} 
 \\&\lesssim  t^{- 3(1/2-1/p)} \| \bj[f^0]\|_{L^{p'}_x}.
 \end{aligned}$$
Finally, we treat the space-time convolution term $G^{osc}_{\pm} \star_{t,x} \nabla_x \Delta_x^{-1} \partial^2_t S$. Observe that we can further integrate by part in time, yielding  
\begin{equation}
\label{eq:expGosc}
\begin{aligned}
G^{osc}_{\pm} \star_{t,x} \nabla_x \Delta_x^{-1} \partial^2_t S &=   \frac{1}{\lambda_\pm(i\partial_x)} \Big[ - G^{osc}_{\pm}(0,\cdot)\star_x \nabla_x \Delta_x^{-1} \partial^2_t S(t,x)  +G^{osc}_{\pm}(t,\cdot)\star_x \nabla_x \Delta_x^{-1} \partial^2_t S(0,x) 
\\&\quad +  G^{osc}_{\pm} \star_{t,x} \nabla_x \Delta_x^{-1} \partial^3_t S(t,x) \Big].
\end{aligned}
\end{equation}
Note that we can write 
$$\partial_t^2S = \sum_{i,j} c_{j,k} \partial_{x_ix_j}^2S^{ij}, \qquad \partial_t^3S = \sum_{i,j,k}  c_{i,j,k} \partial_{x_ix_jx_k}^3S^{ijk},$$ 
for some coefficients $c_{j,k}, c_{i,j,k}$, where $S^{ij} = \int \hv_i \hv_jg^0(x-\hv t)\; dv$ and $S^{ijk} = \int \hv_i \hv_j\hv_kg^0(x-\hv t)\; dv$. Let us observe that since $\mu$ is even,
$$
S^{ij} (0)= \int \hv_i \hv_j f^0(x)\; dv.
$$
Now, taking a smooth cutoff Fourier symbol $ \chi(k)$ with bounded support
and using Lemma \ref{lem:fouriermult} with $\sigma(k) = \chi(k)\frac{k_i k_j}{|k|^2}$, we bound 
for all $p \in[1, \infty]$,
$$
\begin{aligned}
\| \chi(i\partial_x)\nabla_x\Delta_x^{-1}\partial_t^2S(t)\|_{L^p_x} &\lesssim  \| S^{ij}(t)\|_{L^p_x}^{1/2} \| \partial_x S^{ij}(t)\|_{L^p_x}^{1/2}, 
\\
\| \chi(i\partial_x)\nabla_x\Delta_x^{-1}\partial_t^3S(t)\|_{L^p_x} &\lesssim  \| \partial_x S^{ijk}(t)\|_{L^p_x}^{1/2} \| \partial_x^{2} S^{ijk}(t)\|_{L^p_x}^{1/2}. 
\end{aligned}
$$ 
As a result, thanks to the dispersive estimates \eqref{bound-SfreeLp} of Lemma~\ref{lem:bound-SfreeLp}, we infer that for all $p \in[1, \infty]$,
$$
\begin{aligned}
\| \chi(i\partial_x)\nabla_x\Delta_x^{-1}\partial_t^2S(t)\|_{L^p_x} & \lesssim  \langle t\rangle^{-3(1-1/p)-1/2} \left(\sum_{|\beta|\le  1}\|  \langle v\rangle^{6} \partial_v^\beta f^0(\cdot,v)\|_{L^{1}_x L^p_v}  + \|  \Delta_x^{-1} \nabla_x \times B^0 \|_{L^{1}_x}\right),
\\
\| \chi(i\partial_x)\nabla_x\Delta_x^{-1}\partial_t^3S(t)\|_{L^p_x} &\lesssim   \langle t\rangle^{-3(1-1/p)-3/2} \left(\sum_{|\beta|\le 2}\|  \langle v\rangle^{7} \partial_v^\beta f^0(\cdot,v)\|_{L^{1}_x L^p_v} + \|  \Delta_x^{-1} \nabla_x \times B^0 \|_{L^{1}_x}\right).
\end{aligned}
$$ 
The first term in the right-hand side of~\eqref{eq:expGosc} is actually a regular term; arguing as above thanks to  Lemma~\ref{lem:fouriermult}, we  obtain
 $$
\begin{aligned}
 \| \frac{1}{\nu_\pm(i\partial_x)^3} G_{\pm}^{osc}(0) \star_{x} \nabla_x \Delta_x^{-1} \partial^2_t S(t)\|_{L^p_x}  \lesssim   \epsilon_0 \langle t\rangle^{-3(1-1/p)-1/2}.
 \end{aligned}$$
On the other hand, using \eqref{est-HoscL2inf}, \eqref{est-HoscLp}, we can bound 
the contribution of the  second  term in the  expansion~\eqref{eq:expGosc}, which yields
$$
\begin{aligned}
 \| \frac{1}{\nu_\pm(i\partial_x)^3} G_{\pm}^{osc}(t) \star_{x} \nabla_x \Delta_x^{-1} \partial^2_t S(0)\|_{L^p_x} & \lesssim \langle t \rangle^{- 3(1/2-1/p)} \|   \int \hv_i \hv_j \partial_x f^0(x)\; dv\|_{L^{p'}_x}^{1/2} \|   \int \hv_i \hv_j f^0(x)\; dv\|_{L^{p'}_x}^{1/2}.
 \end{aligned}$$
 
Finally, we bound the contribution of the last term of the expansion~\eqref{eq:expGosc} for all $p \in [2,\infty]$ by 
$$
\begin{aligned}
 \| \frac{1}{\nu_\pm(i\partial_x)^3} G^{osc}_{\pm} \star_{t,x} \nabla_x \Delta_x^{-1} \partial^3_t S\|_{L^p_x} 
 &\lesssim 
 \int_0^{t} (t-s)^{-3(1/2-1/p)} \| \chi(i\partial_x)\partial_x \Delta_x^{-1}\partial_t^3 S(s)\|_{L^{p'}_x} \; ds
\\ &\lesssim \epsilon_0
 \int_0^{t} (t-s)^{-3(1/2-1/p)} \langle s\rangle^{- 3/p-3/2} \; ds
\\ &\lesssim  \epsilon_0\langle t\rangle^{- 3(1/2-1/p)} . 
\end{aligned}$$
By \eqref{rep-phiosc} and combining the above estimates, we have thus proved that for all $p\in [2,\infty)$
$$\| \nabla_x\phi^{osc}_\pm(t)\|_{L^p_x} \lesssim \epsilon_0 \langle t\rangle^{-3(1/2 - 1/p)} $$
 and
 $$
 \| \nabla_x\phi^{osc}_\pm(t)\|_{L^\infty_x} \lesssim \epsilon_0 \langle t\rangle^{-3/2} \left(\epsilon_0 + \| \nabla_x \Delta_x^{-1} \rho_0\|_{L^1_x}  + \| \nabla_x \Delta_x^{-1} \nabla_x \cdot \bj[f^0]\|_{L^{1}_x} \right).
 $$
 The proof is finally complete.
\end{proof}

\subsection{Bounds on $\nabla_x \times A$ and $\partial_t A$}

We now bound each term in the representation formula \eqref{rep-magneticA} for $A$ obtained in Proposition~\ref{prop-decompB}. 

\begin{prop}
\label{lem:Adecay}
We have for all $p \in [2,\infty)$,
\begin{equation}
\label{eq:Aoscdecay}
\begin{aligned}
\| \partial_t A^{osc}(t) \|_{L^p_x}  + \| \nabla_x \times A^{osc}(t) \|_{L^p_x}  &\lesssim \epsilon_0 \langle t\rangle^{-3(1/2 - 1/p)},  \\
\| \partial_t A^{osc}(t) \|_{B^0_{\infty,2}}  + \| \nabla_x \times A^{osc}(t) \|_{B^0_{\infty,2}}  &\lesssim \epsilon_0 \langle t\rangle^{-3/2} ,   
\end{aligned}
\end{equation}
and for all $p \in [1,\infty]$ and all $\alpha \in \mathbb{N}^3$, $| \alpha | \leq \lfloor (N_{0} - 19)/2 \rfloor$,  and any $\delta>0$,
\begin{equation}
\label{eq:Aregdecay}
\begin{aligned}
\| \partial^\alpha_x \partial_t  A^r(t)\|_{L^p_x} + \| \partial^\alpha_x \nabla_x \times A^r(t)\|_{L^p_x} \lesssim \epsilon_0  \langle t\rangle^{-4/3 - |\alpha|/3+1/p+\delta}.
\end{aligned}
\end{equation}

\end{prop}

\begin{proof}

We start with the oscillation component $A^{osc}_\pm$, which reads
$$
\begin{aligned}
A^{osc}_\pm(t,x) &= H_{\pm}^{osc}(t) \star_{x} \Big[A^1 \pm i \nu_*(i\partial_x)A^0 +  \frac{1}{\nu_\pm(i\partial_x)}  \mP S^\bj(0) \Big]
+ H^{osc}_{\pm} \star_{t,x}  \frac{1}{\nu_\pm(i\partial_x)} \partial_t\mP S^\bj(t).
\end{aligned}
$$
Note that $S^\bj(0) = \bj[g^0]$. Let us first study $\nabla_x \times A^{osc}_\pm$.
Using the estimate  \eqref{est-HoscL2inf} of Proposition~\ref{prop-Greenphysical} and recalling the definition of $A^0$ and $A^1$ in \eqref{new-data}, we  bound
$$
\begin{aligned}
 &\Big \| H_{\pm}^{osc}(t) \star_{x} \Big[\nabla_x \times A^1 \pm i \nu_*(i\partial_x)\nabla_x \times A^0 +  \frac{1}{\nu_\pm(i\partial_x)}  \nabla_x \times  \mP S^\bj(0) \Big]\Big\|_{B^0_{\infty,2}}\\
 & \lesssim \langle t\rangle^{-\frac 32} \left(\| b_\pm(i\partial_x) \nabla_x \times A^1 \|_{B^3_{1,2}}  + \| b_\pm(i\partial_x) \nu_\pm(i\partial_x) \nabla_x \times A^0\|_{B^3_{1,2}}  +  \| \frac{b_\pm(i\partial_x)}{\nu_\pm(i\partial_x)}  \nabla_x \times \mP S^\bj(0) \|_{B^3_{1,2}}\right) \\
  &\lesssim \langle t\rangle^{-\frac 32} \Bigg(\| \nabla_x \times E^0 \|_{B^2_{1,2}}  + \| B^0\|_{B^3_{1,2}}  +  \|  \nabla_x \times \mP \bj[f^0] \|_{B^1_{1,2}} \Bigg),  
    \end{aligned}
$$
where we have used that $\nabla_x \times \nabla_x =0$, that since $\nabla_x \cdot B^0=0$, $ \nabla_x \times \nabla_x B^0=  - \Delta_x B^0$ and that $\mP A^0= A^0$.
Recall that the symbol $b_\pm(k)$ is sufficiently regular and bounded by $C_0 \langle k\rangle^{-1}$.
Similarly, by the estimate \eqref{est-HoscLp} of Corollary~\ref{coro:decayLposc}, we have for all $p \in [2,+\infty)$,
$$
\begin{aligned}
 &\Big \| H_{\pm}^{osc}(t) \star_{x} \Big[\nabla_x \times A^1 \pm i \nu_*(i\partial_x)\nabla_x \times A^0 +  \frac{1}{\nu_\pm(i\partial_x)}  \nabla_x \times  \mP S^\bj(0) \Big]\Big\|_{L^p_x}\\
  &\qquad\lesssim \langle t\rangle^{-3\left(\frac 12- \frac1p\right)}  \Bigg(\| \nabla_x \times E^0 \|_{B^{3\left(1-\frac{2}{p}\right)-1}_{1,2}}  + \| B^0\|_{B^{3\left(1-\frac{2}{p}\right)}_{1,2}}  +  \|  \nabla_x \times \mP \bj[f^0] \|_{B^{3\left(1-\frac{2}{p}\right)-2}_{1,2}} \Bigg).
  \end{aligned}
$$
Similar bounds hold for $\partial_t A^{osc}_\pm$, noting that $\partial_tH^{osc}_\pm = \pm i \nu_*(i\partial_x) H^{osc}_\pm$; however a few cancellations due to the curl operator do not happen. This results in the following estimates
$$
\begin{aligned}
 &\Big \| \partial_t  H_{\pm}^{osc}(t) \star_{x} \Big[ A^1 \pm i \nu_*(i\partial_x) A^0 +  \frac{1}{\nu_\pm(i\partial_x)}    \mP S^\bj(0) \Big]\Big\|_{B^0_{\infty,2}}\\
  &\qquad\lesssim \langle t\rangle^{-\frac 32} \Bigg(\| \nabla_x  E^0 \|_{B^2_{1,2}}  + \| \nabla_x \Delta_x^{-1} \nabla_x \times B^0\|_{B^3_{1,2}}  +  \|  \nabla_x \mP \bj[f^0] \|_{B^1_{1,2}} \Bigg),  \\
   &\Big \|  \partial_t  H_{\pm}^{osc}(t) \star_{x} \Big[ A^1 \pm i \nu_*(i\partial_x) A^0 +  \frac{1}{\nu_\pm(i\partial_x)}   \mP S^\bj(0) \Big]\Big\|_{L^p_x}\\
  &\qquad\lesssim \langle t\rangle^{-3\left(\frac 12- \frac1p\right)}  \Bigg(\| \nabla_x  E^0 \|_{B^{3\left(1-\frac{2}{p}\right)-1}_{1,2}}  + \| \nabla_x \Delta_x^{-1} \nabla_x  B^0\|_{B^{3\left(1-\frac{2}{p}\right)}_{1,2}}  +  \|  \nabla_x \mP \bj[f^0] \|_{B^{3\left(1-\frac{2}{p}\right)-2}_{1,2}} \Bigg).
    \end{aligned}
$$

\bigskip

Next, we treat the space-time convolution term 
$H^{osc}_{\pm} \star_{t,x} \partial_t\mP S^\bj(t)$
for which we may write 
$$\partial_t\mP S^\bj(t) = -\sum_k \partial_{x_k} \mP S^{jk},$$ 
where $S^{jk} = \int \hv_j \hv_k g^0(x-\hv t)\; dv$ satisfies the same decay estimates as $S^\bj$. 
Applying Lemma~\ref{lem:fouriermult} with $\sigma(i\partial_x) = \frac{\mathbb{P}}{\nu_\pm(i\partial_x)}$, we have for all $ p\in [1, \infty]$, for all $\delta \in (0,1)$,
$$
\begin{aligned}
 \| \frac{b_\pm(i\partial_x)}{\nu_\pm(i\partial_x)}\partial_t \partial_x^\alpha \mP S^\bj(t)\|_{B^{3(1-2/p)}_{p',2}} &\lesssim  \|   \partial_x^\alpha S^{jk}\|_{L^p_x}^{\frac{\delta}{1+\delta}} \|  \partial_x  \partial_x^\alpha S^{jk}\|_{L^{p'}_x}^{\frac{1}{1+\delta}} + \| \nabla_x^{1+ \lfloor 3(1-2/p) \rfloor}    \partial_x^\alpha S^{jk}\|_{L^{p'}_x}.
 \end{aligned}
 $$
 Hence, by the estimate~\eqref{bound-SfreeLp} of Lemma~\ref{lem:bound-SfreeLp} (and \eqref{bound-SfreeLpinitialdata}), we obtain  for all $ p\in [1, \infty]$, for all $\delta \in (0,1]$,
\begin{equation}\label{bd-DPSj}
\begin{aligned}
 &\| \frac{b_\pm(i\partial_x)}{\nu_\pm(i\partial_x)}\partial_t \mP S^\bj(t)\|_{B^{3(1-2/p)}_{p',2}}\\
 &\lesssim \langle t\rangle^{-3/p -1 -|\alpha| +{\frac{\delta}{1+\delta}}} \left(\sum_{|\beta|\le 1+|\alpha|+ \lfloor 3(1-2/p) \rfloor}\|  \langle v\rangle^{6+|\alpha|+\lfloor 3(1-2/p) \rfloor} \partial_v^\beta f^0(\cdot,v)\|_{L^{1}_x L^{p'}_v} + \|  \Delta_x^{-1} \nabla_x \times B^0 \|_{L^{1}_x} \right),
 \end{aligned}\end{equation}
On the other hand, by Proposition~\ref{prop-Greenphysical}, 
$$
\begin{aligned}
 \| \nabla_x \times  \Big(H^{osc}_{\pm} \star_{t,x} \frac{1}{\nu_\pm(i\partial_x)}\partial_t\mP S^\bj(t)\Big)\|_{B^0_{\infty,2}} 
 &\lesssim 
 \int_0^{t} \frac{1}{(t-s)^{3/2}} \Big \|\frac{b_\pm(i\partial_x)}{\nu_\pm(i\partial_x)}\partial_t \partial_x\mP S^\bj(s) \Big\|_{B^3_{1,2}} \; ds.
 \end{aligned}$$
As a consequence, applying~\eqref{bd-DPSj}, we have by taking $\delta=1$,
 $$
\begin{aligned}
 \| \nabla_x \times \Big(H^{osc}_{\pm} \star_{t,x} \frac{1}{\nu_\pm(i\partial_x)}\partial_t\mP S^\bj(t)\Big)\|_{B^0_{\infty,2}}
&\lesssim 
\epsilon_0 \int_0^{t} (t-s)^{-3/2} \langle s\rangle^{-3/2}\; ds
\\ &\lesssim \epsilon_0 \langle t\rangle^{-3/2}, 
\end{aligned}$$
Likewise, by estimate \eqref{est-HoscLp} of Corollary~\ref{coro:decayLposc}, we bound for all $p \in [2,\infty)$,
$$ 
\begin{aligned}
\| \nabla_x \times \left(H^{osc}_{\pm} \star_{t,x} \frac{1}{\nu_\pm(i\partial_x)}\partial_t\mP S^\bj(t) \right)\|_{L^p_x} &\lesssim \int_0^t \langle t-s\rangle^{-3(1/2 - 1/p)} \|\frac{b_\pm(i\partial_x)}{\nu_\pm(i\partial_x)}\partial_t \partial_x\mP S^\bj(s) \|_{B^{3(1-2/p)}_{p',2}}\; ds \\ 
 &\lesssim \epsilon_0 \int_0^t \langle t-s\rangle^{-3(1/2 - 1/p)}  \langle s \rangle^{-3/2}\; ds \\&\lesssim \epsilon_0 \langle t\rangle^{-3(1/2 - 1/p)}.
 \end{aligned}
$$ 
The estimates for the time derivative follow similarly, recalling that $\partial_tH^{osc}_\pm = \pm i \nu_*(i\partial_x) H^{osc}_\pm$. This proves 
$$ \| \nabla_x \times A^{osc}_\pm(t)\|_{L^p_x} +  \| \partial_t A^{osc}_\pm(t)\|_{L^p_x} \lesssim \epsilon_0 \langle t\rangle^{-3(1/2-1/p)},$$
for any $2\le p\le \infty$.

Finally, we study the regular part, namely $A^r$, which is defined by 
$$
\begin{aligned}
A^r(t,x) &= \partial_t H^r(t) \star_x A^0 + H^r(t) \star_x A^1 + H^r \star_{t,x} \mP S^\bj(t)  -  \sum_\pm \frac{b_\pm(i\partial_x)}{\nu_\pm(i\partial_x)} \mP S^\bj(t)  .
\end{aligned}
$$
The last term can be treated as in \eqref{bd-DPSj}, which gives for all $p\in [1,\infty]$, for all $\delta>0$,
$$
\|  \frac{b_\pm(i\partial_x)}{\nu_\pm(i\partial_x)} \partial^\alpha_x \mP S^\bj(t)\|_{L^p_x} \lesssim \epsilon_0  \langle t\rangle^{-3(1-1/p)-|\alpha| +\delta} \left(\sum_{|\beta|\le 1+|\alpha|}\|  \langle v\rangle^{6+|\alpha|} \partial_v^\beta f^0(\cdot,v)\|_{L^{1}_x L^{p'}_v} + \|  \Delta_x^{-1} \nabla_x \times B^0 \|_{L^{1}_x} \right).
$$
The terms involving the initial data terms satisfy the same decay estimates as those for $H^r(t)$. As for the space-time convolution term $ H^r \star_{t,x} \mP S^\bj(t)$, we recall from Proposition \ref{prop-Greenphysical} that 
$$\|\chi(i\partial_x)\partial_x^nH^r(t) \|_{L^p_x} \lesssim \epsilon_0\langle t\rangle^{-4/3+1/p - n/3},
$$
for $n\ge 0$ and for some cutoff symbol $\chi(k)$ with compact support. Note that it suffices to bound the low frequency part, as the other component decays rapidly in $t$. We compute, using~\eqref{bd-DPSj},
$$  
\begin{aligned}
\| \chi(i\partial_x)\partial_x H^r \star_{t,x} \mP S^\bj(t)\|_{L^p_x} 
&\lesssim
 \int_0^{t/2}  \| H^r(t-s)\|_{L^p_x}\|\chi(i\partial_x)\partial_x\mP S^\bj(s)\|_{L^1_x} \; ds
\\&\quad +  \int_{t/2}^t  \|H^r(t-s)\|_{L^1_x}\|\chi(i\partial_x) \partial_x \mP S^\bj(s)\|_{L^p_x} \; ds
\\
&\lesssim 
\epsilon_0   \int_0^{t/2} \langle t-s\rangle^{-4/3+1/p} \langle s \rangle^{-1+\delta} \; ds 
+\epsilon_0   \int_{t/2}^t  \langle t-s\rangle^{-1/3}\langle s\rangle^{-3(1-1/p)-1+\delta}\; ds
\\
&\lesssim
\epsilon_0  \langle t\rangle^{-4/3 + 1/p + \delta} ,
\end{aligned} $$
for $1\le p\le \infty$. The time derivative estimates follow similarly, since 
$$\partial_t( H^r \star_{t,x} \mP S^\bj(t)) =  H^r(t) \star_{x} \mP S^\bj(0) +  H^r \star_{t,x} \partial_t\mP S^\bj(t)$$
with $ \partial_t\mP S^\bj(t) =  \partial_x\mP S^{jk}(t)$ as treated previously. This proves that  
$$ \| \partial_x A^r_\pm(t)\|_{L^p_x} +  \| \partial_t A^r_\pm(t)\|_{L^p_x} \lesssim \epsilon_0 
\langle t\rangle^{-4/3 + 1/p + \delta} ,$$
for any $1\le p\le \infty$ and for any fixed $\delta>0$. The higher derivative estimates follow similarly.

\end{proof}

We are finally in position to conclude the analysis. Theorem~\ref{theo-main} indeed follows from a combination of Propositions~\ref{prop-decompE} , \ref{prop-decompB}, \ref{lem:phidecay} and \ref{lem:Adecay}, setting
\begin{align*}
&E^{osc,(1)} = -\nabla_x \phi^{osc}, \quad E^{osc,(2)} = - \partial_t   A^{osc},\quad E^r = -\nabla_x \phi^r - \partial_t A^r,\\
&B^{osc} = \nabla_x \times A^{osc}, \quad B^{r} = \nabla_x \times A^{r}.
\end{align*}

\appendix

\section{Proof of Proposition~\ref{prop-GreenG}}
\label{sec:proof1}
 
 Let $\chi$ be a smooth cutoff function with support included in $B(0,\kappa_0 + \delta)$,  with values in $[0,1]$, such that $\chi \equiv 1 $ on $B(0,\kappa_0 + \delta/2)$. We then write
$$
\TG_k(\lambda ) =\chi(k) \TG_k(\lambda ) + (1-\chi(k)) \TG_k(\lambda ).
$$

\begin{itemize}

\item  {\sc The regular part.}
Let us directly study the Green function $\FG_{k}(t)$ via the representation \eqref{int-FGdecomp-n}. 
By definition, the Green kernel $\TG_k(\lambda ) $ is of the form 
$$
\begin{aligned}
  \TG_k(\lambda ) &=  1 + \frac{1-D(\lambda,k)}{D(\lambda,k)}  
= 1  -  \frac{ \cL[K_k(t)](\lambda)}{1+  \cL[K_k(t)](\lambda)} .
\end{aligned}
$$
We accordingly set for  the part supported in $|k|\geq \kappa_0 + \delta$ 
$$
\begin{aligned}
  \FG^{r, \text{high}}_k(t) &= -  \frac{1}{2\pi i}\int_{\{\Re \lambda = \gamma_0\}}e^{\lambda t}  (1-\chi(k))  \frac{ \cL[K_k(t)](\lambda)}{1+  \cL[K_k(t)](\lambda)} \; d\lambda \\
  &= -   \frac{e^{\gamma_0 t} }{2\pi }\int_{-\infty}^{+\infty} e^{i \tau t}  (1-\chi(k))  \frac{ \cL[K_k(t)](\gamma_0 + i \tau)}{1+  \cL[K_k(t)](\gamma_0 + i \tau)} \; d\tau.
  \end{aligned}
$$
We have from  Proposition  \ref{prop-nogrowth} and  Theorem \ref{theo-LangmuirE} that there is $c_0>0$ such that for all  $|k|\geq \kappa_0 + \delta/2$ and all $\Re \lambda \geq 0$,
\begin{equation}
\label{eq:boundbelow}
|1+  \cL[K_k](\lambda)| \geq c_0.
\end{equation}

The region  $|k|\geq \kappa_0 + \delta$ is thus the regular part  for which, there is no zero of the electric dispersion function $D(\lambda,k)$ in $\{ \Re \lambda \geq 0\}$.
This case is similar to the torus case, or to the screened potential case in whole space, and which was already much studied, see \cite{MV,Young, HKNR2, GNR1}. Namely the resolvent kernel is not singular and there is no oscillatory component.

By Corollary~\ref{coro:LKK}, the Leibnitz formula and~\eqref{eq:boundbelow}, we deduce that for every  $n$,
$$
\begin{aligned}
\left| (1-\chi(k))  \partial_\tau^{n} \frac{ \cL[K_k](\gamma_0 + i \tau)}{1+  \cL[K_k](\gamma_0 + i \tau)} \right| \lesssim   \frac{1}{(|k|^2 + \tau^2)^{ n/2 +1}}.
  \end{aligned}
$$
Therefore, taking first $n=0$, since $\int_{-\infty}^{+\infty}   \frac{1}{|k|^2 + \tau^2} \, d \tau \lesssim \frac{1}{|k|}$, we deduce that 
$$
  |\FG^{r,\text{high}}_k(t)|  \lesssim \frac{e^{\gamma_0 t}}{|k|}.
$$
On the other hand, by successive integrations by parts with respect to $\tau$, noting that there is no contribution 
from infinity, we obtain 
$$
\begin{aligned}
  |\FG^{r,\text{high}}_k(t)| 
    &\lesssim  e^{\gamma_0 t} \frac{1}{t^{N}}  \int_{-\infty}^{+\infty}   \frac{ 1}{(|k|^2 + \tau^2)^{ N/2 +1}}  \; d\tau \\
  &\lesssim  e^{\gamma_0 t} \frac{1}{|k|} \frac{1}{(|k|t)^{N}}
  \end{aligned}
$$
for $N \leq (N_{0}-9))/2$.
Letting $\gamma_0 \to 0^+$ (using Cauchy's residue theorem and Lebesgue's domination theorem), we end up with~\eqref{bd-Gr} for $\alpha=0$.

\item {\sc The singular case $|k|< \kappa_0 + \delta$}, which we shall focus on.

\end{itemize}

The analysis in the singular case is reminiscent of that of  \cite{HKNR3,BMM-lin}, up to a major difference: it is not possible to directly extend the resolvent kernel to a meromorphic function on a domain of the form $\{ \Re \lambda > - \delta |k|\}$. Using Corollary~\ref{coro:extenD}, we have two different continuations for $|\Im \lambda|  <|k|$ and $ |\Im \lambda|>|k|$, 
moreover, we still have a singularity at  the threshold $|k|=\kappa_0$. For this reason, we shall rely on a different strategy,  we refer to \cite{Toan} for closely related proofs.. 
Thanks to the precise description of the poles obtained in Theorem \ref{theo-LangmuirE}, we can resort to a Taylor expansion to isolate the leading singular contribution which appears as a rational function to which we can apply the residue formula.
The remainder in the Taylor expansion can be controlled by deforming the contour only up to the imaginary axis.

\bigskip

\noindent {\bf The case $|k|\leq \eps$.}
Given a small parameter $\eps >0$, we first focus on the low frequency regime $|k|\leq \eps$.  It follows from Theorem \ref{theo-LangmuirE} that the resolvent kernel $\TG_k(\lambda) = \frac{1}{D(\lambda,k)}$ is holomorphic in $\{\Re \lambda >0\}$. Hence, by Cauchy's residue theorem, we may move the contour of integration $ \{\Re \lambda = \gamma_0\}$ towards the imaginary axis so that 
\begin{equation}\label{int-FGdecomp}
\FG_{k}(t) =  \frac{1}{2\pi i}\int_{\Gamma}e^{\lambda t}  \TG_{k}(\lambda)\; d\lambda,
\end{equation}
where we have decomposed $\Gamma$ 
as 
\begin{equation}\label{def-Gamma} 
\Gamma = \Gamma_1 \cup \Gamma_2 \cup \cC_\pm
\end{equation}
having set 
$$
\begin{aligned}
 \Gamma_1 &= \{ \lambda = i\tau, \quad |\tau \pm \tau_*(|k|)|\ge  |k|, \quad |\tau| > |k|\} ,
 \\
  \Gamma_2 &= \{\lambda = i\tau, \quad |\tau \pm \tau_*(|k|)|\ge  |k|, \quad |\tau| \le |k|\},
   \\
  \cC_\pm &= \{\Re \lambda \ge 0, \quad |\lambda \mp i \tau_*(|k|)| =  |k|\} .
    \end{aligned}$$

Note that the semicircle $\cC_\pm$ is to avoid the singularity due to the poles at $\lambda_\pm = \pm i \tau_*(|k|)$ of $\TG_{k}(\lambda)$ established in Theorem \ref{theo-LangmuirE}, while the integrals over $\Gamma_1$ and $\Gamma_2$ are understood as taking the limit of $\Re \lambda \to 0^+$, using Lebesgue's domination theorem.

To establish decay in $t$ as claimed in \eqref{bd-Gr}, we need to first integrate by parts repeatedly with respect to $\tau$, for $\Re \lambda=\gamma_0>0$ (using Corollary~\ref{coro:LKK} to ensure that there is no boundary term) to get for all $n=0, \ldots, N$,
\begin{align*} 
\FG_{k}(t) &=  \frac{1}{2\pi }\int_{-\infty}^{+\infty} e^{(\gamma_0+i\tau) t} \TG_{k}(\gamma_0 + i\tau)\; d\tau \\
&=  \frac{1}{2\pi (it)^n }\int_{-\infty}^{+\infty} e^{(\gamma_0+i\tau) t} \partial_{\tau}^n \TG_{k}(\gamma_0 +i \tau)\; d\tau \\
&= \frac{1}{2i \pi t^n }\int_{\Re \lambda= \gamma_0} e^{\lambda t} \partial_{\lambda}^n \TG_{k}(\lambda)\; d\lambda,
\end{align*}
then use the same change of contours as above and take the limit of $\Re \lambda \to 0^+$, to get 
\begin{equation}\label{int-FGdecomp-n}
\FG_{k}(t) =  \frac{(-1)^n}{2\pi i (|k|t)^n }\int_{\Gamma}e^{\lambda t} |k|^n \partial_\lambda^n \TG_{k}(\lambda)\; d\lambda.
\end{equation}
 In what follows, we shall use the formulation \eqref{int-FGdecomp-n}, instead of \eqref{int-FGdecomp}, to bound the Green function $\FG_{k}(t) $.

\bigskip \noindent{\it Further decomposition of $\FG_k(t)$.}
We start again from
$$
\TG_k(\lambda ) 
= 1 - \frac{ \cL[K_k(t)](\lambda)}{1+  \cL[K_k(t)](\lambda)}.
$$
Using the expansion \eqref{expansion-DL}--\eqref{comp-dtjKprop} of $ \cL[K_k(t)](\lambda)$ in Proposition~\ref{prop:D}, we write 
\begin{equation}
\label{eq:Gk}
\begin{aligned}
\TG_k(\lambda ) 
&=1 +  \frac{ - \lambda^2 \tau_0^2+  |k|^2 e_0+ \mathcal{R}(\lambda,k)}{\lambda^4 +\lambda^2 \tau_0^2 -  |k|^2e_0
- \mathcal{R}(\lambda,k)} 
\end{aligned}
\end{equation}
where we have set 
\begin{equation}\label{def-remainderRG}
 \mathcal{R}(\lambda,k) : = - \int_0^\infty e^{-\lambda t} \partial_t^4K_k(t) \; dt.
\end{equation}
Using \eqref{bd-RK} for $m=1$, we note that 
\begin{equation}\label{good1st-bdR}|\mathcal{R}(\lambda,k) |\lesssim |k|^2,\end{equation}
uniformly for all $\Re \lambda \ge 0$.
Differentiating~\eqref{def-remainderRG} with respect to $\lambda$ and using the fast decay of $\mu$, we get
\begin{equation}
\label{good1st-bdRderiv}
|\partial_\lambda^n \mathcal{R}(\lambda,k) |\lesssim {|k|^{2-n}}.
\end{equation}
In fact, upon integrating by parts repeatedly in $\lambda$ as carried out in the proof of \eqref{expansion-DL} and noting $\partial_t^{2j}K_k(0) =0$ for $j\ge 0$, we also obtain
\begin{equation}\label{good-Rb}
|\partial_\lambda^n \mathcal{R}(\lambda,k) |\lesssim |k|^4 |\lambda|^{2-n}, 
\end{equation}
for $n\ge 0$, where we stress that the bounds hold uniformly for $\Re \lambda \ge 0$.

We further write
\begin{equation}\label{TG-decompose}
\begin{aligned}
\TG_k(\lambda) 
&=1 + \TG_{k,0}(\lambda) + \TG_{k,1}(\lambda ) 
\end{aligned}
\end{equation}
where 
\begin{equation}\label{def-TGplus}
\begin{aligned}
\TG_{k,0}(\lambda) & =   \frac{-\lambda^2 \tau_0^2+  |k|^2 e_0}{\lambda^4 + \lambda^2\tau_0^2 - |k|^2e_0},
\\
\TG_{k,1}( \lambda ) 
&= \frac{\lambda^4  \mathcal{R}(\lambda,k) }{(\lambda^4 +\lambda^2 \tau_0^2-  |k|^2e_0)(\lambda^4 +\lambda^2\tau_0^2 -  |k|^2e_0
- \mathcal{R}(\lambda,k))} .
\end{aligned}
\end{equation}
We also denote by $\FG_{k,0}(t)$ and $\FG_{k,1}(t)$ the corresponding Green function, see \eqref{int-FGdecomp}. Denote by $\lambda_{\pm,0}(k)$ and $\lambda_\pm(k)$ the purely imaginary poles of $\TG_{k,0}(\lambda)$ and $\TG_{k}(\lambda)$, respectively. 
Via the representations \eqref{int-FGdecomp} and \eqref{int-FGdecomp-n}, we shall prove   
\begin{lem}
\label{lem:residue01}
We have the decomposition
\begin{equation}\label{bd-Grk0}
\Big| \FG_{k,0}(t)  -  \sum_\pm e^{\lambda_{\pm,0}t} \mathrm{Res} (\TG_{k,0}(\lambda_{\pm,0}))  \Big| \lesssim |k|^3 e^{- \sqrt{e_0} |kt|} ,
\end{equation}
\begin{equation}\label{bd-Grk1}
\Big| \FG_{k,1}(t)  -  \sum_\pm e^{\lambda_{\pm}t}\mathrm{Res} (\TG_{k}(\lambda_{\pm}))  +  \sum_\pm e^{\lambda_{\pm,0}t} \mathrm{Res} (\TG_{k,0}(\lambda_{\pm,0}))   \Big| \le C_1 |k|^3 \langle kt\rangle^{-N} ,
\end{equation}
where $\lambda_\pm(k)$ and $\lambda_{\pm,0}(k)$ are purely imaginary poles of $\TG_{k}(\lambda)$ and $\TG_{k,0}(\lambda)$, respectively. 

\end{lem}

In view of the decomposition \eqref{TG-decompose}, this would complete the proof of the decomposition \eqref{decomp-FG} and the remainder bounds \eqref{bd-Gr}, noting the residue $\mathrm{Res} (\TG_{k,0}(\lambda_{\pm,0}))$ in \eqref{bd-Grk1} is cancelled out with that in \eqref{bd-Grk0}.

\bigskip \noindent{\it Decomposition of $\FG_{k,0}(t)$.}
Let us start with bounds on $\FG_{k,0}(t)$ via the representation \eqref{int-FGdecomp}. 
Note that the polynomial $\lambda^4 + \lambda^2\tau_0^2 - |k|^2e_0$ has four distinct roots: 
\begin{equation}\label{def-lambda1234}\lambda_{\pm,0} = \pm i \Big(\frac{\tau_0^2+\sqrt{\tau_0^4+ 4 e_0 |k|^2}}{2}\Big)^{1/2} , \qquad \mu_{\pm,0} = \pm \Big( \frac{2e_0 |k|^2}{\tau_0^2 + \sqrt{\tau_0^4 + 4 e_0 |k|^2}}\Big)^{1/2} .\end{equation}
Note that $\lambda_{\pm,0}$ are purely imaginary roots, while $\mu_{\pm,0}$ are real valued. Note in particular that 
\begin{equation}\label{fake-roots}|\lambda_{\pm,0}(k) \mp i \tau_*(|k|)| \lesssim |k|^2 , \qquad |\mu_{\pm,0}(k) \mp \sqrt{e_0} |k|| \lesssim |k|^2.\end{equation}
Except $\mu_{+,0} \sim \frac{\sqrt{e_0}}{\tau_0} |k|$, the other three roots lie to the left of $\Gamma$, and thus, by Cauchy's residue theorem, we compute 
\begin{equation}
\label{eq:G00}
\begin{aligned} 
\FG_{k,0}(t)& =\frac{1}{2\pi i } \int_\Gamma e^{\lambda t} \TG_{k,0}(\lambda) \; d\lambda
\\&= \sum_\pm e^{\lambda_{\pm,0}t} \mathrm{Res} (\TG_{k,0}( \lambda_{\pm,0}))  + e^{\mu_{-,0}t} \mathrm{Res} (\TG_{k,0}(\mu_{-,0}))  
+ \frac{1}{2\pi i}\lim_{\gamma_0 \to -\infty}\int_{\Re \lambda = \gamma_0}
e^{\lambda t} \TG_{k,0}(\lambda)\; d\lambda
\\&=  \sum_\pm e^{\lambda_{\pm,0}t} \mathrm{Res} (\TG_{k,0}(\lambda_{\pm,0}))   + e^{\mu_{-,0}t} \mathrm{Res} (\TG_{k,0}(\mu_{-,0}))  ,
\end{aligned}
\end{equation}
upon using the fact that the last integral vanishes in the limit of $\gamma_0 \to -\infty$, noting that $\TG_{k,0}(\lambda) $ decays at order $\lambda^{-2}$ for $|\lambda|\to \infty$. It remains to bound $ e^{\mu_{-,0}t} \mathrm{Res} (\TG_{k,0}(\mu_{-,0})) $. A direct calculation yields 
\begin{equation}\label{fakeres-G0}
\begin{aligned}
e^{\mu_{-,0}t} \mathrm{Res} (\TG_{k,0}(\mu_{-,0})) 
&=  \frac{\mu_{-,0}^4 e^{\mu_{-,0}t} }{(\mu_{-,0} - \lambda_{+,0})(\mu_{-,0} - \lambda_{-,0})(\mu_{-,0} - \mu_{+,0})}
\\
& = \frac{\mu_{-,0}^3 e^{\mu_{-,0}t} }{2(\mu_{-,0} - \lambda_{+,0})(\mu_{-,0} - \lambda_{-,0})},
\end{aligned}\end{equation}
upon using $\mu_{+,0} = -\mu_{-,0}$. By definition, we note that $\lambda_{\pm,0} = \pm i \tau_0(1 + \cO(|k|^2))$ and $\mu_{\pm,0} = \pm \sqrt{ e_0} |k| (1 + \cO(|k|^2))$. 
Hence, 
\begin{equation}
\label{eq:G001}
e^{\mu_{-,0}t} |\mathrm{Res} (\TG_{k,0}(\mu_{-,0})) | \lesssim |k|^3 e^{- \sqrt{e_0} |kt|},
\end{equation}
giving \eqref{bd-Grk0}.

\bigskip 

Next, we prove \eqref{bd-Grk1} via the representation \eqref{int-FGdecomp-n}. 

\bigskip \noindent{\it Bounds on $\TG_{k,1}(\lambda)$ on $\Gamma_1$.}
We start with the integral of $\TG_{k,1}(\lambda)$ on $\Gamma_1$. We shall prove

\begin{lem}
\label{lem:gamma1}
We have
\begin{equation}
\label{eq:gamma1}
\Big|\int_{\{|\tau\pm \tau_*|\ge |k|, \;|\tau| > |k|\}} e^{i \tau t} |k|^n \partial_\lambda^n  \TG_{k,1}(i\tau) \; d\tau\Big|  \lesssim |k|^3,
\end{equation}
for $0\le n\le N$. 
\end{lem}

To this end, we shall obtain

\begin{lem} 
\label{lem:TG1gamma1}
For all $0\leq n\leq N$, for $|\tau|>|k|$ and $|\tau \pm \tau_*|\ge |k|$
\begin{equation}\label{bd-TGk10}
|(|k|\partial_\lambda)^n\TG_{k,1}(i\tau)|
\lesssim \left \{ \begin{aligned} |k|^4 |\tau|^2(1 + |\tau|^4)^{-2} \quad& \quad \mbox{if}\quad |\tau|\ge \frac32 \tau_0,
\\
 |k|^4 |\tau|^2  (|\tau \pm \tau_*|^2 + |k|^2)^{-1}  \quad& \quad \mbox{if}\quad \frac12\tau_0\le |\tau|\le \frac32 \tau_0,
 \\
 |k|^4 |\tau|^2( |\tau|^2 + |k|^2)^{-2} \quad& \quad \mbox{if}\quad |k|< |\tau| \le \frac12  \tau_0. 
  \end{aligned}\right.
\end{equation}

\end{lem}

We observe that Lemma~\ref{lem:gamma1} is a direct consequence of Lemma~\ref{lem:TG1gamma1}, using
$$\int_{\RR} (x^2 + |k|^2)^{-1}dx \le 4|k|^{-1}.$$

\begin{proof}[Proof of Lemma~\ref{lem:TG1gamma1}]

Recall that the polynomial $\tau^4 -\tau^2 \tau_0^2- |k|^2e_0$ has four distinct roots $\tau_{\pm,0} \sim \pm \tau_0 $ and $\tau_{\pm,1} \sim \pm i \frac{\sqrt{e_0}}{\tau_0} |k|$. Note also that $|\tau_* - \tau_0|\lesssim |k|^2$. Therefore, when $|\tau \pm \tau_*|\ge |k|$, we have 
\begin{equation}\label{lowbd-tau} 
\Big| \frac{1}{\tau^4 -\tau^2\tau_0^2 - |k|^2e_0} \Big |\lesssim \left \{ \begin{aligned} (1+|\tau|^4)^{-1}
\qquad& \quad \mbox{if}\quad |\tau|\ge \frac32   \tau_0,
\\
(|\tau \pm \tau_*| + |k|)^{-1} \qquad& \quad \mbox{if}\quad \frac12  \tau_0\le |\tau|\le \frac32  \tau_0,
 \\
 ( |\tau|^2 + |k|^2 )^{-1}\qquad& \quad \mbox{if}\quad |k|< |\tau| \le \frac12   \tau_0.
  \end{aligned}\right.
\end{equation}
We now show that the exact same upper bounds hold for $(\tau^4 - \tau^2 \tau_0^2-  |k|^2e_0
- \mathcal{R}(i\tau,k))^{-1}$ in the region when $|\tau|> |k|$ and $|\tau \pm \tau_*| \ge  |k|$. Indeed, the bound is clear when $|\tau| \ge \tau_0/2$, for which $|\mathcal{R}(i\tau,k)| \lesssim |k|^4$, upon using \eqref{good-Rb}.
On the other hand, in the case when $|k|< |\tau|\le \tau_0/2$, we note from \eqref{eq:formulaDtauleqk} that 
$$D(i\tau ,k) = 1 - \frac{1}{\tau^2} \omega(|k|^2/\tau^2),\qquad \omega(y) =  -\int_{-1}^1 \frac{u^2}{1-y u^2}  \kappa(u) \; du,
$$
where $\omega(y)$ and all its derivatives are real-valued and nonnegative (recalling $\kappa(u)\le 0$). As a result, the remainder $\mathcal{R}(i\tau,k)$ is real and positive.\footnote{In fact, a direct calculation yields
$
\mathcal{R}(i\tau,k)
= - \frac{|k|^4 }{\tau^2}\int_{-1}^1 \frac{ u^6 \kappa(u)}{1 - |k|^2 u^2 / \tau^2} \; du
$. }
In addition, since $|\tau|\le \tau_0/2$, we have $\tau^4 \le \frac{\tau_0^2}4 \tau^2$ and so 
\begin{equation}\label{low-Rsmall} 
\Big |\tau^4 - \tau^2 \tau_0^2-  |k|^2e_0
- \mathcal{R}(i\tau,k) \Big | \ge \frac 34 \tau^2\tau_0^2+  |k|^2e_0
\gtrsim |\tau|^2 + |k|^2,\end{equation}
as claimed.

Next, recalling \eqref{good-Rb}, we have $|\tau^4 \mathcal{R}(i\tau,k) | \lesssim |\tau|^2 |k|^4$. 
Putting these into \eqref{def-TGplus} yields 
\begin{equation}\label{bd-TGk10proof}
|\TG_{k,1}(i\tau)|
\lesssim \left \{ \begin{aligned} |k|^4 |\tau|^2(1 + |\tau|^4)^{-2} \quad& \quad \mbox{if}\quad |\tau|\ge \frac32  \tau_0,
\\
 |k|^4 |\tau|^2  (|\tau \pm \tau_*|^2 + |k|^2)^{-1}  \quad& \quad \mbox{if}\quad \frac12  \tau_0\le |\tau|\le \frac32  \tau_0,
 \\
 |k|^4 |\tau|^2( |\tau|^2 + |k|^2)^{-2} \quad& \quad \mbox{if}\quad |k|< |\tau| \le \frac12  \tau_0,
  \end{aligned}\right.
\end{equation}
whenever $|\tau \pm \tau_*|\ge |k|$.

Similarly, we next bound the integral of $\partial_\lambda^n\TG_{k,1}(\lambda) $ on $\Gamma_1$. It suffices to show that  $|k|^n \partial_\lambda^n \TG_{k,1}(i\tau)$ satisfies the same bounds as those for $\TG_{k,1}(i\tau)$. Indeed, in view of \eqref{lowbd-tau}, we have
$$
\Big | \partial_\lambda \Big(\frac{1}{\lambda^4+ \lambda^2\tau_0^2 - |k|^2e_0}\Big)_{\lambda = i\tau}\Big| 
\lesssim \left \{ \begin{aligned} \langle \tau\rangle^3(1 + |\tau|^4)^{-2} 
\qquad& \quad \mbox{if}\quad |\tau|\ge \frac32  \tau_0,
\\
(|\tau \pm \tau_*| + |k|)^{-2} \qquad& \quad \mbox{if}\quad \frac12  \tau_0\le |\tau|\le \frac32  \tau_0,
 \\
|\tau| (  |\tau|^2 + |k|^2)^{-2} \qquad& \quad \mbox{if}\quad |k| < |\tau| \le \frac12  \tau_0.
  \end{aligned}\right.
$$
Note that $(|\tau \pm \tau_*| + |k|)^{-1} \le |k|^{-1} $ and $|\tau| (  |\tau|^2 + |k|^2)^{-1} \le |k|^{-1}$. This proves that the $|k|\partial_\lambda$ derivative of $(\tau^4 -\tau^2 \tau_0^2 - |k|^2e_0)^{-1}$ satisfies the same bounds as those for $(\tau^4 -\tau^2 \tau_0^2- |k|^2e_0)^{-1}$, see \eqref{lowbd-tau}. The $|k|\partial_\lambda$ derivatives of $(\lambda^4 +\lambda^2\tau_0^2 -  |k|^2e_0
- \mathcal{R}(\lambda,k))^{-1}$ follow similarly, upon using \eqref{good-Rb} and the lower bound \eqref{low-Rsmall}.
\end{proof}

\bigskip \noindent{\it Bounds on $\TG_{k,1}(\lambda)$ on $\Gamma_2$.}
We next consider the integral on $\Gamma_2$. 

\begin{lem}
\label{lem:gamma2} We have
\begin{equation}
\label{eq:gamma2}
\Big|\int_{\Gamma_2} e^{ i\tau t}  |k|^n \partial_\lambda^n \TG_{k,1}(i\tau) \; d\tau \Big|\lesssim |k|^3,
\end{equation}
for $0\le n\le (N_{0}-9)/2$. 
\end{lem}

\begin{proof}[Proof of Lemma~\ref{lem:gamma2}] In this case, $\mathcal{R}(i\tau,k)$ is of order $|k|^2$ and thus no longer a remainder in the expansion of $ \cL[K_k(t)](i\tau)$.  Recall by definition that 
$$
\TG_{k,1}( i\tau) 
= \frac{\mathcal{R}(i\tau,k) }{(\tau^4 -\tau^2\tau_0^2 -  |k|^2e_0)(1 +  \cL[K_k(t)](i\tau))} .
$$
Using \eqref{lowbd-Ditau2},  we have
\begin{equation}\label{low-tau4} 
| 1+ \cL[K_k(t)](i\tau)| \gtrsim |k|^{-2}.
\end{equation}
Therefore, recalling $|\mathcal{R}(i\tau,k)|\lesssim |k|^2$ and using $ |\tau^4 -\tau^2 \tau_0^2-  |k|^2e_0| \gtrsim |\tau|^2 + |k|^2$ since $|\tau|\le |k|$, we obtain 
$$
\begin{aligned}
| \TG_{k,1}(i\tau)|
& \lesssim 
\frac{ |k|^4}{ |\tau|^2 + |k|^2} 
\lesssim |k|^2,
\end{aligned}$$
which gives 
$$
\Big|\int_{\{ |\tau| \le |k|\}} e^{ i\tau t}  \TG_{k,1}(i\tau) \; d\tau \Big|\lesssim |k|^3.
$$

Similarly, we now bound the $|k|\partial_\lambda$ derivatives. First, by definition, we compute 
$$
\partial^n_\lambda \cL[K_k(t)](i\tau ) 
 = \int_0^\infty e^{-i\tau t} (-it)^n K_k(t)\; dt  .
$$
Using the decay estimates \eqref{bounds-Kkt} on $K_k(t)$, we have $|k^n\partial^n_\lambda \cL[K_k(t)](i\tau )|\lesssim |k|^{-2}$. Therefore, together with \eqref{low-tau4}, 
$$
 \Big| k\frac{d}{d\lambda} \Big( \frac{1}{1 +  \cL[K_k(t)](\lambda)) } \Big)_{\lambda = i\tau} \Big| \lesssim |k|^4| k\partial_\lambda \cL[K_k(t)](i\tau )| \lesssim |k|^2.$$  
That is, the $|k|\partial_\lambda$ derivative of $(1 +  \cL[K_k(t)](i\tau))^{-1}$ satisfies the same bound as that for $(1 +  \cL[K_k(t)](i\tau))^{-1}$, that is bounded by $C_0 |k|^2$. Similarly, recalling \eqref{def-remainderRG}, we also have $| k^n \partial_\lambda^n \mathcal{R}(i\tau,k)|\lesssim |k|^{2}$. Therefore, $|k|^n \partial_\lambda^n \TG_{k,1}(i\tau)$ satisfies the same bounds as in \eqref{bd-TGk10proof} for $\TG_{k,1}(i\tau)$, yielding~\eqref{eq:gamma2}. 

\end{proof}

\bigskip \noindent{\it Bounds on $\TG_{k,1}(\lambda)$ on $\cC_\pm$.}
Finally, we study the case when $\lambda$ near the singularity of $\TG_k(\lambda)$: namely, when $\lambda$ is on the semicircle $|\lambda \mp i\tau_*(|k|)| = |k|$ with $\Re \lambda \ge 0$. 

\begin{lem}
\label{lem:nearsing}
We have for all $0\le n\le N$, the decomposition
\begin{equation}
\label{eq:nearsing}
\begin{aligned}
 \frac{(-1)^n}{2\pi i (|k|t)^n}\int_{\cC_\pm}e^{\lambda t} |k|^n\partial_\lambda^n\TG_{k,1}(\lambda)  \; d\lambda 
 &=  \sum_\pm  e^{\lambda_\pm t} \mathrm{Res} (\TG_{k}(\lambda_{\pm}))  -  \sum_\pm e^{\lambda_{\pm,0} t}  \mathrm{Res} (\TG_{k,0}(\lambda_{\pm,0}))  \\
&+ \cO(|k|^3 \langle kt\rangle^{-n}).
 \end{aligned}
 \end{equation}
 \end{lem}

\begin{proof}[Proof of Lemma~\ref{lem:nearsing}]

On $\cC_\pm$, since $|\lambda|\ge \tau_0/2$, we note from \eqref{good-Rb} that 
\begin{equation}\label{good-GRb1}
|\partial_\lambda^n \mathcal{R}(\lambda,k) |\lesssim |k|^4 ,
\end{equation}
uniformly in $\Re \lambda \ge 0$ for $0\le n\le N$. Recall from the decomposition \eqref{TG-decompose} that $\TG_{k,1}(\lambda ) = \TG_{k}(\lambda ) - \TG_{k,0}(\lambda )$, 
which reads
\begin{equation}\label{recomp-G1}
\begin{aligned}
\TG_{k,1}(\lambda ) = \frac{ \lambda^4 }{\lambda^4 +\lambda^2\tau_{0}^2 -  |k|^2e_0
- \mathcal{R}(\lambda,k)}  - \frac{ \lambda^4}{\lambda^4 +\lambda^2\tau_{0}^2 -  |k|^2e_0} .
\end{aligned}\end{equation}
Let $\lambda_\pm(k)$ and $\lambda_{\pm,0}(k)$ be the poles of the Green kernels 
$\TG_k(\lambda)$ and $\TG_{k,0}(\lambda) $, respectively, that lie on the imaginary axis, see Theorem~\ref{theo-LangmuirE}  and \eqref{def-lambda1234}. Clearly, they are isolated zeros of $\lambda^4 +\lambda^2\tau_0^2 -  |k|^2e_0
- \mathcal{R}(\lambda,k)$ and $\lambda^4 +\lambda^2\tau_0^2 -  |k|^2e_0$, respectively. By construction, we have
\begin{equation}\label{diff-lambdaG} 
 |\lambda_\pm(k) - \lambda_{\pm,0}(k)|\lesssim |k|^4.
\end{equation}
In addition, we have by a Taylor expansion, 
$$
\begin{aligned}
\lambda^4 +\lambda^2\tau_{0}^2 -  |k|^2e_0
& = \sum_{j=1}^{N} a_{\pm,j,0}(k) (\lambda - \lambda_{\pm,0}(k))^j, 
\\
\lambda^4 +\lambda^2 \tau_{0}^2 -  |k|^2e_0
- \mathcal{R}(\lambda,k)
 &=
\sum_{j=1}^{N} a_{\pm,j}(k) (\lambda - \lambda_\pm(k))^j+  \mathcal{R}_{\pm}(\lambda,k),
\end{aligned}$$
for any $\lambda \in \{|\lambda \mp i\tau_*| \le |k|\}$ with $\Re \lambda \ge 0$, where  by convention $a_{\pm, j,0}= 0$ for $j>4$ and $\partial_\lambda^j  \mathcal{R}_{\pm}(\lambda_{\pm},k) =0$ for $0\le j\le N$. 
Note that in the Taylor formula of the above second line, we are using that $D$ is holomorphic with a smooth extension
up to the boundary so that there is no term involving $\overline{\lambda}$.
In view of \eqref{good-GRb1} and \eqref{diff-lambdaG}, it is direct to deduce 
\begin{equation}\label{coeff-RaG} |a_{\pm,j,0}(k)  - a_{\pm,j}(k)| \lesssim |k|^4, \qquad  |\partial_\lambda^j \mathcal{R}_{\pm}(\lambda,k)|\lesssim |k|^4,\end{equation}
for $0\le j\le N$. Since the leading coefficients $a_{\pm,1,0}(k)$ and $a_{\pm,1}(k)$ never vanish, we obtain 
\begin{equation}\label{exp-inverseGR}
\begin{aligned}
\frac{\lambda^4}{\lambda^4 +\lambda^2 -  |k|^2e_0} 
& = \frac{1}{\lambda - \lambda_{\pm,0}(k)}
\sum_{j=0}^{N} b_{\pm,j,0}(k) (\lambda - \lambda_{\pm,0}(k))^j+  \widetilde{\mathcal{R}}_{\pm,0}(\lambda,k),
\\
\frac{\lambda^4}{\lambda^4 +\lambda^2 -  |k|^2e_0
- \mathcal{R}(\lambda,k)} 
 &= 
 \frac{1}{\lambda - \lambda_\pm(k)}
\sum_{j=0}^{N} b_{\pm,j}(k) (\lambda - \lambda_\pm(k))^j+  \widetilde{\mathcal{R}}_{\pm}(\lambda,k),
\end{aligned}\end{equation}
for $\lambda \in \{|\lambda \mp i \tau_*(|k|)| \le |k|\}$ with $\Re \lambda \ge 0$, where the coefficients $b_{\pm,j,0}(k)$ and $b_{\pm,j}(k)$ can be computed in terms of $a_{\pm,j,0}(k)$ and $a_{\pm,j}(k)$, together with $\partial_\lambda^j \widetilde{\mathcal{R}}_{\pm,0}(\lambda_{\pm,0},k) =0$ and $\partial_\lambda^j  \widetilde{\mathcal{R}}_{\pm}(\lambda_{\pm},k) =0$ for $0\le j\le N$. Importantly, it follows from \eqref{diff-lambdaG} and \eqref{coeff-RaG} that 
\begin{equation}\label{coeff-GRa1}
 |b_{\pm,j,0}(k)  - b_{\pm,j}(k)| \lesssim |k|^4, \qquad  |\partial_\lambda^j \widetilde{\mathcal{R}}_{\pm,0}(\lambda,k) - \partial_\lambda^j \widetilde{\mathcal{R}}_{\pm}(\lambda,k)|\lesssim |k|^4,\end{equation}
for $0\le j\le N$. Note in particular that by construction, 
\begin{equation}\label{def-Gbn0}
b_{\pm,0,0}(k) =  \mathrm{Res} (\TG_{k,0}(\lambda_{\pm,0}))  , \qquad b_{\pm,0}(k) = \mathrm{Res} (\TG_{k}(\lambda_{\pm})) 
\end{equation}
for each $\pm$.

We are now ready to bound $\TG_{k,1}(\lambda ) $, using \eqref{recomp-G1} and the above expansions. Indeed, using \eqref{exp-inverseGR}, we first write
\begin{equation}\label{exp-Glambdak1}
\begin{aligned}
\TG_{k,1}(\lambda )  &= 
\sum_{j=0}^{N} \Big[ \frac{b_{\pm,j}(k)}{ (\lambda - \lambda_\pm(k))^{-j+1}} -  \frac{b_{\pm,j,0}(k)}{ (\lambda - \lambda_{\pm,0}(k))^{-j+1}}\Big] +  \widetilde{\mathcal{R}}_{\pm}(\lambda,k) -  \widetilde{\mathcal{R}}_{\pm,0}(\lambda,k),
\end{aligned}\end{equation}
where, using \eqref{coeff-GRa1}, the remainder satisfies $|  \widetilde{\mathcal{R}}_{\pm}(\lambda,k) -  \widetilde{\mathcal{R}}_{\pm,0}(\lambda,k) |\lesssim |k|^4$ for $\lambda \in \{|\lambda \mp i \tau_*(|k|)| \le |k|\}$ with $\Re \lambda \ge 0$. Since the terms in the summation are meromorphic in $\mathbb{C}$, we may apply the Cauchy's residue theorem to deduce 
$$
\begin{aligned}
 \frac{1}{2\pi i}\int_{\cC_\pm}e^{\lambda t} \TG_{k,1}(\lambda)  \; d\lambda &=  \sum_\pm  e^{\lambda_\pm t} b_{\pm,0}(k)  -  \sum_\pm e^{\lambda_{\pm,0} t} b_{\pm,0,0}(k) 
  \\&\quad + 
 \frac{1}{2\pi i}  \sum_{j=0}^{N}  \int_{\cC_\pm^*}  e^{\lambda t} \Big[ \frac{b_{\pm,j}(k)}{ (\lambda - \lambda_\pm(k))^{-j+1}} -  \frac{b_{\pm,j,0}(k)}{ (\lambda - \lambda_{\pm,0}(k))^{-j+1}}\Big]  \; d\lambda
\\&\quad  + \frac{1}{2\pi i }\int_{\cC_\pm} e^{\lambda t} \Big[ \widetilde{\mathcal{R}}_{\pm}(\lambda,k) -  \widetilde{\mathcal{R}}_{\pm,0}(\lambda ,k)\Big]  \; d\lambda,
 \end{aligned}$$
 where $\cC_\pm^*$ denotes the semicircle $|\lambda \mp i \tau_*| = |k|$ with $\Re \lambda \le 0$.
 
 Let us first treat the last term of the above sum. By a further use of Cauchy's residue theorem and a continuity argument, we have
 $$
 \frac{1}{2\pi i }\int_{\cC_\pm} e^{\lambda t} \Big[ \widetilde{\mathcal{R}}_{\pm}(\lambda,k) -  \widetilde{\mathcal{R}}_{\pm,0}(\lambda ,k)\Big]  \; d\lambda=  \frac{1}{2\pi  }\int_{|\tau\mp i \tau_*|\leq |k|} e^{i\tau t} \Big[ \widetilde{\mathcal{R}}_{\pm}(i\tau,k) -  \widetilde{\mathcal{R}}_{\pm,0}(i \tau ,k)\Big]  \; d\tau.
$$
 Since $| \widetilde{\mathcal{R}}_{\pm}(i\tau,k) -  \widetilde{\mathcal{R}}_{\pm,0}(i\tau ,k) |\lesssim |k|^4$, the last integral term is clearly bounded by $C_0 |k|^5$. As for the integral on $\cC_\pm^*$, recalling \eqref{diff-lambdaG}, we have 
 $|\lambda- \lambda_\pm(k)| \ge |k|/2$ and $|\lambda- \lambda_{\pm,0}(k)| \ge |k|/2$ on $\cC_\pm^*$. Therefore, together with \eqref{diff-lambdaG} and \eqref{coeff-GRa1}, we bound for $\lambda \in \cC_\pm^*$, 
$$\Big |  \frac{b_{\pm,0}(k) }{\lambda - \lambda_\pm(k)}  - \frac{b_{\pm,0,0}(k)}{\lambda - \lambda_{\pm,0}(k)} \Big | \lesssim \frac{|k|^4}{|\lambda - \lambda_\pm(k) ||\lambda - \lambda_{\pm,0}(k)|} \lesssim |k|^2,$$
and 
$$  \sum_{j=1}^{N} \Big | \frac{b_{\pm,j}(k)}{ (\lambda - \lambda_\pm(k))^{-j+1}} -  \frac{b_{\pm,j,0}(k)}{ (\lambda - \lambda_{\pm,0}(k))^{-j+1}}\Big | \lesssim |k|^4.
$$
This gives
 $$
\Big|  \frac{1}{2\pi i} \sum_{j=0}^{N} \int_{\cC_\pm^*}  e^{\lambda t}  \Big[  \frac{b_{\pm,j}(k) }{(\lambda - \lambda_\pm(k))^{-j+1}}  - \frac{b_{\pm,j,0}(k)}{(\lambda - \lambda_{\pm,0}(k))^{-j+1}} \Big] \; d\lambda\Big| \lesssim  \int_{\cC_\pm^*}  |k|^2 \, d \lambda \lesssim |k|^3$$ 
as desired. 

Similarly, we now bound the integral of $|k|^n \partial_\lambda^n\TG_{k,1}(\lambda ) $, using the higher-order expansions in \eqref{exp-inverseGR}. We first compute 
\begin{equation}\label{exp-DGlambdak1}
\begin{aligned}
\partial_\lambda^n\TG_{k,1}(\lambda )  &= 
\sum_{j=0}^{N} \frac{d^n}{d\lambda^n}\Big[ \frac{b_{\pm,j}(k)}{ (\lambda - \lambda_\pm(k))^{-j+1}} -  \frac{b_{\pm,j,0}(k)}{ (\lambda - \lambda_{\pm,0}(k))^{-j+1}}\Big] +  \partial_\lambda^n\widetilde{\mathcal{R}}_{\pm}(\lambda,k) -   \partial_\lambda^n\widetilde{\mathcal{R}}_{\pm,0}(\lambda,k),
\end{aligned}\end{equation}
for any $1\le n\le N$. Using \eqref{coeff-GRa1}, we have 
$$
\Big| \frac{1}{2\pi }\int_{\cC_\pm} e^{i\tau t} |k|^n \Big[\partial_\lambda^n\widetilde{\mathcal{R}}_{\pm}(i\tau,k) -  \partial_\lambda^n\widetilde{\mathcal{R}}_{\pm,0}(i\tau ,k)\Big]  \; d\tau\Big| \lesssim |k|^{n+5}. 
$$
We next check the terms in the summation. By Cauchy's residue theorem, we have
$$
\begin{aligned}
\frac{(-1)^n}{2\pi i t^n } \int_{\{|\lambda \mp i \tau_*| =  |k|\}} e^{\lambda t}
\frac{d^n}{d\lambda^n}\Big( \sum_{j=0}^{N}\frac{b_{\pm,j}(k)}{ (\lambda - \lambda_\pm(k))^{-j+1}} \Big) \; d\lambda = e^{\lambda_\pm t} b_{\pm,0}(k)  
\\
\frac{(-1)^n}{2\pi i t^n } \int_{\{|\lambda \mp i \tau_*| =  |k|\}} e^{\lambda t}
\frac{d^n}{d\lambda^n}\Big( \sum_{j=0}^{N}\frac{b_{\pm,j,0}(k)}{ (\lambda - \lambda_{\pm,0}(k))^{-j+1}} \Big) \; d\lambda = e^{\lambda_{\pm,0} t} b_{\pm,0,0}(k) . 
\end{aligned}$$
Therefore, it remains to prove that 
\begin{equation}\label{Dl-gamma3aG} 
\Big|\frac{1}{2\pi i}  \int_{\cC_\pm^*}  e^{\lambda t} |k|^n\frac{d^n}{d\lambda^n}\Big[ \frac{b_{\pm,j}(k)}{ (\lambda - \lambda_\pm(k))^{-j+1}} -  \frac{b_{\pm,j,0}(k)}{ (\lambda - \lambda_{\pm,0}(k))^{-j+1}}\Big] \; d\lambda \Big| \lesssim |k|^3,
\end{equation}
for any $0\le j,n \le N$. We focus on the most singular term: namely, the term with $j=0$; the others are similar. Since $|\lambda \mp i \tau_*| = |k|$, we have $|\lambda- \lambda_\pm(k)| \ge |k|/2$ and $|\lambda- \lambda_{\pm,0}(k)| \ge |k|/2$. Therefore, on $\cC_\pm^*$, we bound
$$
\begin{aligned}
\Big |\frac{d^n}{d\lambda^n} \Big[ \frac{b_{\pm,0}(k) }{\lambda - \lambda_\pm(k)}  - \frac{b_{\pm,0,0}(k)}{\lambda - \lambda_{\pm,0}(k)} \Big] \Big | & =  n! 
\Big | \frac{b_{\pm,0}(k) }{(\lambda - \lambda_\pm(k))^{n+1}}  - \frac{b_{\pm,0,0}(k)}{(\lambda - \lambda_{\pm,0}(k))^{n+1}} \Big | 
\\& \lesssim \frac{ |b_{\pm,0}(k) - b_{\pm,0,0}(k)|}{|\lambda - \lambda_\pm(k) |^{n+1}}  +  \frac{ |\lambda_\pm(k) - \lambda_{\pm,0}(k)| }{|\lambda - \tilde\lambda_\pm(k)|^{n+2}} 
\end{aligned}$$
for some $\tilde\lambda_\pm(k)$ in between $\lambda_\pm(k)$ and $\lambda_{\pm,0}(k)$. Using \eqref{diff-lambdaG} and \eqref{coeff-GRa1}, the above fraction is bounded by $|k|^{-n+2}$ on $\cC_\pm^*$, and the estimates \eqref{Dl-gamma3aG} thus follow.  Combining the previous estimates, we have therefore obtained~\eqref{eq:nearsing}.

\end{proof}

\bigskip

At this stage of the proof, combining \eqref{int-FGdecomp-n} with \eqref{eq:G00}, \eqref{eq:G001}, Lemmas~\ref{lem:gamma1}, \ref{lem:gamma2} and \ref{lem:nearsing},
 we obtain \eqref{bd-Grk1} as claimed. The proof of Lemma~\ref{lem:residue01} is therefore complete.

\bigskip \noindent{\it Residue of $\TG_{k}(\lambda)$.}
Owing to the previous analysis, we can see that $b_{\pm,0}(k)$ is a function of $|k|^2$, which we denote as $\alpha_\pm(|k|^2)$. To ease readability, we set 
$a_\pm (k) := b_{\pm,0}(k)$, which gives the oscillatory term as stated in \eqref{decomp-FG}. 
We can see as well that
$$ 
a_\pm (|k|^2)  = \pm {2 \over i \tau_0}  + {\mathcal{O}_{k \to 0}}(|k|^2).
$$

\bigskip \noindent{\it Regularity of $\TG_{k}(\lambda)$ in $k$.}
Finally, we study the regularity of $\TG_{k}(\lambda)$ in $k$. In view of \eqref{fakeres-G0}, it follows that
\begin{equation}\label{est-dkFGtr0k}
\begin{aligned} 
\Big |\partial_k^\alpha [ |k|^{-2} e^{\mu_{-,0}t} \mathrm{Res} (\TG_{k,0}(\mu_{-,0})) ] \Big| \lesssim \langle t\rangle^{|\alpha|-1} e^{-\theta_0 |kt|}, \qquad \forall |\alpha|\ge 1,
\end{aligned}\end{equation}
for any $\theta_0<\sqrt{e_0}$. On the other hand, by definition \eqref{def-Kt}, we 
note that 
$$
\begin{aligned}| \partial_k^\alpha \partial_t^n [ |k|^2K_k(t)] | &\le \Big| \partial_k^\alpha \int e^{-ikt \cdot \hv } (k\cdot \hv)^{n+1} \varphi'(\langle v\rangle)\; dv \Big|
\\&\lesssim |k|^{n+1} \langle t\rangle^{|\alpha|} \langle kt\rangle^{-N}
\end{aligned}
$$
for any $n,\alpha$. Therefore, recalling \eqref{def-remainderRG} and using the above estimate with $n=4$, we bound 
\begin{equation}\label{dk-Rlambda1}
|\partial_k^\alpha [|k|^2\mathcal{R}(\lambda,k)]|  \lesssim \int_0^\infty |k|^5 \langle t\rangle^{|\alpha|} \langle kt\rangle^{-N} dt \lesssim \langle t\rangle^{|\alpha|}|k|^{4}.
\end{equation}
That is, the derivatives $\partial_k^\alpha$ cost a factor of $t^\alpha$. Hence, as above, we obtain the bounds on $\partial_k^\alpha \FG^r_k(t)$ as claimed.

\bigskip

\noindent {\bf The case $\eps< |k|\leq \kappa_0$.}
The analysis is mostly similar and as a matter of fact slightly simpler than in the low frequency case. Given a small parameter $\eta>0$, using again Cauchy's residue theorem, we have
 \begin{equation*}\label{int-FGdecomp**}
\FG_{k}(t) =  \frac{1}{2\pi i}\int_{\widetilde\Gamma}e^{\lambda t} \TG_{k}(\lambda)\; d\lambda,
\end{equation*}
where 
\begin{equation*}
\widetilde\Gamma = \widetilde\Gamma_0 \cup \widetilde\cC_\pm
\end{equation*}
having set 
$$
\begin{aligned}
 \widetilde\Gamma_0 &= \{ \lambda = i\tau, \quad |\tau \pm \tau_*(|k|)|\ge  \eta\} ,
 \\
  \widetilde\cC_\pm &= \{\Re \lambda \ge 0, \quad |\lambda \pm i\tau_*(|k|)| =  \eta \} .
    \end{aligned}$$
The parameter $\eta$ is chosen small enough so that for all  $\eps\leq |k|\leq \kappa_0 $,  there is no pole of $\TG_{k}(\lambda)$ other than $z_\pm(k)$  on the semi-disk $ \{\Re z\geq 0, |\lambda - z_\pm(k)| \leq \eta \}$. 

Recalling the definition of $\TG_k$ in~\eqref{eq:Gk}, we note that 
$ |\partial_\lambda^n\TG_k(i \tau) |\lesssim_{k} \frac{1}{\langle \tau\rangle^2}$ for $0\le n\le N$. Therefore, we have
$$
\Big|\int_{ \widetilde\Gamma_0} e^{i \tau t}  \partial_\lambda^n  \TG_k(i\tau) \; d\tau\Big|  \lesssim_k 1,$$
for $0\le n\le N$.  For what concerns the analysis on $\widetilde\cC_\pm $, we argue as  for $\cC_\pm $ in the previous low frequency analysis, by extracting the main contribution of the pole thanks to a Taylor expansion of $D$ at the points  $\pm i\tau_{*}(k)$
Note that since the partial derivative have continuous extension up to the boundaries this Taylor expansion exists 
and that it involves only powers of $(\lambda -  \pm i\tau_{*}(k))$ since the function is holomorphic for $\Re \lambda >0$.
 The estimates are however more straightforward since we do not need to track the dependence with respect to $k$. In particular we do not need to write the decomposition~\eqref{TG-decompose}.
The outcome is
$$
\int_{ \widetilde\cC_\pm} e^{\lambda t}  \partial_\lambda^n  \TG_k(\lambda) \; d\lambda = 
e^{z_\pm(k) t} \mathrm{Res}(\TG_k(z_\pm(k))) + \mathfrak{R}^{(n)}_k(t),$$
for $0\le n\le N$, where
$$
\left|\mathfrak{R}^{(n)}_k(t)\right| \lesssim_k 1.
$$

\bigskip

\noindent {\bf The case $\kappa_0\leq|k|<\kappa_0 +\delta$.}
In this region, using Corollary~\ref{coro:extenD}, we work with the extension non-holomorphic  extension of $D$ on the whole plane, which is denoted by $\widetilde{\widetilde{D}}$. We argue mostly as in the previous case.  For $\eta>0$ chosen small enough, using again Cauchy's residue theorem, we have
 \begin{equation*}\label{int-FGdecomp***}
\FG_{k}(t) =  \frac{1}{2\pi i}\int_{\widetilde\Gamma}e^{\lambda t} \TG_{k}(\lambda)\; d\lambda,
\end{equation*}
where 
\begin{equation*}
\widetilde{\widetilde\Gamma} = \widetilde{\widetilde\Gamma}_0 \cup \widetilde{\widetilde\cC}_\pm
\end{equation*}
having set 
$$
\begin{aligned}
\widetilde{\widetilde\Gamma}_0 &= \{ \lambda = i\tau, \quad |i\tau -\lambda^{\text{elec}}_\pm(k)|\ge  \eta\} ,
 \\
  \widetilde{\widetilde\cC}_\pm &= \{\Re \lambda \ge 0, \quad |\lambda -\lambda^{\text{elec}}_\pm(k)| =  \eta \},
    \end{aligned}$$
    and we argue as before. 
    
    To extract the poles by using a Taylor formula as previously, we observe that 
    originally, $\lambda_{\pm}^{elec}(k)$ are defined by studying the function $\widetilde{\widetilde D}$ which is not holomorphic, see Theorem \ref{theo-LangmuirE}, which means that the Taylor formula on the lign  between $\lambda \in   \widetilde{\widetilde\cC}_\pm$  
    and $\lambda_{\pm}^{elec}(k)$ should involve a polynomial expansion in terms of $(\lambda -  \lambda_{\pm}^{elec})^\alpha
    \overline{(\lambda -  \lambda_{\pm}^{elec})^\beta}$.
    Nevertheless, as soon as $|k| \neq \kappa_{0}$ we have that $\lambda_{\pm}^{elec}(k)$
    are in the domain where $\widetilde{\widetilde D}$ coincides with  $\tilde D$ which is an holomorphic extension of $D$, 
    see Remark \ref{remarkholo}.  By continuity of the derivatives  up to the boundary as stated in Corollary  \ref{coro:extenD}, 
    we thus  still have a Taylor formula around  $\lambda_{\pm}^{elec}(k)$ for every $k$, $\kappa_{0} \leq |k| \leq \kappa_{0} + \delta$
    which involves only powers of $(\lambda-\lambda_{\pm}^{elec}(k))$ (without $\overline{\lambda}$).  Moreover, since we start from an holomorphic function
    in $\Re \lambda >0$,  the remainder will be also an holomorphic function in $\Re \lambda >0$ so that we can indeed control it by moving the circles
    $ \widetilde{\widetilde\cC}_\pm$ onto the imaginary axis in the end.

\bigskip

\noindent {\bf Smoothness of $a_\pm(k)$.}
There remains to prove that $a_\pm(k)$ is a  $ \mathcal{C}^{K}$, $K:= \lfloor N_{0}/2 - 4 \rfloor$ function. This is  clear on the regions $|k|< \kappa_0$ and $|k|>\kappa_0$, separately. To conclude, we use the fact that  $a_\pm(k) = \alpha_\pm(|k|^2)$ is a radial function and it follows, by the previous constructions and by the continuity properties of $\widetilde{\widetilde{\Phi}}$ at the point $z= \pm 1$ (see the third item of Proposition~\ref{prop:D} and its proof) that
$$
\lim_{r \to \kappa_0^-} \alpha_\pm^{(k)} (r^2) = \lim_{r \to \kappa_0^+} \alpha_\pm^{(k)} (r^2) 
$$
for $ k=0, \ldots, K$. Therefore we deduce that $a_\pm(k)$ is a $\ \mathcal{C}^{K}$ function.

This concludes the proof of Proposition~\ref{prop-GreenG}.

\section{Proof of Proposition~\ref{prop-Green}}
\label{sec:proof2}

As in the proof of Proposition~\ref{prop-GreenG}, it is not possible to directly extend the resolvent kernel to a meromorphic function on a domain of the form $\{ \Re z > - \delta |k|\}$. However, it would be possible here to rely on the holomorphic extension of Corollary~\ref{coro:extenM}. The proof we present here does not use it, but rather follows the approach previously described for Proposition~\ref{prop-GreenG}.

Recall that 
\begin{equation}\label{rep-Mlambdas1}
M(\lambda,k) =  \lambda^2 + |k|^2  +\tau_0^2  +  \cL[N_k(t)](\lambda) ,
\end{equation}
where $N_k(t)$ decays rapidly in $t$, $|\partial_t^n N_k(t)|\le C_{n,N} |k|^{n+1} \langle kt \rangle^{-N}$, see \eqref{bounds-Nkt} in Lemma~\ref{lem:Nk}. Since $M(\lambda,k)$ is holomorphic in $\{\Re \lambda >0\}$, by Cauchy's residue theorem, we may move the contour of integration $\Gamma = \{\Re \lambda = \gamma_0\}$ towards the imaginary axis so that 
\begin{equation}\label{int-FHdecomp}
\FH_{k}(t) =  \frac{1}{2\pi i}\int_{\Gamma}e^{\lambda t} \TH_{k}(\lambda)\; d\lambda,
\end{equation}
for each $k\in \RR^3$, where we have decomposed $\Gamma$ into 
\begin{equation}\label{def-GammaM} 
\Gamma = \Gamma_1 \cup \Gamma_2 \cup \cC_\pm,
\end{equation}
having set 
$$
\begin{aligned}
 \Gamma_1 &= \{ \lambda = i\tau, \quad |\tau \pm  \nu_*(|k|)|\ge  |k|, \quad |\tau| > |k|\} ,
 \\
  \Gamma_2 &= \{\lambda = i\tau, \quad |\tau \pm  \nu_*(|k|)|\ge   |k|,\quad |\tau| \le  |k|\},
   \\
  \cC_\pm &= \{\Re \lambda \ge 0, \quad |\lambda \mp i  \nu_*(|k|)| =  |k|\}.
    \end{aligned}$$
Note that the semicircle $\cC_\pm$ is to avoid the singularity due to the poles $\nu_\pm= \pm i \nu_*(|k|)$ of $\TG_{k}(\lambda)$ established in Theorem \ref{theo-LangmuirB}, while the integrals over $\Gamma_1$ and $\Gamma_2$ are understood as taking the limit of $\Re \lambda \to 0^+$. To establish decay in $t$, we argue as for~\eqref{int-FGdecomp-n} to get for all $n=0,\ldots, N$
\begin{equation}\label{int-FHdecomp-n}
\FH_{k}(t) =  \frac{(-1)^n}{2\pi i t^n }\int_{\Gamma}e^{\lambda t} \partial_\lambda^n \TH_{k}(\lambda)\; d\lambda,
\end{equation}
relying on Corollary~\eqref{coro:LNK} to ensure that there is no boundary term.

\bigskip \noindent{\it Bounds on $ \cL[N_k(t)](\lambda)$.}
By Corollary~\ref{coro:LNK}, we have
\begin{equation}\label{bound-cLM1}
|\partial_\lambda^n \cL[N_k](\lambda)|\lesssim  \lesssim |k|^{-n},
\end{equation}
uniformly in $k\in \RR^3$ and $\Re \lambda \ge 0$. Clearly, the above estimate is not good for small $k$. However, by a small variant of Corollary~\ref{coro:LNK}, we also have for all $n \in \mathbb{N}$, 
\begin{equation}\label{bound-cLM2}
\begin{aligned}
|\partial_\lambda^n \cL[N_k(t)](\lambda) |
  &\lesssim |k|^2 |\lambda|^{-n-2},
\end{aligned}
\end{equation}
uniformly in $k\in \RR^3$ and $\Re \lambda \ge 0$. Observe that the above estimate is better than that of \eqref{bound-cLM1} in the case when $|\lambda| \gg |k|$. In addition, recalling \eqref{rep-Mlambdas1}, \eqref{bound-cLM2} shows that in this regime $\cL[N_k](\lambda) $ may be treated as a perturbation in $M(\lambda,k)$. For instance, when $\lambda$ is near the singularity $\pm i \nu_*(|k|)$ of the resolvent kernel $\TH_{k}(\lambda)$, where $\nu_*(|k|) \sim \sqrt{\tau_0^2 + |k|^2}$, $\cL[N_k(t)](\lambda) $ is of order $|k|^2 |\lambda|^{-2}$, which is of order $|k|^2$ for small $k$ and of order one for large $k$, both of which are a perturbation of $\lambda^2 + \tau_0^2 + |k|^2$, the leading term in $M(\lambda,k)$.

\bigskip \noindent{\it Further decomposition of $\FH_k(t)$.}
Proceeding as for the electric Green function, we further write 
\begin{equation}\label{TH-decompose}
\begin{aligned}
\TH_k(\lambda) 
&=\TH_{k,0}(\lambda) + \TH_{k,1}(\lambda ),
\end{aligned}
\end{equation}
where 
\begin{equation}\label{def-THplus}
\begin{aligned}
\TH_{k,0}(\lambda) & =   \frac{1}{ \lambda^2 + |k|^2  +\tau_0^2 },
\\
\TH_{k,1}( \lambda ) 
&= - \frac{ \cL[N_k(t)](\lambda)  }{( \lambda^2 + |k|^2  +\tau_0^2  )( \lambda^2 + |k|^2  +\tau_0^2  +  \cL[N_k(t)](\lambda) )} .
\end{aligned}
\end{equation}
We also denote by $\FH_{k,0}(t)$ and $\FH_{k,1}(t)$ the corresponding Green function, see \eqref{int-FHdecomp}. In what follows we denote by $\nu_{\pm,0}(k)$ and $\nu_{\pm}(k)$ the poles of $\TH_{k,0}(\lambda)$ and $\TH_{k}(\lambda)$, respectively. It follows that 
$\nu_{\pm,0}(k) = \pm i \nu_{*,0}(|k|)$ and $\nu_\pm (k)= \pm i \nu_*(|k|)$, with $\nu_{*,0}(|k|) =  \sqrt{\tau_0^2 + |k|^2}$ and $\nu_*(|k|)$ being defined as in Theorem \ref{theo-LangmuirB}.  We have

\begin{lem}
\label{lem:distance-poles}
There holds for all $k \in \RR^3$,
\begin{equation}\label{distance-poles}
\Big |\nu_\pm(k)- \nu_{\pm,0}(k)\Big| \le \frac{|k|}{2}.
\end{equation}
\end{lem}

\begin{proof}[Proof of Lemma~\ref{lem:distance-poles}]
 From~\eqref{eqs-xstar}, we have
 $$
 \nu_*^2 = \nu_{*,0}^2 + \widetilde\psi(\frac{|k|^2}{\nu_*^2}),
 $$
 with  $\widetilde\psi :=\psi -1$, where $\psi$ is defined in~\eqref{eq:formulaMtauleqk}. It follows that for all $|x|<1$,
 $\widetilde\psi(x) \geq 0$ and $\widetilde\psi (x) \leq \tau_0^2 \frac{|x|}{1-|x|}$. We deduce that $\nu_* \geq \nu_{*,0}$ and
 \begin{align*}
 \nu_* -  \nu_{*,0} &\leq \frac{1}{\nu_* +  \nu_{*,0}} \tau_0^2 \frac{|k|^2}{\nu_*^2-|k|^2} \leq \frac{1}{2}\frac{|k|^2}{\sqrt{|k|^2+\tau^2}},
 \end{align*}
 thus~\eqref{distance-poles} follows.
\end{proof}

We claim that the following lemma holds.
\begin{lem}
\label{lem:decompmag} We have the decompositions, uniformly with respect to $k \in \mathbb{R}^3$,
\begin{equation}
\label{eq:Htilde0}
\FH_{k,0}(t) = \sum_\pm e^{\nu_{\pm,0}t} \mathrm{Res} (\TH_{k,0}( \nu_{\pm,0})),
\end{equation}
\begin{equation}\label{bd-Hrk1}
\begin{aligned}
\Big| \FH_{k,1}(t)  &-  \sum_\pm e^{\nu_{\pm}t}\mathrm{Res} (\TH_{k}(\nu_{\pm}))  +  \sum_\pm e^{\nu_{\pm,0}t} \mathrm{Res} (\TH_{k,0}(\nu_{\pm,0}))   \Big| 
\\& \le C_1 |k|\langle k\rangle^{-2}\langle kt\rangle^{-N} + |k| \langle k^3t\rangle^{-N} \chi_{\{|\partial_t| \ll |k|\ll1\}}.
\end{aligned}
\end{equation}

\end{lem}

In view of the decomposition \eqref{TH-decompose}, this would complete the proof of Proposition \ref{prop-Green}, upon noting the residual $\mathrm{Res} (\TH_{k,0}(\nu_{\pm,0}))$ in \eqref{bd-Hrk1} is cancelled out with that in the above expression for $\FH_{k,0}(t)$.

\bigskip \noindent{\it Decomposition of $\FH_{k,0}$.}
By Lemma~\ref{lem:distance-poles}, the poles $\nu_{\pm,0}(k) $ are inside the disk $|\lambda \mp i \nu_*|\le  |k|$, and hence on the left of the contour of integration $\Gamma$. Therefore, since $\TH_{k,0}(\lambda)$ is meromorphic, it follows from Cauchy's residue theorem that 
$$
\begin{aligned} 
\FH_{k,0}(t)& =\frac{1}{2\pi i } \int_\Gamma e^{\lambda t} \TH_{k,0}(\lambda) \; d\lambda
\\&= \sum_\pm e^{\nu_{\pm,0}t} \mathrm{Res} (\TH_{k,0}( \nu_{\pm,0})) + \frac{1}{2\pi i}\lim_{\gamma_0 \to -\infty}\int_{\Re \lambda = \gamma_0}
e^{\lambda t} \TH_{k,0}(\lambda)\; d\lambda
\\&= \sum_\pm e^{\nu_{\pm,0}t} \mathrm{Res} (\TH_{k,0}( \nu_{\pm,0})) ,
\end{aligned}$$
upon noting that the last integral vanishes in the limit of $\gamma_0 \to -\infty$, since $\TH_{k,0}(\lambda) $ decays at order $\lambda^{-2}$ for $|\lambda|\to \infty$. We deduce that~\eqref{eq:Htilde0} holds.

\bigskip

Next, we prove \eqref{bd-Hrk1} via the representation \eqref{int-FHdecomp-n}. 

\bigskip \noindent{\it Bounds on $\TH_{k,1}(\lambda)$ on $\Gamma_1$.}
We start with the integral of $\TH_{k,1}(\lambda)$ on $\Gamma_1$: namely for $\lambda = i\tau$, where $|\tau|> |k|$ and $|\tau \pm \nu_*(|k|)| \ge |k|$. 

\begin{lem}
\label{lem:H1gamma1}
We have for all $0\le n \leq N$.
\begin{equation}
\label{eq:H1gamma1}
\Big| \frac{(-1)^n}{2\pi i t^n } \int_{\Gamma_1}e^{\lambda t} |k|^n \partial_\lambda^n \TH_{k,1}(\lambda)\; d\lambda \Big| \lesssim  |k| \langle k\rangle^{-4}. 
\end{equation}

\end{lem}

This gives the desired decay estimates as stated in \eqref{bd-Hrk1} (recalling that we only consider the case when $|kt|\ge 1$; the other case is done directly without integration by parts in $\lambda$).

\begin{lem}
\label{lem:H1gamma1bound}We have on $\Gamma_1$
\begin{equation}\label{eq:H1gamma1bound}
\begin{aligned}
|(|k|\partial_\lambda)^n \TH_{k,1}(i\tau)| 
& \lesssim  \frac{|k|^2}{(k^2 + 1)^2}\Big[ \frac{ 1}{ |\tau - \nu_*(|k|)|^2 + |k|^2} +  \frac{ 1}{ \tau^2 + |k|^2} \Big].
\end{aligned}
\end{equation}

\end{lem}

Lemma~\ref{lem:H1gamma1} is a consequence of Lemma~\ref{lem:H1gamma1bound}
using $\int_{\RR} (x^2 + |k|^2)^{-1}dx \le 4|k|^{-1}$; we indeed obtain 
$$
\Big|\int_{\{ |\tau| >|k|, \; |\tau \pm \nu_*| \ge  |k|\}} e^{i \tau t} \TH_{k,1}(i\tau) \; d\tau\Big| 
 \lesssim \frac{ |k|^2}{(|k|^2+ 1)^2} \int \frac{ 1}{ \tau^2 + |k|^2} \; d\tau \lesssim \frac{ |k|}{(|k|^2+ 1)^2}. 
$$

\begin{proof}[Proof of Lemma~\ref{lem:H1gamma1bound}]

By symmetry, it suffices to study the case when $\tau >0$. Recall that $\nu_{*,0}(k) =  \sqrt{\tau_0^2 + |k|^2}$. We recall that by \eqref{distance-poles}, we know that $|\nu_* - \nu_{*,0}|\le |k|/2$. Hence, for $|\tau - \nu_*| > |k|$, we have
$$
\begin{aligned} 
|\tau - \nu_{*,0}|
&\ge |\tau - \nu_*| - |\nu_* - \nu_{*,0}|
 \ge\frac12|\tau - \nu_*|
 \ge\frac14 (|\tau - \nu_*| + |k|).
\end{aligned}
$$
which yields 
\begin{equation}\label{low-M0}
\begin{aligned} 
\Big |\tau^2 - |k|^2  -\tau_0^2 \Big | 
\gtrsim (|\tau - \nu_*(|k|)| + |k|) (\tau+ |k|+ 1).
\end{aligned}
\end{equation}
We next derive a similar lower bound for $M(i\tau,k)$. Indeed, since $|\tau|> |k|$, we can use the formulation \eqref{eq:formulaMtauleqk} and \eqref{comp-tau0}, namely 
$$
M(i\tau,k) = -\tau^2 + |k|^2 +\tau_0^2 + \cL[N_k(t)](i\tau),
$$
where 
\begin{equation}\label{new-cLN} \cL[N_k(t)](i\tau) = - \frac12 \int_{-1}^1 \frac{|k|^2q(u)/\tau^2}{1- |k|^2 u^2/\tau^2}\;du = - \frac{|k|^2}{2} \int_{-1}^1 \frac{q(u)}{\tau^2- |k|^2 u^2}\;du.
\end{equation}
Clearly, $\cL[N_k(t)](i\tau) \ge 0$, and $\cL[N_k(t)](i\tau) \le  \frac{\tau_0^2|k|^2}{\tau^2 - |k|^2}$, recalling \eqref{comp-tau0}. In addition, for $|k| < \tau \le \nu_* - |k|$, it follows from \eqref{rangex} that $\tau^2 \le |k|^2 + \tau_0^2$. Therefore, 
$$
\begin{aligned}
|M(i\tau,k)| 
&= |k|^2  +\tau_0^2 + \cL[N_k(t)](i\tau) - \tau^2 
\ge  |k|^2  +\tau_0^2 - \tau^2 = \Big |\tau^2 - |k|^2  -\tau_0^2 \Big |.
\end{aligned}$$
On the other hand, for $|\tau|\ge \nu_* + |k|$, using again \eqref{rangex}, we note that $\tau^2 \ge \nu_*^2 + |k|^2 + 2 \nu_* |k|\ge 4|k|^2 + \tau_0^2$. Therefore, together with the above upper bound on $\cL[N_k(t)](i\tau)$, we have 
$$
\begin{aligned}
|M(i\tau,k)| 
&= \Big|\tau^2 - |k|^2  -\tau_0^2  -  \cL[N_k(t)](i\tau) \Big|
\ge  \Big |\tau^2 - |k|^2  -\tau_0^2 \Big |  -  \frac{\tau_0^2|k|^2}{\tau^2 - |k|^2}
\\
&\ge  \frac{2}{3}\Big |\tau^2 - |k|^2  -\tau_0^2 \Big | + |k|^2 -  \frac{\tau_0^2|k|^2}{\tau_0^2 + 3 |k|^2}
\ge  \frac{2}{3}\Big |\tau^2 - |k|^2  -\tau_0^2 \Big | .
\end{aligned}$$
This proves that $|M(i\tau,k)| $ satisfies the same lower bound as $|\tau^2 - |k|^2  -\tau_0^2|$ does, namely
\begin{equation}\label{low-M1}
\begin{aligned} 
|M(i\tau,k)| 
\gtrsim (|\tau - \nu_*(|k|)| + |k|) (\tau+ |k|+ 1).
\end{aligned}
\end{equation}
for $|\tau|> |k|$ and $|\tau \pm \nu_*(|k|)| \ge |k|$. Therefore, using \eqref{bound-cLM2} for $\lambda = i\tau$ with $|\tau|>|k|$, \eqref{low-M0}, and \eqref{low-M1}, we obtain 
$$ 
\begin{aligned}
|\TH_{k,1}(i\tau)| 
& \le \frac{ |\cL[N_k(t)](i\tau) | }{|(\tau^2 - |k|^2  -\tau_0^2)M(i\tau,k) |} \lesssim \frac{ |k|^2\tau^{-2} }{(\tau^2 - |k|^2  -\tau_0^2)^2} 
\\
& \lesssim \frac{ |k|^2}{ (|\tau - \nu_*(|k|)|^2 + |k|^2) (|k|^2+ 1)(\tau^2 + |k|^2)},
\end{aligned}$$
for $|\tau \pm \nu_*(|k|)| \ge |k|$. In addition, recalling $\nu_*(|k|)\ge \sqrt{|k|^2 + \tau_0^2} \ge \tau_0>0$, we must have either $|\tau - \nu_*|\gtrsim 1$ or $|\tau|\gtrsim 1$. This gives 
\begin{equation}\label{upbd-Hk1}
\begin{aligned}
|\TH_{k,1}(i\tau)| 
& \lesssim  \frac{|k|^2}{(k^2 + 1)^2}\Big[ \frac{ 1}{ |\tau - \nu_*(|k|)|^2 + |k|^2} +  \frac{ 1}{ \tau^2 + |k|^2} \Big].
\end{aligned}
\end{equation}

Similarly, we next bound the integral of $\partial_\lambda^n\TH_{k,1}(\lambda) $ on $\Gamma_1$. We shall prove that $|k| \partial_\lambda$ derivatives of $\TH_{k,1}(\lambda)$ satisfy the same estimates as does $\TH_{k,1}(\lambda)$. In view of  \eqref{def-THplus}, we shall check each term in the expression of $\TH_{k,1}(\lambda)$. Using \eqref{bound-cLM2}, we have $|k||\partial_\lambda\cL[N_k(t)](i\tau)|\lesssim |k|^2 \tau^{-2}$ for $|\tau|>|k|$, which is the same estimate that was used for $\cL[N_k(t)](i\tau)$. Next, we bound   
$$ |k| \Big|\frac{d}{d\lambda} \Big(\frac{1}{\lambda^2 + |k|^2  +\tau_0^2} \Big)_{\lambda = i\tau} \Big| \le \frac{2|\tau||k|}{(\tau^2- |k|^2  -\tau_0^2)^2} \lesssim  \frac{1}{|\tau^2- |k|^2  -\tau_0^2|} ,
$$
in which we used \eqref{low-M0} to bound $|\tau^2- |k|^2  -\tau_0^2|\gtrsim |k||\tau|$. Similarly, recalling \eqref{rep-Mlambdas1} and using \eqref{bound-cLM2}, we bound 
$$|k| |\partial_\lambda M(i\tau,k) | \lesssim |\tau| |k| + |k|^2\tau^{-2}$$
for $|\tau|>|k|$. Using this and \eqref{low-M1}, we bound 
$$ 
\frac{|k||\partial_\lambda M(i\tau,k)|}{|M(i\tau,k)|} \lesssim \frac{|\tau| |k| + |k|^2\tau^{-2}}{(|\tau - \nu_*(|k|)| + |k|) (\tau+ |k|+ 1)} \lesssim 1 + \frac{ |k|^2\tau^{-2}}{(|\tau - \nu_*(|k|)| + |k|) (\tau+ |k|+ 1)}.
$$
The last fraction is also bounded, since $|\tau|>|k|$ and either $|\tau - \nu_*|\gtrsim 1$ or $|\tau|\gtrsim 1$ (recalling $\nu_* \ge \tau_0>0$). 
This proves that $|k|\partial_\lambda \TH_{k,1}(\lambda) $ satisfies the same estimate as that for $\TH_{k,1}(\lambda)$. The estimates for higher derivatives follow similarly. This yields~\eqref{eq:H1gamma1}.

\end{proof}

\bigskip \noindent{\it Bounds on $\TH_{k,1}(\lambda)$ on $\Gamma_2$.}
Next, we consider the integral on $\Gamma_2$, where $\lambda = i\tau$ with $|\tau \pm  \nu_*(|k|)|\ge |k|$ and $|\tau | \le |k|$. We shall prove the following result.

\begin{lem}
\label{lem:H1gamma2}
We have for all $0\le n \leq N$.
\begin{equation}
\label{eq:H1gamma2}
\Big|\int_{\Gamma_2} e^{ i\tau t}  \TH_{k,1}(i\tau) \; d\tau \Big|\lesssim
\left \{ \begin{aligned} |k| \log (2 + |k|^{-2}) \quad& \quad \mbox{if}\quad |k|\le 1,
\\
 \langle k\rangle^{-3}
  \quad& \quad \mbox{if}\quad |k|\ge 1,
  \end{aligned}\right.
\end{equation}

\begin{equation}
\label{eq:H1gamma2deriv}
\begin{aligned}
\Big|\int_{\Gamma_2} e^{ i\tau t} |k|^{3n} \partial_\lambda^n \TH_{k,1}(i\tau) \; d\tau \Big|
&\lesssim  |k|  \quad& \quad \mbox{if}\quad |k|\le 1,
\\
\Big|\int_{\Gamma_2} e^{ i\tau t} |k|^{n} \partial_\lambda^n \TH_{k,1}(i\tau) \; d\tau \Big|
&\lesssim   \langle k\rangle^{-3}
  \quad& \quad \mbox{if}\quad |k|\ge 1.
  \end{aligned}
\end{equation}

\end{lem}

Lemma~\ref{lem:H1gamma2} is a consequence of the following  Lemma~\ref{lem:H1gamma2bound}. 

\begin{lem}
\label{lem:H1gamma2bound}
We have 
\begin{equation}
\label{eq:H1gamma2bound}
|\TH_{k,1}(i\tau)| 
 \lesssim \frac{ |k|}{ (|k|^2 + 1)(|k|^3 + |\tau|)},\end{equation}
 and for all $1\le n \leq N$,
\begin{equation}
\label{eq:H1gamma2bound2}
||k|^n\partial_\lambda^n\TH_{k,1}(i\tau)| 
 \lesssim \left(1+ \frac{ |k|}{ (|k|^3 + |\tau|)}\right)^n |\TH_{k,1}(i\tau)|.\end{equation}
\end{lem}
Indeed we have
$$
\begin{aligned}
\Big|\int_{\Gamma_2} e^{ i\tau t}  \TH_{k,1}(i\tau) \; d\tau \Big|
&\lesssim \int_0^{|k|} \frac{|k|}{(|k|^2 + 1)(|k|^3+ |\tau|)} \; d\tau \lesssim |k| \log (2 + |k|^{-2})
\end{aligned}$$
for $|k|\le 1$, while 
$$
\begin{aligned}
\Big|\int_{\Gamma_2} e^{ i\tau t}  \TH_{k,1}(i\tau) \; d\tau \Big|
&\lesssim \int_0^{|k|} \frac{|k|}{(|k|^2 + 1)(|k|^3+ |\tau|)} \; d\tau \lesssim \langle k\rangle^{-3}
\end{aligned}$$
for $|k|\ge 1$, whence \eqref{eq:H1gamma2bound} follows.
For $0\leq n <N-1$, observe that at $\tau=0$, the rhs of \eqref{eq:H1gamma2bound2} is of order $|k|^{-2}$. That is, $|k|\partial_\lambda \TH_{k,1}(i\tau)$ for large $k$ and $|k|^3 \partial_\lambda\TH_{k,1}(i\tau)$ for small $k$ satisfy the same estimates as $\TH_{k,1}(i\tau)$. Therefore, for $0\leq n<N-1$, we obtain for $|k|\ge 1$
$$
\Big|\int_{\Gamma_2} e^{ i\tau t} |k|^n \partial_\lambda^n \TH_{k,1}(i\tau) \; d\tau \Big|\lesssim \langle k\rangle^{-3},
$$
while for $|k|\le 1$, 
$$
\Big|\int_{\Gamma_2} e^{ i\tau t}  |k|^{3n} \partial_\lambda^n \TH_{k,1}(i\tau) \; d\tau \Big|\lesssim |k|,
$$
which gives the desired decay estimates as stated in \eqref{eq:H1gamma2bound2}. We stress that there was no log factor $\log (2 + |k|^{-2})$ in the above estimates, when $n\not =0$, since the worst term is of the form $(|k|^3 + |\tau|)^{-n-1}$; all other terms are better by a factor of $|k|^2$, for small $k$, which in particular absorbs the log factor.

\begin{proof}[Proof of Lemma~\ref{lem:H1gamma2bound}]

Recall the lower bound from \eqref{lower-Mstau},
\begin{equation}\label{low-Ml}
|M(i\tau,k) |\gtrsim (|k|^2 - \tau^2) + \frac{|\tau|}{|k|},
\end{equation}
uniformly for $|\tau| \le |k|$. In addition, note that $|\tau^2 - |k|^2  -\tau_0^2| \gtrsim |k|^2 -\tau^2 + 1$ for $|\tau|\le |k|$. It is therefore clear that whenever $|\tau|\le |k|/2$ or $|k|\lesssim 1$, we have 
\begin{equation}\label{low-onG2}|\tau^2 - |k|^2  -\tau_0^2| \gtrsim |k|^2 -\tau^2 + 1 \gtrsim |k|^2 + 1.
\end{equation}
On the other hand, when $|k|\gtrsim1$, we observe that the constraints $|\tau|\le |k|$ and $|\tau - \nu_*|\ge |k|$ (by the definition of the integration contour $\Gamma_2$), upon recalling \eqref{rangex}, imply that 
$$ |\tau| \le \nu_*(|k|) - |k| \le \frac{\tau_0^2 + \sqrt{\tau_0^4 + \tau_0^2 |k|^2}}{2( |k| + \sqrt{|k|^2 + \tau_0^2})} \lesssim 1.$$
In particular, when $|k|\gtrsim1$, since $\tau$ remains bounded, we also have $|k|^2 -\tau^2 + 1 \gtrsim |k|^2 + 1$. That is, the lower bound \eqref{low-onG2} is valid on $\Gamma_2$ for all values of $k\in \RR^3$. 

In addition, we can refine \eqref{low-Ml} to get 
\begin{equation}\label{low-reMl}
|M(i\tau,k) |\gtrsim |k|^2 + \frac{|\tau|}{|k|},
\end{equation}
uniformly on $\Gamma_2$. 
Indeed, \eqref{low-reMl} is direct from \eqref{low-Ml} when $|\tau|\le |k|/2$. On the other hand, when $|k|/2\le |\tau|\le |k|$, we have from \eqref{low-Ml} that $|M(i\tau,k) |\gtrsim |k|^2 - \tau^2 + 1+ \frac{|\tau|}{|k|} $, which again gives \eqref{low-reMl}, upon using \eqref{low-onG2}.

We are now ready to bound $\TH_{k,1}(i\tau)$. Indeed, using \eqref{low-onG2}, \eqref{low-reMl}, and \eqref{bound-cLM1} which holds for $|\tau|\le |k|$, we bound 
$$ 
\begin{aligned}
|\TH_{k,1}(i\tau)| 
& \le \frac{ |\cL[N_k(t)](i\tau) | }{|(\tau^2 - |k|^2  -\tau_0^2)M(i\tau,k) |}
 \lesssim \frac{ |k|}{ (|k|^2 + 1)(|k|^3 + |\tau|)},
\end{aligned}$$
which corresponds to~\eqref{eq:H1gamma2bound}.

Let us next bound the $\partial_\lambda$ derivatives. In view of \eqref{bound-cLM1}, $|k|\partial_\lambda\cL[N_k(t)](i\tau)$ satisfies the same bounds as does $\cL[N_k(t)](i\tau)$: namely, $|k|\partial_\lambda\cL[N_k(t)](i\tau)\lesssim 1$. On the other hand, using \eqref{low-onG2}, we have 
$$ 
 \Big|\Big(\frac{|k|\partial_\lambda (\lambda^2 + |k|^2  +\tau_0^2)}{\lambda^2 + |k|^2  +\tau_0^2} \Big)_{\lambda = i\tau} \Big| \le \frac{2|\tau| |k|}{|\tau^2- |k|^2  -\tau_0^2|} \lesssim  \frac{|\tau| |k|}{|k|^2 + 1} \lesssim 1,
$$
for $|\tau|\le |k|$. Similarly, using again \eqref{bound-cLM1}, we bound 
$$
\begin{aligned} 
 ||k| \partial_\lambda M(i\tau,k) |
&\le  
2 | \tau||k| + |k||\partial_\lambda \cL[N_k(t)](i\tau )| 
\lesssim |\tau| |k|+ 1
\end{aligned}$$ 
and so, using \eqref{low-reMl}, we have
$$\frac{|k||\partial_\lambda M(i\tau,k) |}{|M(i\tau,k) | } \lesssim \frac{   |\tau| |k|^2 + |k|}{ |k|^3+ |\tau|} \lesssim 1 + \frac{|k|}{ |k|^3 + |\tau|}.$$
We argue similarly for higher derivatives, so that the claimed~\eqref{eq:H1gamma2bound} follows.

\end{proof}

\begin{remark}\label{rem-k3t} Note that the decay estimates in $|k|^3t$ in the regime when $|\tau|\ll |k|\ll1$ appears sharp. Indeed, using \eqref{new-Msmall}, we may compute 
\begin{equation}\label{dlambda-M}
\partial_\lambda M(\lambda,k) =  2\lambda - \frac{1}{2} \int \frac{ |\mP_k\hv|^2 }{\lambda +  ik \cdot \hat v} \varphi'(\langle v\rangle)  dv + \frac{\lambda }{2} \partial_\lambda \Big( \int \frac{ |\mP_k\hv|^2 }{\lambda +  ik \cdot \hat v} \varphi'(\langle v\rangle)  dv\Big),
\end{equation}
where, 
by the Plemelj formula, we have 
$$
\lim_{\gamma\to 0^+}\frac{1}2 \int \frac{ |\mP_k\hv|^2 }{\gamma + i\tau  +  ik \cdot \hat v} \varphi'(\langle v\rangle)  dv =  \frac{i}{|k|} P.V. \int_{-1}^1 \frac{1}{u + \tau / |k|} q(u) du  - \frac{\pi}{|k|} q(-\tau/|k|),
$$
which is of order $|k|^{-1}$ due to the last term, since $q(0) \not =0$. On the other hand, in view of \eqref{low-reMl}, $M(i\tau,k)$ is of order $(|\tau|+|k|^3) /|k|$, giving the scaling $|k|^3\partial_\lambda$ or decay in $|k|^3t$ as stated. 
\end{remark}

\bigskip \noindent{\it Bounds on $\TH_{k,1}(\lambda)$ on $\cC_\pm$.}
Finally, we study the case when $\lambda$ near the singularity of $\TH_k(\lambda)$: namely, when $\lambda$ is on the semicircle $|\lambda \pm i\nu_*(|k|)| = |k|$ with $\Re \lambda \ge 0$. 

\begin{lem}
\label{lem:nearsingmag}
We have for all $0\le n\le N$, the decomposition
\begin{equation}
\label{eq:nearsingmag}
\begin{aligned}
 \frac{(-1)^n}{2\pi i (|k|t)^n}\int_{\cC_\pm}e^{\lambda t} |k|^n\partial_\lambda^n\TH_{k,1}(\lambda)  \; d\lambda 
 &=  \sum_\pm  e^{\nu_\pm t}  \mathrm{Res} (\TH_{k}(\nu_{\pm})) 
 -  \sum_\pm e^{\nu_{\pm,0} t} \mathrm{Res} (\TH_{k,0}(\nu_{\pm,0}))\\
 &+ \cO(|k| \langle k\rangle^{-2}\langle kt\rangle^{-n}).
 \end{aligned}
 \end{equation}
 \end{lem}

\begin{proof}[Proof of Lemma~\ref{lem:nearsingmag}]

Recalling \eqref{rangex}, 
we first observe that for $\lambda \in \{|\lambda \pm i \nu_*(|k|)|\le |k|\}$, there holds the following: 

\begin{itemize} 

\item For $|k|\le 1$, since $\nu_* \ge \sqrt{\tau_0^2 +|k|^2}$, we have $c_0 \le |\lambda|\le C_0$. In particular, using \eqref{bound-cLM2}, we obtain 
\begin{equation}\label{unibd-DMlambda}
|\partial_\lambda^n \cL[N_k(t)](\lambda) |
\lesssim |k|^2
\end{equation}
for $n\ge 0, \Re \lambda \ge 0$. 

\item For $|k|\ge 1$, using \eqref{bound-cLM1}, we obtain 
\begin{equation}\label{unibd-DMlambda1}
|\partial_\lambda^n \cL[N_k(t)](\lambda) |
\lesssim \langle k\rangle^{-n}
\end{equation}
for $n\ge 0, \Re \lambda \ge 0$. 

\end{itemize}

We now bound $\TH_k(\lambda)$. Recall from the decomposition \eqref{TH-decompose} that $\TH_{k,1}(\lambda ) = \TH_{k}(\lambda ) - \TH_{k,0}(\lambda )$, 
which reads
\begin{equation}\label{recomp-H1}
\begin{aligned}
\TH_{k,1}(\lambda ) = \frac{ 1 }{\lambda^2 + |k|^2  +\tau_0^2  +  \cL[N_k(t)](\lambda) }  - \frac{ 1}{\lambda^2 + |k|^2  +\tau_0^2} .
\end{aligned}\end{equation}
Let $\nu_\pm(k) = \pm i \nu_*(|k|)$ and $\nu_{\pm,0}(|k|) = \pm i\nu_{*,0}(k)$ be the poles of the Green kernels 
$\TH_k(\lambda)$ and $\TH_{k,0}(\lambda) $, respectively. Using \eqref{unibd-DMlambda} and \eqref{unibd-DMlambda1} for small and large $k$ respectively, and the fact that $\cL[N_k(t)](i\tau) \ge 0$ (see, e.g., \eqref{new-cLN}), we obtain 
\begin{equation}\label{diff-lambda} 
|\nu_\pm(k) - \nu_{\pm,0}(k)|\le \frac{| \cL[N_k(t)](i\nu_*)|}{2\sqrt{|k|^2 + \tau_0^2}} \lesssim |k|^2 \langle k\rangle^{-3}, 
\end{equation}
uniformly in $k\in\RR^3$. Note that the above estimate  gives a bound of order $\langle k\rangle^{-1}$ for $|k|\gg1$ and is therefore better than that in \eqref{distance-poles} for large $k$. 

Next, we shall use  Taylor's expansion to expand the kernel $\TH_{k,1}(\lambda )$ near the poles. Indeed, we first write 
$$
\begin{aligned}
\lambda^2 + |k|^2  +\tau_0^2 = 2 \nu_{\pm,0} (\lambda - \nu_{\pm,0}) + (\lambda - \nu_{\pm,0})^2
\end{aligned}$$
and therefore, since $2 \nu_{\pm,0} = \pm2 i \sqrt{|k|^2 + \tau_0^2}$, 
we obtain 
\begin{equation}\label{exp-inverseR1}
\begin{aligned}
 \frac{ 1 }{\lambda^2 + |k|^2  +\tau_0^2 } 
& = \frac{1}{\lambda - \nu_{\pm,0}}
\sum_{j=0}^{N} b_{\pm,j,0}(k) (\lambda - \nu_{\pm,0})^j+  \widetilde{\mathcal{R}}_{\pm,0}(\lambda,k),
\end{aligned}
\end{equation}
where $b_{\pm,j,0}(k)  = (-1)^j (2\nu_{\pm,0})^{-j-1}$, and 
$$|\partial_\lambda^j\widetilde{\mathcal{R}}_{\pm,0}(\lambda,k)|\lesssim |\nu_{\pm,0}|^{-N-1}|\lambda - \nu_{\pm,0}|^{N-j-1} \lesssim \langle k\rangle^{-N-1}|\lambda - \nu_{\pm,0}|^{N-j-1} .$$
Similarly, for  $\lambda \in \{|\lambda - \nu_\pm(k)| \le |k|\}$ with $\Re \lambda \ge 0$, we write 
$$
\begin{aligned}
\lambda^2 + |k|^2  +\tau_0^2  +  \cL[N_k(t)](\lambda)  &=\sum_{j=1}^{N}a_{\pm,j}(k) (\lambda - \nu_\pm)^j+   \mathcal{R}_{\pm}(\lambda,k) ,
\end{aligned}$$
where, using \eqref{unibd-DMlambda}-\eqref{unibd-DMlambda1}, we have  
\begin{equation}\label{coeff-Ra1} 
|a_{\pm,1}(k)  -  2 \nu_{\pm,0}| \lesssim |k|^2 \langle k\rangle^{-3}, \qquad |a_{\pm,2} - 1|\lesssim |k|^2 \langle k\rangle^{-4}, \qquad |a_{\pm,j}|\lesssim |k|^2 \langle k\rangle^{-j-2},
\end{equation}
for $n\ge 3$, and 
\begin{equation}\label{coeff-Ra2} 
| \partial_\lambda^j \mathcal{R}_{\pm}(\lambda,k) |\lesssim |k|^2 \langle k\rangle^{-N-2} |\lambda - \nu_\pm|^{N - j}.
\end{equation} 
This yields
\begin{equation}\label{exp-inverseR2}
\begin{aligned}
 \frac{ 1 }{\lambda^2 + |k|^2  +\tau_0^2  +  \cL[N_k(t)](\lambda) }  &= 
 \frac{1}{\lambda - \nu_\pm(k)}
\sum_{j=0}^{N} b_{\pm,j}(k) (\lambda - \nu_\pm(k))^j+  \widetilde{\mathcal{R}}_{\pm}(\lambda,k),
\end{aligned}\end{equation}
where the coefficients $b_{\pm,j}(k)$ can be computed in terms of $a_{\pm,j}(k)$. Note that $b_{\pm,1} = \frac{1}{a_{\pm,1}}$, which is of order $\langle k\rangle^{-1}$. Using \eqref{diff-lambda} and \eqref{coeff-Ra1}-\eqref{coeff-Ra2}, we obtain 
\begin{equation}\label{coeff-Ra}
\begin{aligned}
 |b_{\pm,j}(k) - b_{\pm,j,0}(k)  | &\lesssim |k|^2 \langle k\rangle^{-j-4},  
 \\
|\partial_\lambda^j \widetilde{\mathcal{R}}_{\pm}(\lambda,k) - \partial_\lambda^j \widetilde{\mathcal{R}}_{\pm,0}(\lambda,k) | &\lesssim |k|^2 \langle k\rangle^{-j-4},
 \end{aligned}\end{equation}
for $0\le j\le N$. Note in particular that by construction, 
\begin{equation}\label{def-bn0}
b_{\pm,0,0}(k) =  \mathrm{Res} (\TH_{k,0}(\nu_{\pm,0}))  , \qquad b_{\pm,0}(k) = \mathrm{Res} (\TH_{k}(\nu_{\pm})).
\end{equation}

We are now ready to bound $\TH_{k,1}(\lambda ) $, using \eqref{recomp-H1} and the above expansions. Indeed, using \eqref{exp-inverseR1} and  \eqref{exp-inverseR2}, we write
\begin{equation}\label{exp-Hlambdak1}
\begin{aligned}
\TH_{k,1}(\lambda )  &= 
\sum_{j=0}^{N} \Big[ \frac{b_{\pm,j}(k)}{ (\lambda - \nu_\pm(k))^{-j+1}} -  \frac{b_{\pm,j,0}(k)}{ (\lambda - \nu_{\pm,0}(k))^{-j+1}}\Big] +  \widetilde{\mathcal{R}}_{\pm}(\lambda,k) -  \widetilde{\mathcal{R}}_{\pm,0}(\lambda,k),
\end{aligned}\end{equation}
where, using \eqref{coeff-Ra}, the remainder satisfies $|  \widetilde{\mathcal{R}}_{\pm}(\lambda,k) -  \widetilde{\mathcal{R}}_{\pm,0}(\lambda,k) |\lesssim|k|^2 \langle k\rangle^{-4}$ for $\lambda \in \{|\lambda - \nu_\pm|\le |k|\}$ with $\Re \lambda \ge 0$. Since the terms in the sum above are meromorphic in $\mathbb{C}$, we may apply the Cauchy's residue theorem to deduce 
$$
\begin{aligned}
 \frac{1}{2\pi i}\int_{\cC_\pm}e^{\lambda t} \TH_{k,1}(\lambda)  \; d\lambda &=  \sum_\pm  e^{\nu_\pm t} b_{\pm,0}(k)  -  \sum_\pm e^{\nu_{\pm,0} t} b_{\pm,0,0}(k) 
  \\&\quad + 
 \frac{1}{2\pi i}  \sum_{j=0}^{N}  \int_{\cC_\pm^*}  e^{\lambda t} \Big[ \frac{b_{\pm,j}(k)}{ (\lambda - \nu_\pm(k))^{-j+1}} -  \frac{b_{\pm,j,0}(k)}{ (\lambda - \nu_{\pm,0}(k))^{-j+1}}\Big]  \; d\lambda
\\&\quad  + \frac{1}{2\pi i }\int_{\cC_\pm} e^{\lambda t} \Big[ \widetilde{\mathcal{R}}_{\pm}(\lambda,k) -  \widetilde{\mathcal{R}}_{\pm,0}(\lambda ,k)\Big]  \; d\lambda
 \end{aligned}$$
 where $\cC_\pm^*$ denotes the semicircle $|\lambda \mp i \nu_*(|k|) | = |k|$ with $\Re \lambda \le 0$. 
 Since $| \widetilde{\mathcal{R}}_{\pm}(i\tau,k) -  \widetilde{\mathcal{R}}_{\pm,0}(i\tau ,k) |\lesssim |k|^2 \langle k\rangle^{-4}$, Cauchy's residue theorem and a continuity argument ensures that the last integral term in the above sum is  bounded by 
 $C_0 |k|^2\langle k\rangle^{-3}$. As for the integral on $\cC_\pm^*$, recalling \eqref{diff-lambda}, we have 
 $|\lambda- \nu_\pm(k)| \ge |k|/2$ and $|\lambda- \nu_{\pm,0}(k)| \ge |k|/2$ on $\cC_\pm^*$. Therefore, together with \eqref{diff-lambda} and \eqref{coeff-Ra}, we bound for $\lambda \in \cC_\pm^*$, 
$$\Big |  \frac{b_{\pm,0}(k) }{\lambda - \nu_\pm(k)}  - \frac{b_{\pm,0,0}(k)}{\lambda - \nu_{\pm,0}(k)} \Big | \lesssim \frac{|k|^2 \langle k\rangle^{-4}}{|\lambda - \nu_\pm(k) ||\lambda - \nu_{\pm,0}(k)|} \lesssim \langle k\rangle^{-4},$$
and 
$$  \sum_{j=1}^{N} \Big | \frac{b_{\pm,j}(k)}{ (\lambda - \nu_\pm(k))^{-j+1}} -  \frac{b_{\pm,j,0}(k)}{ (\lambda - \nu_{\pm,0}(k))^{-j+1}}\Big | \lesssim |k|^2\langle k\rangle^{-5},
$$
both of which are bounded by $C_0 \langle k\rangle^{-3}$. This gives
 $$
\Big|  \frac{1}{2\pi i}  \sum_{j=0}^{N}  \int_{\cC_\pm^*}  e^{\lambda t} \Big[ \frac{b_{\pm,j}(k)}{ (\lambda - \nu_\pm(k))^{-j+1}} -  \frac{b_{\pm,j,0}(k)}{ (\lambda - \nu_{\pm,0}(k))^{-j+1}}\Big]  \; d\lambda\Big| \lesssim  \int_{\cC_\pm^*}  \langle k\rangle^{-3} d|\lambda|\lesssim |k|\langle k\rangle^{-3}$$ 
as desired. 

Similarly, we now bound the integral of $|k|^n \partial_\lambda^n\TH_{k,1}(\lambda ) $, using the higher-order expansions in \eqref{exp-inverseR1} and \eqref{exp-inverseR2}. First, from \eqref{exp-Hlambdak1}, we compute 
\begin{equation}\label{exp-DHlambdak1}
\begin{aligned}
\partial_\lambda^n\TH_{k,1}(\lambda )  &= 
\sum_{j=0}^{N} \frac{d^n}{d\lambda^n}\Big[ \frac{b_{\pm,j}(k)}{ (\lambda - \nu_\pm(k))^{-j+1}} -  \frac{b_{\pm,j,0}(k)}{ (\lambda - \nu_{\pm,0}(k))^{-j+1}}\Big] +  \partial_\lambda^n\widetilde{\mathcal{R}}_{\pm}(\lambda,k) -   \partial_\lambda^n\widetilde{\mathcal{R}}_{\pm,0}(\lambda,k),
\end{aligned}\end{equation}
for any $0\le n\le N$. Using \eqref{coeff-Ra}, we have 
$$
\Big| \frac{1}{2\pi }\int_{|\tau \pm \nu_*|\le  |k|} e^{i\tau t} |k|^n \Big[\partial_\lambda^n\widetilde{\mathcal{R}}_{\pm}(i\tau,k) -  \partial_\lambda^n\widetilde{\mathcal{R}}_{\pm,0}(i\tau ,k)\Big]  \; d\tau\Big| \lesssim |k|^{3}\langle k\rangle^{-4}
. 
$$
We next check the terms in the summation. By Cauchy's residue theorem, we have
$$
\begin{aligned}
\frac{(-1)^n}{2\pi i t^n } \int_{\{|\lambda \mp i\nu_*| = |k|\}} e^{\lambda t}
\frac{d^n}{d\lambda^n}\Big( \sum_{j=0}^{N}\frac{b_{\pm,j}(k)}{ (\lambda - \nu_\pm(k))^{-j+1}} \Big) \; d\lambda = e^{\nu_\pm t} b_{\pm,0}(k)  
\\
\frac{(-1)^n}{2\pi i t^n } \int_{\{|\lambda \mp i\nu_*| =  |k|\}} e^{\lambda t}
\frac{d^n}{d\lambda^n}\Big( \sum_{j=0}^{N}\frac{b_{\pm,j,0}(k)}{ (\lambda - \nu_{\pm,0}(k))^{-j+1}} \Big) \; d\lambda = e^{\nu_{\pm,0} t} b_{\pm,0,0}(k) ,
\end{aligned}$$
for any $n\ge 0$. 
Therefore, it remains to prove that 
\begin{equation}\label{Dl-gamma3a} 
\Big|\frac{1}{2\pi i}  \int_{\cC_\pm^*}  e^{\lambda t} |k|^n\frac{d^n}{d\lambda^n}\Big[ \frac{b_{\pm,j}(k)}{ (\lambda - \nu_\pm(k))^{-j+1}} -  \frac{b_{\pm,j,0}(k)}{ (\lambda - \nu_{\pm,0}(k))^{-j+1}}\Big] \; d\lambda \Big| \lesssim |k| \langle k\rangle^{-3},
\end{equation}
for any $0\le j,n \le N$. We focus on the most singular term: namely, the term with $j=0$; the others are similar. Since $|\lambda -\nu_\pm| = |k|$, using \eqref{distance-poles}, we have $|\lambda- \nu_{\pm,0}(k)| \ge |k|/2$. Therefore, on $\cC_\pm^*$, we bound
$$
\begin{aligned}
\Big |\frac{d^n}{d\lambda^n} \Big[ \frac{b_{\pm,0}(k) }{\lambda - \nu_\pm(k)}  - \frac{b_{\pm,0,0}(k)}{\lambda - \nu_{\pm,0}(k)} \Big] \Big | & =  n! 
\Big | \frac{b_{\pm,0}(k) }{(\lambda - \nu_\pm(k))^{n+1}}  - \frac{b_{\pm,0,0}(k)}{(\lambda - \nu_{\pm,0}(k))^{n+1}} \Big | 
\\& \lesssim \frac{ |b_{\pm,0}(k) - b_{\pm,0,0}(k)|}{|\lambda - \nu_\pm(k) |^{n+1}}  +  \frac{ |\nu_\pm(k) - \nu_{\pm,0}(k)| }{|\lambda - \tilde\nu_\pm(k)|^{n+2}} 
\end{aligned}$$
for some $\tilde\nu_\pm(k)$ in between $\nu_\pm(k)$ and $\nu_{\pm,0}(k)$. Using \eqref{diff-lambda} and \eqref{coeff-Ra}, the above fraction is bounded by $|k|^{-n}\langle k\rangle^{-3}$, and the estimates \eqref{Dl-gamma3a} thus follow. Combining, we have therefore obtained 
$$
\begin{aligned}
 \frac{(-1)^n}{2\pi i (|k|t)^n}\int_{\cC_\pm}e^{\lambda t} |k|^n\partial_\lambda^n\TH_{k,1}(\lambda)  \; d\lambda 
 &=  \sum_\pm  e^{\nu_\pm t} b_{\pm,0}(k)  -  \sum_\pm e^{\nu_{\pm,0} t} b_{\pm,0,0}(k) + \cO(|k| \langle k\rangle^{-2}\langle kt\rangle^{-n}),
 \end{aligned}$$
for any $0\le n\le N$. Recalling \eqref{int-FHdecomp-n} and \eqref{def-bn0}, we obtain \eqref{bd-Hrk1} as claimed.

\end{proof}


\bigskip \noindent{\it Residue of $\TH_{k}(\lambda)$.}
From the expansion \eqref{exp-inverseR2}, we have 
$$ b_\pm(k) := b_{\pm,0}(k) =  \frac{1}{2\nu_\pm(k) + \partial_\lambda  \cL[N_k(t)](\nu_\pm) }.
$$
Recall that $\nu_\pm(k) \sim \sqrt{|k|^2 + \tau_0^2}$, while $ |\partial_\lambda  \cL[N_k(t)]|\lesssim |k|^2 \langle k\rangle^{-3}$ in view of \eqref{unibd-DMlambda}-\eqref{unibd-DMlambda1}. This yields the behavior of the oscillatory term for large $|k|$ as stated in \eqref{decomp-FH}. 

\bigskip \noindent{\it Regularity of $\TH_{k}(\lambda)$ in $k$.}
Finally, we study the regularity of $\TH_{k}(\lambda)$ in $k$ and prove \eqref{bound-Hrk} for $\alpha \not =0$. Indeed, by definition, we note 
$$
\begin{aligned}| \partial_k^\alpha  [N_k(t)] | &\le \Big| \partial_k^\alpha \int e^{-ikt \cdot \hv } (k\cdot \hv) |\mP_k\hv|^2 \varphi'(\langle v\rangle)\; dv \Big|
\\&\lesssim ( |k| t^{|\alpha|} + t^{|\alpha|-1})\langle kt\rangle^{-N} 
\lesssim |k|^{1-|\alpha|} \langle kt\rangle^{-N+|\alpha|},
\end{aligned}
$$
for any $1\le |\alpha| <N$. That is, $|k|^{|\alpha|}\partial_k^\alpha N_k(t)$ satisfy similar decay estimates as those for $N_k(t)$, see \eqref{bounds-Nkt}. 
Therefore, we bound 
\begin{equation}\label{dk-Rlambda1-M}
||k|^{|\alpha|}\partial_k^\alpha  \cL[N_k(t)](\lambda) |  \lesssim \int_0^\infty |k| \langle kt\rangle^{-N+|\alpha|} dt \lesssim 1,
\end{equation}
uniformly in $\Re \lambda \ge 0$. Hence, following  similar lines as above, we obtain the bounds on $|k|^{|\alpha|}\partial_k^\alpha \FH_k(t)$ as claimed.

This eventually concludes the proof of  Proposition~\ref{prop-Green}.

\section{Littlewood-Paley decomposition and Besov spaces}
\label{sec:LittlewoodPaley}
In this paper we use the homogeneous Littlewood-Paley decomposition on $\mathbb{R}^{d}$, $d \in \mathbb{N}$. There exists (see e.g. \cite{BCD}) a smooth cutoff function $\varphi\in [0,1]$ that is compactly supported in the annulus $ { 1 \over 4} \leq |k | \leq 4$ and equal to one in the inner annulus  $ { 1 \over 2} \leq |k | \leq 2$ and such that 
$$
\sum_{k \in \mathbb{Z}} \varphi(2^{-k} \cdot ) = 1.$$
For any tempered distribution $h$, its homogeneous Littlewood Paley decomposition reads as 
\begin{equation}\label{def-LP}
h= \sum_{k \in \mathbb{Z}} P_k h, 
\end{equation}
where $P_k$ denotes the Littlewood-Paley projection on the dyadic interval $[2^{k-1}, 2^{k+1}]$, whose Fourier transform is given by $\widehat{P_k h}= \varphi(2^{-k}\cdot ) \Fh $.

We recall the definition of (inhomogeneous) Besov spaces.
It is a well-known fact that this definition is independent of the choice of the function $\varphi$ satisfying the above properties.

\begin{defi}
Let  $s \in \mathbb{R}$, $p,q \in [1,+\infty]$. The Besov space $B^{s}_{p,q}$ is defined as
$$
B^{s}_{p,q} =  \left\{ h \in \mathcal{S}',  \left\| \sum_{k\leq -1} P_k h \right\|_{L^q} +\left\| (2^{s  k}\| P_k h\|_{L^q})_{k \in \mathbb{N}}\right\|_{\ell^p(\mathbb{N})} < +\infty  \right\}.
$$
\end{defi}

Let us also recall some classical Bernstein lemmas (we refer again to \cite{BCD}).

\begin{lem}
\label{lem:Bernstein1}
 There exists $C_0>0$ such that, for all $p \in [1,\infty]$, for all $n \in \mathbb{N}$ and all $k \in \mathbb{Z}$,
 $$
 C_0 2^{kn} \| h\|_{L^p} \leq \| P_k \nabla^n h \|_{L^p} \leq C_0^{-1}  2^{kn} \| h\|_{L^p}.
 $$
\end{lem}

\begin{lem}
\label{lem:Bernstein2}
Let $m\in \mathbb{R}$ and $r >0$.
Let $\sigma \in \mathcal{C}^{\lfloor d/2 \rfloor +1}(\mathbb{R}^d \setminus\{0\})$ such that for all $\alpha \in \mathbb{N}^d$ with $|\alpha|\leq {\lfloor d/2 \rfloor +1}$, there exists $C_\alpha>0$ such that for all $\xi \in \mathbb{R}^d$, $|\partial^\alpha \sigma(\xi)| \leq C_\alpha |\xi|^{m-|\alpha|}\langle \xi \rangle^{-r}$. Then there exists $C>0$ depending only on the $C_\alpha$ such that for all $p \in [1,+\infty]$, and all $k \in \mathbb{Z}$
$$
\| P_k( \sigma(i\partial) h\|_{L^p} \leq C 2^{km} \langle 2^{k} \rangle^{-r} \|h \|_{L^p}.
$$
\end{lem}

As an application, let us provide a continuity result related to Fourier multipliers with decaying symbol.

\begin{lem}
\label{lem:fouriermult}
 Let $\delta \in (0,1]$. Let $\sigma\in \mathcal{C}^{\lfloor d/2 \rfloor +1}(\mathbb{R}^d \setminus\{0\})$ such that for all $\alpha \in \mathbb{N}^d$ with $|\alpha|\leq {\lfloor d/2 \rfloor +1}$, there exists $C_\alpha>0$ such that for all $\xi \in \mathbb{R}^d$, $|\partial^\alpha \sigma(\xi)| \leq C_\alpha |\xi|^{-|\alpha|}\langle \xi \rangle^{-\delta} $.
\begin{enumerate}
\item
We have for all $p \in [1,\infty]$,
\begin{equation}
\label{eq:fouriermult}
\|  \sigma(i\partial) \nabla  f \|_{L^p} \lesssim  \| f \|_{L^p}^{\frac{\delta}{1+\delta}}  \| \nabla f \|_{L^p}^{\frac{1}{1+\delta}}.
\end{equation}

\item Let $s \geq 0$. Assume in addition that $ \delta >s-\lfloor s\rfloor$. 
We have for all $p \in [1,+\infty]$, $r \in [1,\infty)$,
\begin{equation}
\label{eq:fouriermultbesov}
\|  \sigma(i\partial) \nabla  f \|_{B^{s}_{p,r}} \lesssim  \| f \|_{L^p}^{\frac{\delta}{1+\delta}}  \| \nabla f \|_{L^p}^{\frac{1}{1+\delta}} + \| \nabla^{1+\lfloor s \rfloor} f\|_{L^p}.
\end{equation}

\end{enumerate}
\end{lem}

\begin{proof} We start with the first estimate. Given $A\in \mathbb{R}$ to be chosen later, we use the Littlewood-Paley decomposition to write
$$
\sigma(i\partial)  \nabla  f = \sum_{k <A } P_k(  \sigma(i\partial) \nabla  f)  +  \sum_{k \geq A } P_k( \sigma(i\partial)  \nabla  f) =: f_1 + f_2.
$$
On the first hand, there holds by Lemmas~\ref{lem:Bernstein1} and~\ref{lem:Bernstein2},
$$
\| f_1\|_{L^p} \lesssim \sum_{k <A } \| P_k( \sigma(i\partial) \nabla  f)\|_{L^p} \lesssim  \sum_{k <A } 2^k    \| f\|_{L^p}  \lesssim 2^A \| f\|_{L^p}. 
$$
On the other hand, for $\delta > (0,1)$, we have by Lemmas~\ref{lem:Bernstein1} and~\ref{lem:Bernstein2}
$$
\| f_2\|_{L^p} \lesssim \sum_{k \geq A } \| P_k( \sigma(i\partial) \nabla  f)\|_{L^p} 
\lesssim  \sum_{k \geq A } 2^{-k\delta}    \| \nabla f\|_{L^p}  \lesssim 2^{-\delta A}    \|\nabla f\|_{L^p}. 
$$
We can now fix $A$ such that $2^A \| f\|_{L^p}=2^{-\delta A}    \|\nabla f\|_{L^p}$ which yields~\eqref{eq:fouriermult}.

For what concerns the second estimate, we have, using again Lemmas \ref{lem:Bernstein1}, ~\ref{lem:Bernstein2} and the first estimate we have just proved,
$$
\begin{aligned}
\|  \sigma(i\partial) \nabla  f \|_{B^{s}_{p,r}} &\lesssim  \|  \sigma(i\partial) \nabla  f \|_{L^p}  
+ \left( \sum_{k \geq 0} 2^{jsr} \| P_k (\sigma(i\partial) \nabla  f)\|_{L^p}^r\right)^{1/r} \\
&\lesssim   \| f \|_{L^p}^{\frac{\delta}{1+\delta}}  \| \nabla f \|_{L^p}^{\frac{1}{1+\delta}} +  \left( \sum_{k \geq 0} 2^{-j(\delta-s+\lfloor s \rfloor) r}\right)^{1/r} \|  \nabla^{1+\lfloor s \rfloor}  f\|_{L^p},
\end{aligned}
$$
which is precisely~\eqref{eq:fouriermultbesov}.

\end{proof}

\section{Potential estimates}\label{sec-potential}

In this section, we recall some classical estimates for the Poisson equation in $\RR^3$. 

\begin{lemma}\label{lem-potential} Let $\chi(k)$ be sufficiently smooth and compactly  and equal to $1$ in $\{|k|\le 1\}$. Then, the following holds: 

\begin{itemize}

\item[(i)] $\nabla^2_x \Delta^{-1}_x$ is a bounded operator from $L^p$ to $L^p$ for each $1<p<\infty$. In addition, for any $K>0$,
\begin{equation}\label{standard-pot1}
\begin{aligned}
 \|\nabla^2_x \Delta^{-1}_x\rho\|_{L^\infty_x} & \lesssim K^{-3}\|\rho\|_{L^1_x} + \|\rho\|_{L^\infty_x} \Big[\log (2+ \|\nabla_x\rho\|_{L^\infty_x}) + \log (2+K) \Big] .
 \end{aligned}
 \end{equation}

\item[(ii)] $(1-\chi(i\partial_x)) \nabla_x \Delta^{-1}_x$ is a bounded operator from $L^p$ to $L^p$ for all $1\le p\le \infty$.  

\item[(iii)] for any $\delta>0$, $\chi(i\partial_x) |\partial_x|^\delta \nabla^2_x \Delta^{-1}_x$ is a bounded operator from $L^p$ to $L^p$ for all $1\le p \le\infty$.   

\item[(iv)] If $\int \rho(x)\; dx=0$, then 
\begin{equation}\label{aLp-elliptic} \| \chi(i\partial_x) \nabla_x \Delta^{-1}_x \rho\|_{L^p_x} \lesssim  \| \langle x\rangle\rho\|_{L^1_x \cap L^\infty_x}, \qquad \forall 1<p\le \infty.\end{equation}
\end{itemize}

\end{lemma}

\begin{proof} The lemma is classical; see, e.g., \cite[Appendix A]{Toan}.
\end{proof}

\section{Oscillatory integrals}
\label{sec:OI}

We prove in this appendix some variants of classical estimates of oscillating integrals. 
The first one is a robust non-stationary phase lemma for  a non-degenerate phase that can have a non-negative real part. The second is a dispersive estimate for an oscillating integral with a phase that behaves like the Klein-Gordon phase $\sqrt{1+|k|^2}$.

\begin{lemma}
\label{lem:disp1}
 Let $d\in \mathbb{N}\setminus\{0\}$. Let $\Re \alpha \in \mathcal{C}^{ d+1}(\mathbb{R}^d,\mathbb{C}) $, with $ \Re \alpha \leq 0$, satisfy the following estimates. There exit $c_0, C_0>0$ such that,
 \begin{align}
 \label{appendix-stat1}
|\nabla_x( \Im \alpha)|&\geq c_0, \\
   \label{appendix-stat2}
 | D_x^{(j)} \alpha| &\leq C_0, \qquad  j= 1, \ldots,N.
 \end{align}
 Then, for all $\chi \in  \mathcal{C}^{d+1}(\mathbb{R}^d)$ with compact support,
 \begin{align}
\label{appendix-L2-alpha}
\left\| \int_{\mathbb{R}^d} e^{\alpha(k) t + i k \cdot x} \chi(k) \widehat{f}(k) dk  \right\|_{L^2(\mathbb{R}^d)} &\lesssim \| f\|_{L^2(\mathbb{R}^d)},
 \\
\label{appendix-disp-alpha}
\left\| \int_{\mathbb{R}^d} e^{\alpha(k) t + i k \cdot x} \chi(k) \widehat{f}(k) dk  \right\|_{L^\infty(\mathbb{R}^d)} &\lesssim \langle t\rangle^{-\frac N2}\| f\|_{L^1(\mathbb{R}^d)}.
\end{align}
\end{lemma}

\begin{proof}
As we have little information on the real part of $\alpha$ (except the fact that it is non-positive), we rely on a robust version of the  non-stationary phase argument (as inspired by  \cite{FRT}). The proof is close to \cite[Proposition 5.4]{HKNR3}, except for the conclusion.
We write 
$$
 \int_{\mathbb{R}^d} e^{\alpha(k) t + i k \cdot x} \chi(k)  \, dk = \int_{\mathbb{R}^d} e^{it \Psi_X(k)} \chi(k)  \, dk=: I(t,X)
$$
where $X= \frac{x}{t}$, $\Psi_X(k)=  \alpha(k) + X\cdot k$. Note that 
$
D^2_k \Psi_X = D^2_k \alpha.
$
We focus on the times $t\geq 1$.
Consider the operator
$$
\operatorname{L} (u) = \frac{1}{i (1+ t |\nabla \Psi_X|^2)} \sum_{j=1}^d \partial_j \overline{\Psi_X} \partial_j u + \frac{1}{(1+ t |\nabla \Psi_X|^2)} u,
$$
where $|\cdot|$ stands for the hermitian norm on $\mathbb{C}^d$. By construction, we have the identity
\begin{equation}
\label{eq-idL}
\operatorname{L}(e^{it \Psi_X}) = e^{it \Psi_X}.
\end{equation}
The formal adjoint $\operatorname{L}^\star$ of $\operatorname{L}$ reads as
\begin{align*}
\operatorname{L}^\star(u)&=-\sum_{j=1}^{d}\frac{\partial_{j}\overline{\Psi_X}}{i(1+t |\nabla\Psi_{X}|^2)}\partial_{j}u+
\Big(
-\sum_{j=1}^{d}\frac{\partial_{j}^{2}\overline{\Psi_X}}{i(1+ t |\nabla\Psi_{X}|^2)}+ 
\sum_{j=1}^{d}\frac{2t\partial_{j}\overline{\Psi_X}\,{\Re}(\nabla\Psi_{X}\cdot\nabla\partial_{j} \overline{\Psi_{X}})}{i(1+t|\nabla\Psi_{X}|^2)^2}\Big)u
\\
&+\frac{1}{(1+t |\nabla\Psi_{X}|^2)}u.
\end{align*}
Using \eqref{eq-idL} repeatedly, we thus get by integration by parts that
$$ |I(t, X) | \lesssim \int_{\mathbb{R}^d} \left|(\operatorname{L}^\star)^N \left( \chi\right) \right|\, d\xi,$$
 Thanks to~\eqref{appendix-stat1} we can write that 
  $(\operatorname{L}^\star)^N = \sum_{ | \gamma | \leq N} a_{\gamma}^{(N)} \partial^\gamma,$
   where the coefficients $a_{\gamma}^{(N)}$ satisfy on the support of the amplitude  the estimate (uniform with respect to $X$)
   $$ |  a_{\gamma}^{(N)}| \lesssim    \frac{1}{\langle t^{ \frac{1}{2}}  \nabla \Psi_{X} \rangle^{N}},$$
     $$  |I(t, X) | \lesssim  \int_{B(0, R)}   { \frac{1}{ \langle t^{ \frac{1}{2}}  \nabla \Psi_{X} \rangle^N}} \, d\xi,$$
    in which we have assumed that $\chi$ is supported in $B(0,R)$.
     To conclude, we note that
     $$    \int_{B(0, R)}   { \frac{1}{  \langle t^{ \frac{1}{2}}  \nabla \Psi_{X} \rangle^{N}}} \, d\xi
      \leq \int_{B(0, R)}   { \frac{1}{  \langle t^{ \frac{1}{2}}  \nabla \Im \Psi_{X} \rangle^{N}}} \, d\xi
$$
and we can use the bound from below~\eqref{appendix-stat1} to get 
$$    |I(t, X) | \lesssim  { \frac{1}{  \langle t \rangle^{\frac{N}{2}}}},$$
which eventually yields \eqref{appendix-disp-alpha}.

\end{proof}

%
%
%

\begin{lemma}
\label{lem-disp} Let $d\in \mathbb{N}\setminus\{0\}$. Let $\beta \in \mathcal{C}^{ \lceil \frac{d}{2}\rceil}(\mathbb{R}^+,\mathbb{R}^+) $ satisfy the following estimates. There exists  $c_0, C_0>0$ such that,  for all $x \in \mathbb{R}^+$,
\begin{equation}\label{appendix-KG-behave1} 
c_0 \sqrt{1+|x|^2} \le \beta(x)\le C_0 \sqrt{1+|x|^2} ,
\end{equation}
\begin{equation}\label{appendix-KG-behave2}
c_0\frac{|x|}{ \sqrt{1+|x|^2}} \le \beta'(x) \le C_0 \frac{|x|}{ \sqrt{1+|x|^2}} , 
\end{equation} 
\begin{equation}\label{appendix-KG-behave3}
 \beta''(x)  \ge c_0 (1+|x|^2)^{-3/2} ,
\end{equation}
\begin{equation}\label{appendix-KG-behave4}
 |\beta^{(j)}(x)|  \le C_0 \qquad  j= 2, \ldots, \lceil \frac{d}{2}\rceil.
\end{equation}
Then there holds for all $t \in \mathbb{R}$,
\begin{align}
\label{appendix-L2}
\left\| \int_{\mathbb{R}^d} e^{ i \beta(|k|) t + i k \cdot x} \widehat{f}(k) dk  \right\|_{L^2(\mathbb{R}^d)} &\lesssim \| f\|_{L^2(\mathbb{R}^d)}, \\
\label{appendix-disp}
\left\| \int_{\mathbb{R}^d} e^{ i \beta(|k|) t + i k \cdot x} \widehat{f}(k) dk  \right\|_{B^0_{\infty,2}} &\lesssim \langle t\rangle^{-\frac d2}\| f\|_{B^d_{1,2}}.
\end{align}

\end{lemma}

\begin{proof} The $L^2$ estimate is straightforward, using the Plancherel formula. For the other estimate, we  mimic the usual proof of dispersive estimates for the Klein-Gordon operator. Below, we closely follow \cite[Lemma 3.1]{RS} and focus only on the times $t\geq 1$. This is based on the Littlewood-Paley decomposition, as recalled in Appendix~\ref{sec:LittlewoodPaley}.
For the low frequencies, consider $\varphi_0 = \sum_{q\leq -1} \varphi(k/2^q)$, and take $\Phi_0$ a radial compactly supported function equal to $1$ on the support of $\varphi_0$.
We need to prove
$$
\left\| \int_{\mathbb{R}^d} e^{\pm i\beta(|k|) t} e^{i x \cdot k}  \Phi_0(k)  \, dk\right\|_{L^\infty} \lesssim \langle t\rangle^{-\frac d2},
$$
We write 
$$
 e^{\pm ib(|k|) t} e^{i x \cdot k} = e^{i t \Phi_{\pm}(x,k)},
$$
with $\Phi_\pm(x,k) = \beta(|k|) \pm \frac{x}{t}$. 
Let us recall the following well-known result (see e.g.  \cite{RS,Gra}).
\begin{lem}
\label{lem:radialfourier}
For any radial symmetric integrable function $ \Phi$, there holds
$$
\int_{\mathbb{R}^d} e^{\pm i \beta(|k|) t} e^{i x \cdot k}   \Phi(k/\lambda) \, dk = \int_0^{+\infty} e^{i \beta(r)} \Phi(r/\lambda) \mathcal{F}(\sigma_{\mathbb{S}^{d-1}}) (|x| r) r^{d-1} \, dr,
$$
where $\sigma_{\mathbb{S}^{d-1}}$ is the surface measure on $\mathbb{S}^{d-1}$. Furthermore, there holds
$$
 \mathcal{F}(\sigma_{\mathbb{S}^{d-1}})(y) = e^{i |y|} Z(|y|) - e^{-i |y|} \overline{Z}(|y|),
$$
where $Z(s)$ satisfies, for all $k \in \mathbb{N}$ and all $s>0$,
\begin{equation}
\label{appendix-decayZ}
|\partial^k_s Z(s) |\lesssim_{k} \frac{1}{(1+s)^{\frac{d-1}{2}+k}}.
\end{equation}
\end{lem}

Therefore, we have
$$
\int_{\mathbb{R}^d} e^{\pm i \beta(|k|) t} e^{i x \cdot k}   \Phi_0(k/\lambda) \, dk = \int_0^{+\infty} e^{i \beta(r)} \Phi_0(r/\lambda) \mathcal{F}(\sigma_{\mathbb{S}^{d-1}}) (|x| r) r^{d-1} \, dr,
$$
and  thanks to~\eqref{appendix-KG-behave3},
 the estimate follows by a stationary phase argument (see e.g. \cite{Stein}).

There remains to treat the high frequency part. Let us consider $\phi$, a radial  function, equal to $1$ on the support of $\varphi$, and that is compactly supported on an annulus slightly larger than the support of $\varphi$.
For any $\lambda>0$, we aim at proving
\begin{align}
\label{appendix-dyadic}
\left\| \int_{\mathbb{R}^d} e^{\pm i \beta(|k|) t} e^{i x \cdot k}\phi(k/\lambda) \, dk\right\|_{L^\infty}\lesssim \lambda^{d}, \quad \left\| \int_{\mathbb{R}^d} e^{\pm i \beta(|k|) t} e^{i x \cdot k}   \phi(k/\lambda) \, dk \right\|_{L^\infty} \lesssim | t|^{-\frac{d}{2}}\lambda^{1+ \frac{d}{2}}.
\end{align}
Since the first estimate follows by a straightforward dilation, we only focus on the second one.
As $\phi$ and $\beta(|\cdot|)$ are radially symmetric, we have by Lemma~\ref{lem:radialfourier} the formula
$$
\int_{\mathbb{R}^d} e^{\pm i \beta(|k|) t} e^{i x \cdot k}   \phi(k/\lambda) \, dk = \sum_{\pm} \lambda^d \int_{1/4}^4 e^{i t \psi_\lambda^\pm (t,|x|,r)} \phi(r) Z_\pm (\lambda |x| r) r^{d-1} \, dr,
$$
denoting $Z_\pm= Z, \overline{Z}$ and $\psi_\lambda^\pm(t,|x|,r) = \beta(\lambda r) \pm \lambda r \frac{|x|}{t}$.

By~\eqref{appendix-decayZ}, we remark once for all that for all $k \in \mathbb{N}$, for all $r \geq 1/4$,
\begin{equation}
\label{appendix-decayderivZ}
|\partial^k_r (Z_\pm(\lambda |x| r))| \lesssim \frac{1}{(1+\lambda|x| r)^{\frac{d-1}{2}}} \lesssim  1.
\end{equation}
In the + case, we note thanks to \eqref{appendix-KG-behave2} that
$$
\partial_r \psi_\lambda^+(t,|x|,r) = \lambda \beta'(\lambda r) + \lambda \frac{|x|}{t}  \gtrsim \lambda.
$$
Thanks to \eqref{appendix-KG-behave4} we deduce by induction that for $k=0, \ldots, \lceil \frac{d}{2}\rceil$, we have
$$
\left|\partial^k_r \frac{1}{\partial_r \psi_\lambda^+(t,|x|,r)} \right| \lesssim \frac{1}{\lambda}.
$$
As a result, by a non-stationary phase argument, using~\eqref{appendix-decayderivZ}, we get
\begin{equation}
\left|\int_{1/4}^4 e^{i t \psi_\lambda^+ (t,|x|,r)} \phi(r) Z_+ (\lambda |x| r) r^{d-1} \, dr\right| \lesssim \frac{1}{(\lambda t)^N},
\end{equation}
for all $N=0,\ldots,  \lceil \frac{d}{2}\rceil $.

In the - case, we have to distinguish between two subcases. Writing
$$
\partial_r \psi_\lambda^-(t,|x|,r) = \lambda \beta'(\lambda r) - \lambda \frac{|x|}{t},
$$
we first note that in the space-time region $\left\{ |x| \leq \frac{t}{100} \text{ or } t \leq \frac{|x|}{100}\right\}$, one has
$$
|\partial_r \psi_\lambda^-(t,|x|,r) | \gtrsim \frac{\lambda}{t} (t+|x|).
$$
By induction, thanks to~\eqref{appendix-KG-behave4}, we thus obtain for all $k=0, \ldots, \lceil \frac{d}{2}\rceil$
$$
\left|\partial^k_r \frac{1}{\partial_r \psi_\lambda^+(t,|x|,r)} \right| \lesssim \frac{t}{\lambda(\lambda(t+|x|))}.
$$
We can therefore argue as in the + case and obtain in $\left\{ |x| \leq \frac{t}{100} \text{ or } t \leq \frac{|x|}{100}\right\}$,
\begin{equation}
\left|\int_{1/4}^4 e^{i t \psi_\lambda^- (t,|x|,r)} \phi(r) Z_- (\lambda |x| r)  r^{d-1} \, dr\right| \lesssim \frac{1}{(\lambda (t+ |x|))^N},
\end{equation}
for all $N=0,\ldots,  \lceil \frac{d}{2}\rceil $. Otherwise, on the space-time region $\left\{  \frac{1}{100}  \leq \frac{|x|}{t} \leq 100\right\}$, we observe that thanks to~\eqref{appendix-KG-behave3},
$$
\partial^2_r \psi_\lambda^-(t,|x|,r) = \lambda^2 \beta''(\lambda r)   \gtrsim \frac{1}{\lambda}.
$$
As a result, using a stationary phase argument and the bounds~\eqref{appendix-decayderivZ}, we obtain that for  $(t,x)\in \left\{  \frac{1}{100}  \leq \frac{|x|}{t} \leq 100\right\}$,
\begin{equation}
\left|\int_{1/4}^4 e^{i t \psi_\lambda^- (t,|x|,r)} \phi(r) Z_- (\lambda |x| r)  r^{d-1} \, dr\right| \lesssim \frac{\lambda}{(\lambda t)^{d/2}}.
\end{equation}
The estimate~\eqref{appendix-dyadic} follows, and by summing over all dyadic frequencies, and finally end up with~\eqref{appendix-disp}.

\end{proof}

\bigskip

 \noindent {\bf Acknowledgements.} The authors thank Igor Rodnianski for many fruitful discussions.  
 TN acknowledges the hospitality of the ENS de Paris for the visits during which part of this work was carried out. 
 DHK acknowledges partial support from ANR-11-LABX-0020-01 and  ANR-19-CE40-0004.
TN's research is supported in part by the NSF under grant DMS-2054726.


\begin{theindex}

  \item $A(t,x)$: magnetic vector potential, \hyperpage{7}
  \item $A^r(t,x) $: regular part of the magnetic vector potential, 
		\hyperpage{31}
  \item $A^{osc}_\pm(t,x) $: oscillatory part of the magnetic vector potential, 
		\hyperpage{31}

  \indexspace

  \item $D(\lambda,k)$: electric dispersion function, \hyperpage{8}
  \item $\widetilde{D}(\lambda,k)$: extension of $D$, \hyperpage{15}
  \item $\widetilde{\widetilde{D}}(\lambda,k)$: extension of $D$, 
		\hyperpage{15}

  \indexspace

  \item $\eps_0$:  bound on the initial conditions, \hyperpage{3}

  \indexspace

  \item $\FG^{osc}_{k,\pm}(t)$: oscillatory part of the electric Green function, 
		\hyperpage{28}
  \item $\FG^{r}_{k}(t)$: regular part of the electric Green function, 
		\hyperpage{28}
  \item $\FG_k(t)$: electric Green function, \hyperpage{28}
  \item $\FH^{osc}_{k,\pm}(t)$: oscillatory part of the magnetic Green function, 
		\hyperpage{29}
  \item $\FH^{r}_{k}(t)$: regular part of the magnetic Green function, 
		\hyperpage{29}
  \item $\FH_k(t)$: magnetic Green function, \hyperpage{29}
  \item $g(t,x,v)$: shifted distribution function, \hyperpage{8}

  \indexspace

  \item $K_k(t)$, \hyperpage{13}
  \item $\kappa(u)$, \hyperpage{10}
  \item $\kappa_0$, \hyperpage{19}

  \indexspace

  \item $\Lambda(z)$, \hyperpage{18}
  \item $\lambda_\pm^{\text{elec}}$: electric dispersion relation, 
		\hyperpage{19}
  \item $\lambda_\pm^{\text{mag}}$: magnetic dispersion relation, 
		\hyperpage{24}

  \indexspace

  \item $M(\lambda,k)$: magnetic dispersion function, \hyperpage{8}

  \indexspace

  \item $N_k(t)$, \hyperpage{17}
  \item $\nu_{*}$, \hyperpage{24}

  \indexspace

  \item $\Omega(y)$, \hyperpage{15}
  \item $\omega$, \hyperpage{21}

  \indexspace

  \item $\Phi(z)$, \hyperpage{15}
  \item $\widetilde\Phi(z)$, \hyperpage{15}
  \item $\widetilde{\widetilde{\Phi}}(z)$, \hyperpage{15}
  \item $\mathbb{P}$: Leray projector, \hyperpage{8}
  \item $\phi(t,x)$: electric scalar potential, \hyperpage{7}
  \item $\phi^r(t,x) $: regular part of the electric scalar potential, 
		\hyperpage{30}
  \item $\phi^{osc}_\pm(t,x) $: oscillatory part of the electric scalar potential, 
		\hyperpage{30}
  \item $\psi(y)$, \hyperpage{18}

  \indexspace

  \item $q(u)$, \hyperpage{10}

  \indexspace

  \item $S(t,x), S^\bj(t,x)$, \hyperpage{29}

  \indexspace

  \item $\tau_*$, \hyperpage{19}
  \item $\tau_0$, \hyperpage{8}

  \indexspace

  \item $x_*$, \hyperpage{21}

\end{theindex}

%

\bibliographystyle{abbrv}

\end{document}